\documentclass[11pt]{article}

\usepackage{amsthm}
\usepackage{amsmath}
\usepackage{amssymb}
\usepackage{bbm}

\newtheorem{dfn}{Definition}[section]
\newtheorem{thm}[dfn]{Theorem}
\newtheorem{prop}[dfn]{Proposition}
\newtheorem{cor}[dfn]{Corollary}
\newtheorem{lem}[dfn]{Lemma}

\newtheorem{rem}[dfn]{Remark}

\newcommand{\D}{\displaystyle}

\newcommand{\del}{\partial}
\newcommand{\x}{\times}
\newcommand{\Ad}{\mathrm{Ad}}
\newcommand{\mto}{\longmapsto}
\newcommand{\ra}{\longrightarrow}

\newcommand{\Ra}{\Longrightarrow}

\newcommand{\vol}{\mathrm{vol}}

\newcommand{\pt}{\mathrm{pt}}
\newcommand{\dvol}{\mathrm{dvol}}
\newcommand{\Iso}{\mathrm{Iso}}
\newcommand{\Ind}{\mathrm{Ind}}
\newcommand{\ds}{\mathrm{ds}}
\newcommand{\dt}{\mathrm{dt}}
\newcommand{\dz}{\mathrm{dz}}
\newcommand{\im}{\mathrm{im}}
\newcommand{\coker}{\mathrm{coker}}
\newcommand{\Res}{\mathrm{Res}}
\newcommand{\codim}{\mathrm{codim}}
\newcommand{\eps}{\varepsilon}

\newcommand{\mcA}{\mathcal{A}}
\newcommand{\mcB}{\mathcal{B}}
\newcommand{\mcC}{\mathcal{C}}
\newcommand{\mcD}{\mathcal{D}}
\newcommand{\mcE}{\mathcal{E}}
\newcommand{\mcF}{\mathcal{F}}
\newcommand{\mcG}{\mathcal{G}}
\newcommand{\mcI}{\mathcal{I}}
\newcommand{\mcJ}{\mathcal{J}}
\newcommand{\mcL}{\mathcal{L}}
\newcommand{\mcM}{\mathcal{M}}
\newcommand{\mcS}{\mathcal{S}}
\newcommand{\mcV}{\mathcal{V}}

\newcommand{\CP}{\mathbbm C\mathrm P}

\newcommand{\EG}{\mathrm{EG}}
\newcommand{\BG}{\mathrm{BG}}
\newcommand{\EH}{\mathrm{EH}}
\newcommand{\BH}{\mathrm{BH}}
\newcommand{\EK}{\mathrm{EK}}

\newcommand{\ET}{\mathrm{ET}}
\newcommand{\BT}{\mathrm{BT}}

\newcommand{\LG}{{\mathfrak g}}
\newcommand{\LH}{{\mathfrak h}}
\newcommand{\LK}{{\mathfrak k}}
\newcommand{\LT}{{\mathfrak t}}
\newcommand{\LU}{{\mathfrak u}}

\newcommand{\C}{\mathbbm C}
\newcommand{\N}{\mathbbm N}
\newcommand{\Q}{\mathbbm Q}
\newcommand{\R}{\mathbbm R}
\newcommand{\T}{\mathbbm T}
\newcommand{\Z}{\mathbbm Z}

\newcommand{\unit}{\mathbbm 1}

\begin{document}

\title{Vortex Invariants and Toric Manifolds}
\author{Jan Wehrheim}

\maketitle

\begin{abstract}
We consider the symplectic vortex equations for a Hamiltonian action of a torus $T$ on $\C^n$. We show that the associated genus zero moduli space itself is homotopic (in the sense of a regular homotopy of $T$-moduli problems) to a toric manifold with combinatorial data directly obtained from the original torus action. This allows to view the wall crossing formula for the computation of vortex invariants by Cieliebak and Salamon \cite{CS} as a consequence of a generalized Jeffrey-Kirwan localization formula for integrals over symplectic quotients.
\end{abstract}

\tableofcontents

\section{Introduction}

Let $(X,\omega)$ be a symplectic manifold. Suppose that a compact Lie group $G$ acts on $X$ in a Hamiltonian way with moment map $\mu$. This is the general setting for the symplectic vortex equations that where introduced by Cieliebak, Gaio, Mundet and Salamon in \cite{CGMS}. They are of the form
$$
(*) \quad \left\{ \quad
\begin{array}{c@{\quad = \quad}c}
\bar{\partial}_{J,A} u & 0\\
\ast F_A + \mu(u) & \tau.
\end{array}
\right.
$$
Here $u : P \rightarrow X$ is a $G$-equivariant map from a principal $G$-bundle $P$ over a closed Riemann surface $\Sigma$ and $A \in \Omega^1(P)$ is a connection form on $P$ with curvature $F_A$. The additional data entering the equations are a $G$-invariant, $\omega$-compatible almost complex structure $J$ on $X$ that gives rise to the Cauchy-Riemann operator $\bar{\partial}_{J,A}$, a metric on $\Sigma$ that defines the Hodge-operator $\ast$, and a parameter $\tau$. For the motivation to study these equations we refer to Cieliebak, Gaio and Salamon \cite{CGS}. There are two central results: By Cieliebak, Gaio, Mundet and Salamon \cite{CGMS} the solutions to $(*)$ give rise to well defined invariants in many cases. And it is shown by Gaio and Salamon in \cite{GS} that in certain cases these vortex invariants coincide with Gromov-Witten invariants of the symplectic quotient $X/\!/G(\tau) := \mu^{-1}(\tau)/G$. The latter result is obtained by introducing a parameter $\eps$ in the second equation of $(*)$ in front of the term $\mu(u)$ and an adiabatic limit analysis for $\eps \ra \infty$. In this limit the solutions to the vortex equations degenerate to holomorphic curves $\Sigma \ra X/\!/G(\tau)$. We study the symplectic vortex equations on $\C^N$ with its standard symplectic and complex structure and with a torus $T$ acting by a representation $\rho : T \ra U(N)$. Symplectic quotients of such linear torus actions are called toric manifolds. In this setup vortex invariants are well defined.

Our main result can be viewed as the counterpart to the adiabatic limit of Gaio and Salamon in \cite{GS}. We introduce the same parameter $\eps$ (with a slight modification if the principal bundle $P$ is not trivial) but we consider the other limit $\eps \ra 0$. One can also interprete this deformation as a rescaling of the symplectic form by $\eps$. We show that this deformation gives rise to a homotopy of regular $T$-moduli problems (Theorem \ref{thm:homotopy}). The main issue is to prove compactness for the parametrized moduli space. Our result then shows that the invariants associated to the deformed vortex equations with $\eps = 0$ agree with the usual vortex invariants. And in fact this deformed picture is very nice: We show that the moduli space to the deformed genus zero vortex equations itself carries the structure of a toric manifold (Theorem \ref{thm:genus_0_moduli_space}). If $\Sigma$ is a surface of arbitrary genus we show more generally that under some additional assumptions the vortex moduli space is a fiber bundle over the Jacobian torus $\left( H^1(\Sigma;\R) / H^1(\Sigma;\Z) \right)^{\dim T}$ with toric fiber (Theorem \ref{thm:genus_g_moduli_space}).

As an application we show how these observations simplify the computation of genus zero vortex invariants. We can express vortex invariants as integrals over toric manifolds (Theorem \ref{thm:computation_of_vortex_invariants}). The wall crossing formula of Cieliebak and Salamon \cite[Theorem 1.1]{CS} that was used for the original computation of these invariants then is a consequence of a localization formula for certain integrals over toric manifolds (Theorem \ref{thm:wall_crossing}). To prove this localization formula we develop a relative version of Atiyah-Bott localization (Theorem \ref{thm:relative_localization}) that applies to the case of invariant integration. This technique is interesting on its own account and we develop it in more generality than needed for the computation of vortex invariants. It gives an alternative approach to integration formulae that are generally referred to as Jeffrey-Kirwan localization.

In sction \ref{chap:equivariant_cohomology} we review equivariant cohomology of manifolds that carry a group action by some compact Lie group $G$. In particular we review the Cartan model and the Cartan map and its generalization in the case of a normal subgroup $H \lhd G$ acting locally freely. Invariant integration is introduced in section \ref{chap:invariant_integration} and generalized in two directions. One is the relative case of invariant integration with respect to a normal subgroup $H \lhd G$. The other is the extension to more general push-forwards in equivariant cohomology. The result is the notion of $H$-invariant push-forward. Section \ref{chap:localization} then features the theory of localization adapted to the setting of invariant integration. We deduce the above-mentioned integration formula (Theorem \ref{thm:relative_localization}).

In section \ref{chap:moduli_problems} we turn to the notion of $G$-moduli problems and the equi\-vari\-ant Euler class from Cieliebak, Mundet and Salamon \cite{CMS}. This is the technique that is used to define invariants via the solutions to the vortex equations. It uses invariant integration. So we examine what our ge\-ne\-ra\-li\-za\-tions to invariant integration from section \ref{chap:invariant_integration} yield in this context.

We turn to toric manifolds in section \ref{chap:toric_manifolds}. We apply the results from all preceding sections and prove the wall crossing formula for certain integrals over toric manifolds (Theorem \ref{thm:wall_crossing}). We make a short digression that puts this result into the context of Jeffrey-Kirwan localization and the related works by Guillemin and Kalkman \cite{GK} and Martin \cite{Mar}.

The symplectic vortex equations and invariants in the case of a linear torus action are introduced in section \ref{chap:vortex_invariants}. We perform the above-mentioned deformation, prove the essential property of being a homotopy of regular $T$-moduli problems and derive the main results.

In section \ref{chap:generalizations} we attempt to extend our deformation result to more general settings. In fact we have no such generalization and we explain the reason for that failure. Finally in section \ref{chap:giventals_toric_map_spaces} we explain the relation of our toric structure on the vortex moduli spaces to Givental's toric map spaces as presented in the work of Iritani \cite{Iri}.

\section{Equivariant cohomology}

\label{chap:equivariant_cohomology}

Throughout this section let $G$ be a compact Lie group and $X$ a closed manifold with a smooth $G$-action from the left. We introduce the Borel and the Cartan model for equivariant cohomology $H_G^*(X)$. In particular we consider a closed normal subgroup $H \lhd G$ and look at the consequences it has on $H_G^*(X)$ if the induced $H$-action is free or trivial.

\subsection{The Borel model}

\label{sec:Borel_model}

We denote by $\BG$ the classifying space for the group $G$. It is defined uniquely up to homotopy equivalence. Given any contractible topological space $\EG$ with a free $G$-action we can take $\BG := \EG / G$ as a representative. We will call any such $\EG$ a classifying total space for $G$. See tom Dieck \cite{tomD} for details.

The Borel construction of the $G$-manifold $X$ is the topological space
$$
X_G := X \times_G \EG := \left( X \times \EG \right) / G,
$$  
where the quotient is taken with respect to the diagonal $G$-action on $X \times \EG$. Then one defines the $G$-equivariant cohomology of $X$ to be the ordinary singular cohomology of its Borel construction,
$$
H_G^*(X) := H^*(X_G).
$$
Throughout we will take either real or complex coefficients.

\begin{rem}
The projection onto the second factor turns $X_G$ into a fiber bundle over $\BG$ with fiber $X$. In particular this makes equivariant co\-ho\-mo\-lo\-gy $H_G^*(X)$ a module over $H^*(\BG)$ via pullback.
\end{rem} 

Now suppose that $H \lhd G$ is a closed normal subgroup of $G$ that acts freely on $X$. We denote by $K := G/H$ the quotient group. On the classifying total space $\EK$ we let $G$ act via the projection $G \ra K$. Hence $H$ acts trivially on $\EK$. Now we can replace $\EG$ by $\EG \times \EK$ as a model for the classifying total $G$-space, because the product of two contractible spaces is still contractible, and the diagonal $G$-action is free, because it is free on the first factor. So up to homotopy equivalence we obtain
$$
X_G = \left( X \times \EG \times \EK \right) / G = \left( \left( X \times_H \EG \right) \times \EK \right) / K.
$$
Now $H$ acts freely on $X$ and $\EG$ is contractible. Hence $X \times_H \EG$ is homotopy equivalent to $X/H$ and this homotopy equivalence extends to the bundle $\left( \left( X \times_H \EG \right) \times \EK \right) / K$, as shown in tom Dieck \cite[I 8.16]{tomD}. So we obtain
$$
X_G = \left( X/H \times \EK \right) / K = \left( X/H \right)_K
$$
up to homotopy equivalence and we deduce

\begin{equation} \label{eqn:free_H_action}
H_G^*(X) = H_K^*(X/H).
\end{equation}

Now suppose on the contrary that a closed normal subgroup $H \lhd G$ acts trivially on $X$. Again we write $X_G = (X \times \EK \times \EG) / G$, but now $H$ acts trivially on $X$ and $\EK$ and we obtain
$$
X_G = \left( X \times \EK \times \EG/H \right) / K.
$$
This shows that up to homotopy equivalence $X_G$ is a fibration over $X_K$ with fiber $\EG/H = \BH$, because $\EG$ also serves as a model for $\EH$: It is contractible and the subgroup $H$ acts freely. This fibration is determined by the $K$-action on $\EG/H$ and it determines the relation between $H_G^*(X)$, $H_K^*(X)$ and $H^*(\BH)$.

If we additionally assume that there exists a monomorphism $G \ra K \times H$ then we get a free $G$-action on the contractible space $\EK \times \EH$ and we can write $X_G$ as $(X \times \EK \times \EH) / G$. Now $H$ again acts trivially on $X$ and $\EK$, whereas $K$ acts trivially on $\EH$ and hence also on $\EH/H = \BH$. So we get $X_G = X_K \times \BH$, i.e.~the fibration is trivial and the relation becomes particularly nice: 
\begin{equation} \label{eqn:trivial_H_action}
H_G^*(X) = H_K^*(X) \otimes H^*(\BH).
\end{equation}

\subsection{The Cartan model}

\label{sec:Cartan_model}

The Cartan model is the de Rham model for equivariant cohomology. The topological constructions on spaces in the Borel model are mimicked on the algebraic level of differential forms.

\subsubsection*{Equivariant de Rham theory}

The Lie algebra of $G$ is denoted by $\LG$, its dual space by $\LG^*$. We write $S(\LG^*)$ for the symmetric algebra generated by the vector space $\LG^*$. We assign the degree of $2$ to all elements of $\LG^*$ and view $S(\LG^*)$ as a graded algebra.

The group $G$ acts from the left on $\LG^*$ by the coadjoint representation $\Ad^*$ defined by
$$
\langle \Ad^*_g(x) , \xi \rangle \;:=\; \langle x , \Ad_{g^{-1}}(\xi) \rangle
$$
for any $x \in \LG^*$ and $\xi \in \LG$, where $\langle .,. \rangle$ denotes the pairing between dual spaces and $\Ad_g(x) = gxg^{-1}$ is the adjoint action of $G$ on its Lie algebra. This action extends element-wise to $S(\LG^*)$.

We denote by $\Omega^*(X)$ the differential forms on $X$. We get a left action of $G$ on $\Omega^*(X)$ by taking inverses and pulling back: $g.\omega := (g^{-1})^*\omega$. For an element $\xi \in \LG$ we define the infinitesimal vector field $X_\xi$ at a point  $p \in X$ as
\begin{equation} \label{def:X_xi}
X_\xi (p) \;:=\; \left. \frac{\mathrm{d}}{\dt} \right|_{t=0} \left[ \exp(-t\xi) . p \right].
\end{equation}
We abbreviate the operation of plugging the vector field $X_\xi$ into the first argument of a differential form by $\iota_\xi$.

\begin{rem}
We introduce the minus sign in the definition of $X_\xi$ in order to be consistent with the conventions for principal bundles later on. This convention also agrees with the one by Guillemin and Sternberg \cite{GS} but differs from the one by Cieliebak, Mundet and Salamon \cite{CMS}.
\end{rem}

We can now define $G$-equivariant differential forms on $X$ as elements in
$$
\Omega_G^*(X) \;:=\; \left[ S(\LG^*) \otimes \Omega^*(X) \right]^G,
$$
where the superscript $G$ means that we take only $G$-invariant elements of the tensor product. One can think of taking the tensor product with $S(\LG^*)$ as an algebraic analogue of taking the product with some representative of $\EG$ in the topological setting. And restricting to $G$-invariant elements corresponds to taking the $G$-quotient of $\EG \times X$. In analogy to the Borel construction $X_G$ for the space $X$ this algebraic version is also called the Cartan construction for $\Omega^*(X)$ and we write
$$
C_G(B) \;:=\; \left[ S(\LG^*) \otimes B \right]^G
$$
for any suitable algebra $B$. Hence $\Omega_G^*(X) = C_G( \Omega^*(X) )$.

To define the $G$-equivariant differential $d_G$ we think of elements $\alpha \in \Omega_G^*(X)$ as polynomials on $\LG$ with values in $\Omega^*(X)$. Then the value of $d_G \alpha$ on an element $\xi \in \LG$ is defined to be
$$
\mathrm{d}_G \alpha (\xi) \;:=\; \mathrm{d} \alpha(\xi) - \iota_\xi \alpha(\xi).
$$
If we choose a basis $(\xi_a)$ for $\LG$ and take the dual basis $(x^a)$ on $\LG^*$ this can be expressed as
$$
\mathrm{d}_G \;=\; 1 \otimes \mathrm{d} - \sum_a x^a \otimes \iota_{\xi_a},
$$
showing that this differential indeed rises the degree by $+1$, because $x^a$ has degree $2$. It is an easy exercise to show that $\mathrm{d}_G^2 = 0$, but the fact that the homology of the Cartan complex indeed computes equivariant cohomology is a nontrivial result.

\begin{thm}[H.~Cartan] \label{thm:Cartan}

$$
H_G^*(X) \;=\; H \left( C_G(\Omega^*(X)) , \mathrm{d}_G \right)
$$

\end{thm}

\subsubsection*{The generalized Cartan map}

We discussed the case of a normal subgroup $H \lhd G$ acting freely on $X$ in the Borel model and obtained identity \ref{eqn:free_H_action}. With the Cartan model we can obtain the corresponding result also for locally free $H$-actions. A group $H$ acts locally freely if the isotropy subgroups
$$
\Iso_p(H) := \left\{ h \in H \;|\; h.p = p \right\}
$$
are finite for all $p \in X$, or equivalently if the infinitesimal action $\xi \mto X_\xi(p)$ is everywhere injective. We will also use the term \emph{regular} for a locally free action and use the notation $\Iso_p(H) =: H_p$ for isotropy groups.

\begin{dfn}
A form $\omega \in \Omega^*(X)$ is called \emph{$H$-horizontal} if 
$$
\iota_\eta \omega \;=\; 0
$$
for all $\eta \in \LH$. The set of $H$-horizontal forms is denoted by $\Omega_{H-\mathrm{hor}}^*(X)$. A form $\omega$ is called \emph{$H$-basic} if it is $H$-horizontal and $H$-invariant. The set of $H$-basic forms is denoted by $\Omega_{H-\mathrm{bas}}^*(X)$.
\end{dfn}

Note that in fact $\Omega_{H-\mathrm{bas}}^*(X) = \Omega^*(X/H)$ if the $H$-action is free. But even in presence of nontrivial isotropy the space $\Omega_{H-\mathrm{bas}}^*(X)$ still serves perfectly well as an algebraic model for differential forms on the quotient: If $K := G/H$ we get a $K$-action on $\Omega_{H-\mathrm{bas}}^*(X)$ and for elements $\theta \in \LK = \mathrm{Lie}(K)$ we get operations $\iota_\theta$ by lifting $\theta$ to $\overline{\theta} \in \LG$ and taking $\iota_{\overline{\theta}}$. Hence we can perform the $K$-Cartan construction on $\Omega_{H-\mathrm{bas}}^*(X)$.

\begin{prop} \label{prop:gen_Cartan}
If a closed normal subgroup $H \lhd G$ acts locally freely on $X$ then the inclusion
$$
\Omega_G^*(X) = \left[ S(\LG^*) \otimes \Omega^*(X) \right]^G \supset \left[ S(\LK^*) \otimes \Omega_{H-\mathrm{bas}}^*(X) \right]^K = C_K \left( \Omega_{H-\mathrm{bas}}^*(X) \right)
$$
induces an isomorphism
$$
H_G^*(X) \;\cong\; H \left( C_K \left( \Omega_{H-\mathrm{bas}}^*(X) \right) , \mathrm{d}_K \right).
$$
\end{prop}

\begin{rem}
This result is certainly not new. It is well known (although we do not know a suitable reference) that one can deal with orbifolds just as well as with manifolds, so that there is no trouble with taking the quotient with respect to an action that is only locally free. Nevertheless we think it is worthwhile to treat the whole theory without leaving the realm of smooth manifolds. We do not go to the partial $H$-quotient but keep track of the whole action on the original manifold and this turns out to be helpful for the discussion of integration and localization later on.
\end{rem}

The essential tool for the proof of proposition \ref{prop:gen_Cartan} (i.e.~the homotopy inverse to the inclusion) is the Cartan map. In case $H = G$ this map is due to H.~Cartan (see Guillemin and Sternberg \cite{GS} for a nice exposition). The generalized map for arbitrary $H$ appears in the work of Cieliebak, Mundet and Salamon \cite{CMS} and is a slight extension of the things said in \cite[section 4.6]{GS} for commuting actions of two groups. For clearity we give an extensive review of the construction.

\begin{dfn} \label{def:G_equivariant_H_connection}
A \emph{$G$-equivariant $H$-connection} on $X$ is a form $A \in \Omega^1(X,\LH)$ that satisfies
$$
A_{g.p}(g_*v ) \;=\; g A_p(v) g^{-1} \quad , \qquad A_p(X_\eta(p)) \;=\; \eta
$$
for all $g \in G$, $p \in X$, $v \in T_pX$, $\eta \in \LH$.
\end{dfn}

\begin{rem}
Suppose $X$ is a principal $G$-bundle with right action 
$$
X \x G \ra X \; ; \; (p,g) \mto pg.
$$
Then $X$ also carries the \emph{left} $G$-action $G \x X \ra X \; ; \; (g,p) \mto pg^{-1}$. A $G$-equivariant $G$-connection on $X$ with this left $G$-action is then the same as a connection form on the principal bundle $X$ in the usual sense.
\end{rem}

\begin{rem}
Such $G$-equivariant $H$-connections do exist if the $H$-action is regular. Here we make use of the compactness of $G$ and of $H$ being normal.
\end{rem}

Let us fix a $G$-equivariant $H$-connection $A$ for the moment. This connection determines a $G$-equivariant projection
\begin{equation} \label{def:pi_A}
\pi_A^* \;:\; \Omega^*(X) \;\ra\; \Omega_{H-\mathrm{hor}}^*(X)
\end{equation}
via the projection $\pi_A$ onto the kernel of $A$:
\begin{align*}
\pi_A \;:\; T_pX & \;\ra\; T_pX \\
v & \;\mto\; v - X_{A_p(v)}(p)
\end{align*}
The curvature $F_A \in \Omega^2(X,\LH)$ is defined by
$$
F_A \;:=\; \mathrm{d}A + \frac{1}{2} \left[ A \wedge A \right],
$$
where $[. \wedge .]$ denotes the product on $\Omega^*(X,\LH) = \Omega^*(X) \otimes \LH$ via the wedge-product on the form part and the Lie bracket on $\LH$. Note that $F_A$ inherits the $G$-invariance from $A$ and that $F_A$ is $H$-horizontal. The $G$-equivariant curvature $F_{A,G}$ is the element of
$$
\Omega_G^*(X,\LG) \;:=\; C_G( \Omega^*(X,\LG) ) \;=\; \left[ S(\LG^*) \otimes \Omega^*(X,\LG) \right]^G
$$
given by
$$
F_{A,G}(\xi) \;:=\; F_A + \xi - \iota_\xi A.
$$
The lemma below shows that $F_{A,G}$ really is an invariant element of the tensor product $S(\LG^*) \otimes \Omega^*(X,\LG)$. We choose a basis $(\xi_j)$ on $\LG$ such that $(\xi_b)_{b \in B}$ is a basis for the subspace $\LH$ for some index set $B$, and we take the dual basis $(x^j)$ for $\LG^*$. Then the $G$-equivariant curvature can be expressed as
$$
F_{A,G} \;=\; 1 \otimes F_A \quad + \quad \sum_j x^j \otimes \left( \xi_j - \iota_{\xi_j} A \right).
$$
In the sum actually only indices $j \not\in B$ appear, because $\eta - \iota_\eta A = 0$ for all $\eta \in \LH$.

\begin{lem}
The equivariant curvature $F_{A,G}$ in fact defines an element in the subspace 
$$
\left[ S(\LK^*) \otimes \Omega_{H-\mathrm{bas}}^*(X,\LG) \right]^K \;\subset\; \left[ S(\LG^*) \otimes \Omega^*(X,\LG) \right]^G.
$$
\end{lem}

\begin{proof}
Since $F_A$ is $G$-invariant and $H$-horizontal the first summand $1 \otimes F_A$ obviously lies in the correct space. For the remaining part we first note that for $j \not\in B$ we have $x^j \in \LK^* = \{ x \in \LG \;|\; x|_\LH = 0 \}$. Furthermore the $\xi_j - \iota_{\xi_j} A$ are all $H$-horizontal. Now $A$ is $G$-invariant and $\iota_{\xi_j} A$ is linear in $\xi_j$. Hence $\xi_j - \iota_{\xi_j} A$ transforms under $G$ just like $\xi_j$. Thus the sum over all $x^j \otimes (\xi_j - \iota_{\xi_j} A)$ is $G$-invariant, because $\sum_j x^j \otimes \xi_j$ is. This shows that
$$
F_{A,G} \;\in\; \left[ S(\LK^*) \otimes \Omega_{H-\mathrm{hor}}^*(X,\LG) \right]^G.
$$
But since $H$ acts trivially on $\LK^*$ it must also act trivially on all terms $\xi_j - \iota_{\xi_j} A$ with $j \not\in B$ and the result follows.
\end{proof}

Now we consider the pairing between $\LG^*$ and the $\LG$-factor of $F_{A,G}$:
\begin{align*}
\LG^* & \;\ra\; S(\LK^*) \otimes \Omega_{H-\mathrm{hor}}^*(X) \\
x \; & \;\mto\; \qquad \langle x , F_{A,G} \rangle
\end{align*}
This map is $G$-equivariant because $F_{A,G}$ is $G$-invariant. We extend it to a $G$-equivariant map
\begin{equation} \label{def:F_AG}
S(\LG^*) \;\ra\; S(\LK^*) \otimes \Omega_{H-\mathrm{hor}}^*(X).
\end{equation}
Now the Cartan map $c_A$ for a given $G$-equivariant $H$-connection $A$ is defined by the tensor product of the two maps \ref{def:pi_A} and \ref{def:F_AG}
$$
\Omega_G^*(X) \;=\; \left[ S(\LG^*) \otimes \Omega^*(X) \right]^G \;\ra\; \left[ S(\LK^*) \otimes \Omega_{H-\mathrm{hor}}^*(X) \otimes \Omega_{H-\mathrm{hor}}^*(X) \right]^G
$$
followed by taking the wedge product of the two $\Omega_{H-\mathrm{hor}}^*(X)$ factors. Since all these operations are $G$-equivariant the map stays in the $G$-invariant part and we in fact obtain a map
\begin{align} \label{def:c_A}
c_A \;:\; \Omega_G^*(X) & \;\ra\; \left[ S(\LK^*) \otimes \Omega_{H-\mathrm{hor}}^*(X) \right]^G \;=\; \left[ S(\LK^*) \otimes \Omega_{H-\mathrm{bas}}^*(X) \right]^K \\
\alpha \quad & \;\mto\; \quad \alpha_A \;:=\; c_A(\alpha) \nonumber
\end{align}
\begin{dfn}
A $G$-equivariant differential form $\alpha$ is called \emph{$H$-basic}, if it is an element of
$$
C_K \left( \Omega_{H-\mathrm{bas}}^*(X) \right) = \left[ S(\LK^*) \otimes \Omega_{H-\mathrm{bas}}^*(X) \right]^K \subset \left[ S(\LG^*) \otimes \Omega^*(X) \right]^G = \Omega_G^*(X).
$$
\end{dfn}

So in case of a regular $H$-action one can use the Cartan map to make $G$-equivariant differential forms $H$-basic.

\begin{lem} \label{lem:basic_forms}
Let $\alpha \in \Omega_G^*(X)$ be $H$-basic and denote by $\alpha^{[m]}$ those components of $\alpha$ with form part having degree $m$. Then the following holds.
\begin{enumerate}
\item $\mathrm{d}_G \alpha = \mathrm{d}_K \alpha$.
\item If the $H$-action is regular then $\alpha^{[m]} = 0$ for all $m > \dim(X) - \dim(H)$.
\item If the $H$-action is regular and $m \ge \dim(X) - \dim(H) - 2$ then
$$
\left( \mathrm{d}_G \alpha \right)^{[m]} = (1 \otimes \mathrm{d}) \left( \alpha^{[m-1]} \right).
$$
\end{enumerate}
\end{lem}

\begin{proof}
The first statement follows because by definition $\iota_\eta \omega = 0$ for all $\eta \in \LH$ and  $\omega \in \Omega_{H-\mathrm{bas}}^*(X)$. For the second statement observe that in case of a regular $H$-action the vectors $X_\eta$ with $\eta \in \LH$ span a $\dim(H)$-dimensional distribution. The third statement follows by plugging the second into the definition of $\mathrm{d}_G$.
\end{proof}

The proof of proposition \ref{prop:gen_Cartan} is now completed by the list of properties of the generalized Cartan map below. In particular the identification
$$
H_G^*(X) \;\cong\;  H ( C_K ( \Omega_{H-\mathrm{bas}}^*(X) ) , \mathrm{d}_K )
$$
does not depend on the chosen connection $A$. We refer to Cieliebak, Mundet and Salamon \cite[Theorem 3.8]{CMS} for the proof of the following:

\begin{prop} \label{prop:Cartan_map_properties}
Let $A$ be a $G$-equivariant $H$-connection on $X$ and $c_A$ the corresponding Cartan map.
\begin{enumerate}
\item $(\mathrm{d}_G \alpha)_A = \mathrm{d}_K (\alpha_A)$, i.e.~the operator $c_A$ is a chain map.
\item If $\mathrm{d}_G \alpha = 0$ and $A'$ is another $G$-equivariant $H$-connection then there exists an $H$-basic element $\beta \in \Omega_G^*(X)$ such that $\alpha_A - \alpha_{A'} = \mathrm{d}_K \beta$.
\item The operator $c_A$ is chain homotopic to the identity.
\end{enumerate}
\end{prop}

\section{Invariant integration}

\label{chap:invariant_integration}

In \cite{AB} Atiyah and Bott explained how the push-forward operation in ordinary cohomology extends to the equivariant theory. On the other hand Cieliebak and Salamon \cite{CS} introduced $G$-invariant integration over manifolds with locally free $G$-action. If one understands this invariant integration as the push-forward of the map $X/G \ra \{ \pt \}$ it is clear that equivariant push-forward and invariant integration should be two special cases of one generalized integration operation. In this section we will construct this \emph{invariant push-forward}. Again this is well known if one is willing to deal with fiber integration along orbibundles. The point of this section is to give a clear picture without leaving the world of smooth manifolds.

We end this section with the computation of invariant push-forward in the simple example of a torus acting on $S^{2n-1}$ via a homomorphism onto $S^1$. This example explains very clearly the origin of the residue operations that later appear in our localization formula, and hence also the residue operations in the wall crossing formula of Cieliebak and Salamon \cite{CS}.

\subsection{Equivariant push-forward}

\label{sec:equivariant_push_forward}

Given a proper map $f : X \ra Y$ between oriented manifolds one has the push-forward in ordinary cohomology
$$
f_* \;:\; H^*(X) \ra H^{*-q}(Y),
$$
which shifts the degree by $q = \dim(X) - \dim(Y)$. If $f$ is a fiber bundle the push-forward is integration along the fiber. If $f$ is an inclusion the push-forward is defined via the Thom isomorphism for the normal bundle of $f(X) \subset Y$ . A general map $f$ is decomposed into an inclusion and a fibration by the graph construction.

Now if $X$ and $Y$ carry orientation preserving $G$-actions and $f$ is equi\-va\-ri\-ant one obtains an induced map $f^G : X_G \ra Y_G$ between the Borel constructions and the above construction can be applied to give the equivariant push-forward
$$
f^G_* \;:\; H_G^*(X) \ra H_G^{*-q}(Y).
$$

\begin{rem} \label{rem:equivariant_push_forward_in_Cartan_model}
In the Cartan model one can apply the ordinary push-forward to the form part of equivariant differential forms. For fiber bundles it is shown in Guillemin and Sternberg \cite[Section 10.1]{GS} that this operation indeed defines a map on equivariant cohomology. And going through the proof for the equivalence of the Borel and the Cartan model one can show that this operation in fact agrees with the above definition of equivariant push-forward in the Borel picture. Of course this is valid for any proper, equivariant map and not only for fibrations.
\end{rem}

\subsection{$G$-invariant integration}

\label{sec:G_invariant_integration}

Suppose $X$ is a compact, oriented $G$-manifold with a locally free action. In this case proposition \ref{prop:gen_Cartan} tells us that $H_G^*(X) = H(\Omega_{G-\mathrm{basic}}^*(X),\mathrm{d})$ with the usual differential on ordinary differential forms. Now basic forms can be integrated over slices for the group action. For this purpose we have to fix an orientation on the Lie group $G$. Recall that Lie groups have trivial tangent bundle and are hence orientable, and the action of any subgroup of $G$ by multiplication on G is orientation preserving.

For every point $x \in X$ there exists a triple $(U_x,\varphi_x,G_x)$ with the following properties (see for example Audin \cite[Theorem I.2.1]{Aud}):

\begin{itemize}

\item $G_x \subset G$ is a finite subgroup.

\item $U_x \subset \R^m$, $m := \dim(X) - \dim(G)$ is an oriented, $G_x$-invariant open neighbourhood of zero with compact closure and orthogonal $G_x$-action.

\item $\varphi_x : U_x \ra X$ is a $G_x$-equivariant embedding such that $\varphi_x(0) = x$ and the induced map
\begin{align*}
G \x_{G_x} U_x & \ra \quad X \\
[g,u] \quad & \mto g . \varphi_x(u)
\end{align*}
is a $G$-equivariant, orientation preserving diffeomorphism onto a $G$-invariant open neighbourhoud of $x$. Here $g \in G_x$ acts on $G$ by right-multiplication with $g^{-1}$ and $G$ acts on $G \x_{G_x} U_x$ by left-multiplication on the $G$-factor.

\end{itemize}

Choose finitely many local slices $(U_i,\varphi_i,G_i)$ such that the open sets $G.\varphi_i(U_i)$ cover $X$ and also choose a partition of unity $(\rho_i)$ by $G$-invariant functions $\rho_i$ subordinate to this cover. Define a map
$$
\int_{X/G} \;:\; \Omega_{G-\mathrm{bas}}^*(X) \ra \R \;\cong\; \Omega^*(\pt)
$$
by setting
\begin{equation} \label{def:G_inv_int}
\int_{X/G} \alpha \;:=\; \sum_{i} \frac{1}{|G_i|} \int_{U_i} \varphi_i^*( \rho_i \alpha ).
\end{equation}
It is shown by Cieliebak and Salamon \cite[Proposition 4.2]{CS} that the integral $\int_{X/G}$ does not depend on the choices for the local slices and the partition of unity. Furthermore Stokes' formula still holds and thus the integral descends to $H_G^*(X)$. The integral vanishes unless $\deg(\alpha) = \dim(X) - \dim(G)$. Instead of requiring $X$ to be compact one can equally well restrict to forms with compact support.

\subsection{$H$-invariant push-forward}

\label{sec:H_invariant_push_forward}

Suppose a normal subgroup $H \lhd G$ of positive dimension acts locally freely on a compact oriented $G$-manifold $X$. Then the equivariant push-forward $\pi^G_*$ of the projection $\pi : X \ra \{ \pt \}$ vanishes: By proposition \ref{prop:gen_Cartan} any $G$-equivariant class on $X$ can be represented by an $H$-basic $G$-equivariant differential form. Any such object vanishes upon integration over $X$: By lemma \ref{lem:basic_forms} it has a form part of degree at most $\dim(X) - \dim(H) < \dim(X)$. This shows that in presence of a regularly acting subgroup $H$ the $G$-equivariant push-forward is not the correct operation. One should quotient out any such subgroup before pushing forward, as it is done in the case of $G$-invariant integration.

We will restrict our construction to the case of fiber bundles because this is all we need in our applications and we do not want to argue with Thom forms at this point. But of course the extension to general maps works just as explained by Atiyah and Bott \cite{AB}.

Let $f : X \ra Y$ be an oriented $G$-equivariant fiber bundle. Denote by $\Omega_0^*(X)$ differential forms on $X$ with compact support and by $\Omega_\mathrm{vc}^*(X)$ those with vertically (i.e.~fiberwise) compact support.

\begin{prop} \label{prop:invariant_integration}
Suppose that a closed oriented normal subgroup $H \lhd G$ acts locally freely on $X$ and trivially on $Y$. Denote by $K := G / H$ the quotient group. Then there exists a linear map
$$
\left( f / H \right)_* \;:\; \Omega_{\mathrm{vc},H-\mathrm{bas}}^*(X) \ra \Omega^{*-q}(Y)
$$
with $q = \dim(X) - \dim(Y) - \dim(H)$ which satisfies the following properties.

\begin{enumerate}

\item The map $(f/H)_*$ induces a $\mathrm{d}_K$-chain map on the $K$-Cartan complexes $C_K( \Omega_{\mathrm{vc},H-\mathrm{bas}}^*(X) ) \ra  C_K( \Omega^{*-q}(Y) )$. The induced map on homology is denoted by
$$
\left( f / H \right)^G_* \;:\; H \left( C_K \left( \Omega_{\mathrm{vc},H-\mathrm{bas}}^*(X) \right) , \mathrm{d}_K \right) \ra H_K^*(Y)
$$
and is called \emph{$H$-invariant push-forward}.

\item For the trivial group $H = \{ e \} \subset G$ we recover the equivariant push-forward of Atiyah and Bott: $( f / \{ e \} )^G_* = f^G_*$.

\item For $Y = \{ \pt \}$ and $H = G$ we recover the $G$-invariant integration of Cieliebak and Salamon: $( f / G )^G_* = \int_{X/G}$.

\item For $\alpha \in \Omega_{0,H-\mathrm{bas}}^*(X)$ and the projections $\pi_X,\pi_Y : X,Y \ra \{ \pt \}$ we have the functoriality property
$$
\left( \pi_X / H \right)^G_* (\alpha) \;=\; \left( \pi_Y \right)^K_* \circ \left( f / H \right)^G_* (\alpha).
$$
If the remaining $K$-action on $Y$ is regular then
$$
\left( \pi_X / G \right)^G_* (\alpha) \;=\; \left( \pi_Y / K \right)^K_* \circ \left( f / H \right)^G_* (\alpha).
$$

\item For all $\alpha \in \Omega_{0,H-\mathrm{bas}}^*(X)$ and $\beta \in \Omega^*(Y)$ we have
$$
\left( f / H \right)_* \left( \alpha \wedge f^*\beta \right) \;=\; \left( f / H \right)_* \alpha \wedge \beta,
$$
i.~e.~$(f/H)_*$ is a homomorphism of $\Omega^*(Y)$-modules.

\end{enumerate}

\end{prop}

\begin{rem}
If $X$ is compact and $Y = \{\pt\}$ then we also write the $H$-invariant push-forward $(\pi/H)^G_*$ as
$$
\int_{X/H} \;:\; H_G^*(X) \ra H_K^*(\pt) = S(\LK^*)^K.
$$
\end{rem}

\begin{rem}
We usually drop the superscript $G$ if the big group with respect to which we do the equivariant theory is clear from the context.
\end{rem}

\begin{proof}
We first explain the situation in the Borel model. For this we have to assume that $H$ acts freely on $X$. Then we can write the Borel construction for $X$ as $X_G = (X/H)_K$. Since $H$ acts trivially on $Y$ the $G$-equivariant map $f$ descends to a $K$-equivariant map $f/H : X/H \ra Y$ and induces a map $(f/H)^G : X_G \ra Y_K$ between the Borel constructions. We then take the push-forward of this map in cohomology to obtain $(f/H)^G_*$. It shifts the degree by $\dim(X/H) - \dim(Y)$, which coincides with the formula in the proposition.

We now define the $H$-invariant push-forward for $f : X \ra Y = \{ \pt \}$ in the Cartan model. Forget about the $G$ action and consider the $H$-invariant integral
$$
( f / H )_* \;:=\; \int_{X/H} : \Omega_{0,H-\mathrm{bas}}^*(X) \ra \Omega^*( \pt )
$$
as in \ref{def:G_inv_int}. Given one collection of local slices $(U_i,\varphi_i,H_i)$ for the $H$-action with partition functions $(\rho_i)$ and any element $g \in G$ one obtains another admissible collection of slices and partition functions by
$$
\left(\, U_i \,,\, g \varphi_i \,,\, g H_i g^{-1} \,\right) \quad \mathrm{and} \quad ( g^*\rho_i ).
$$
This shows that integration $\int_{X/H}$ is $G$-invariant. In order to extend a map to the $K$-Cartan models this map has to be $K$-equivariant. So since $K$ acts trivially on $\Omega^*( \pt )$ this extension works. Now
$$
\left( C_K \left( \Omega^*( \pt ) \right) , \mathrm{d}_K \right) = \left( S(\LK^*)^K , 0 \right),
$$
hence it suffices to show that the integral $\int_{X/H}$ of $\mathrm{d}_K$-exact forms vanishes. This follows by lemma \ref{lem:basic_forms} and Stokes' formula. Thus property $3$ is satisfied.

Now consider a fiber bundle $f : X \ra Y$ with $H$ acting trivially on the base. For a point $y \in Y$ pick a chart $\psi_y : V_y \ra Y$ with $V_y \subset \R^n$, $n:=\dim(Y)$, an open neighbourhood of zero with compact closure, and an $H$-equivariant trivialization of $\psi_y^* X \cong V_y \x X_y$ with $H$ acting trivially on $V_y$ and by the induced action on the fiber $X_y \subset X$ over $y$. Choose local slices $(U_i,\varphi_i,H_i)$ and partition functions $(\rho_i)$ for the regular $H$-action on $X_y$ as before. These give rise to local slices
$$
\left(\, V_y \x U_i \,,\, \psi_y \x \varphi_i \,,\, H_i \,\right)
$$
for the $H$-action on $X$ covering all of $f^{-1}(\psi_y(V_y))$. Recall that $U_i \subset \R^m$ with $m = \dim(X_y) - \dim(H) = \dim(X) - \dim(Y) - \dim(H) = q$. Denote integration over the $\R^m$-coordinates of $V_y \x U_i \subset \R^n \x \R^m$ by
$$
{\pi_{(y,i)}}_* \;:\; \Omega^*(V_y \x U_i) \ra \Omega^{*-m}(V_y)
$$
and define the map
$$
\left( f / H \right)_* \;:\; \Omega_{\mathrm{vc},H-\mathrm{bas}}^*(X) \ra \Omega^{*-m}(Y)
$$
for forms $\alpha$ with support in $f^{-1}(\psi_y(V_y))$ by
$$
\left( f / H \right)_* \alpha \;:=\; \sum_{i} \frac{1}{|H_i|} {\pi_{(y,i)}}_* \left[ \left( \psi_y \x \varphi_i \right)^* ( \rho_i \alpha ) \right].
$$
Extend this to all of $\Omega_{\mathrm{vc},H-\mathrm{bas}}^*(X)$ by choosing a locally finite atlas $(\psi_{y_j})$ for $Y$, trivializing $X$. Property $5$ now follows immediately from this definition and the corresponding property for the fiber integrals ${\pi_{(y_j,i)}}_*$. Note that $f^*\beta$ is $H$-basic for every $\beta \in \Omega^*(Y)$ since $H$ acts trivially on $Y$ and $f$ is equivariant.

Now as before the $G$-action on $X$ moves the above $H$-slices around. This shows that the map $(f/H)_*$ is $K$-equivariant. It commutes with the ordinary differential $\mathrm{d}$ and the contractions $\iota_\xi$, because ordinary fiber integration does (see for example Bott and Tu \cite[6.14.1]{BT} and Guillemin Sternberg \cite[10.5]{GS}). This proves part $1$. Part $2$ follows by remark \ref{rem:equivariant_push_forward_in_Cartan_model}. The first functoriality property in part $4$ is a consequence of Fubini's theorem. For the second identity first note that regularity of the remaining $K$-action on $Y$ together with regularity of the $H$-action on $X$ implies that the whole group $G$ acts locally freely on $X$. Hence all operations are well defined. Now for a point $x \in X$ the orders of the involved isotropy subgroups satisfy
$$
|G_x| \;=\; |H_x| \cdot |K_{f(x)}|
$$
and the second identitiy also follows by applying Fubini's theorem.
\end{proof}

\subsection{Example}

\label{sec:example_invariant_integration}

The first invariant integrals that we can compute by hand without the technique of localization come from torus actions on $S^{2n-1}$ that factor through a weighted $S^1$-action. Let $T$ be a torus with Lie algebra $\LT$. We denote the integral lattice in $\LT$ by
$$
\Lambda \;:=\; \left\{ \xi \in \LT \;|\; \exp \left( \xi \right) = \unit \in T \right\}
$$
and its dual by
$$
\Lambda^* \;:=\; \left\{ w \in \LT^* \;|\; \forall \xi \in \Lambda : \left\langle w,\xi \right\rangle \in \Z \right\}.
$$
Here $\langle \;,\; \rangle$ denotes the pairing between $\LT$ and its dual space $\LT^*$. We consider the $T$-action on the unit sphere $S^{2n-1} \subset \C^n$ that is induced by the representation
\begin{equation} \label{def:torus_action}
\rho \;:\; T \ra \mathrm{U}(n) \;;\; \exp(\xi) \mto \mathrm{diag}\left( e^{-2 \pi i \langle w_j , \xi \rangle} \right)_{j = 1 , \ldots , n}
\end{equation}
for a given collection of weight vectors $w_j \in \Lambda^* \setminus \{0\}$. We assume in this section that the weights $w_j$ are all positive multiples of one element $w \in \Lambda^*$. We pick a primitive element $e_1 \in \Lambda$ such that $\langle w , e_1 \rangle \ne 0$ and denote by $T_1 \subset T$ the subtorus generated by $e_1$ and by $\LT_1 \subset \LT$ its Lie algebra. The quotient group and its Lie algebra are denoted by $T_0 := T/T_1$ and $\LT_0 := \LT/\LT_1$ respectively. We can identify
$$
\LT_0^* \;\cong\; \left\{ x \in \LT^* \;|\; \langle x , e_1 \rangle = 0 \right\}.
$$
Now $T_1$ acts locally freely on $S^{2n-1}$ with isotropy group $\Z / \langle w , e_1 \rangle \Z$ and we want to compute the $T_1$-invariant push-forward
$$
\left( \pi / T_1 \right)_* \;:\; H_T^*\left(S^{2n-1}\right) \ra H_{T_0}^*(\pt) \cong S\left( \LT_0^* \right)
$$
of the projection $\pi : S^{2n-1} \ra \{\pt\}$. Suppose we are given elements $x_j \in \LT^*$ for $j = 1 , \ldots , m$. Then the product
$$
\mathbbm{x} \;:=\; \prod_{j = 1}^m x_j \in S(\LT^*)
$$
represents a class $[\mathbbm{x}] \in H_T^{2m}(S^{2n-1})$ if we view $\mathbbm{x}$ as a $\mathrm{d}_T$-closed $T$-equivariant differential form. In fact it follows from Kirwan's theorem (Theorem \ref{thm:Kirwan}) that every class in $H_T^*(S^{2n-1})$ can be represented in such a way. Hence our goal is to compute $( \pi / T_1 )_*[\mathbbm{x}]$.

\begin{prop} \label{prop:T_1_invariant_integration}
For any $\xi \in \LT$ we have
\begin{equation} \label{eqn:T_1_invariant_integration}
\left( \pi / T_1 \right)_* \left[ \mathbbm{x} \right] (\xi) \;=\; \frac{1}{2 \pi i} \oint \frac{ \prod_{j = 1}^m \langle x_j , \xi + z e_1 \rangle }{ \prod_{j = 1}^n \langle w_j , \xi + z e_1 \rangle } \; \dz \; .
\end{equation}
The integral is around a circle in the complex plane enclosing all poles of the integrand.
\end{prop}

\begin{proof}
We write $w_j = \ell_j \cdot w$ with positive integers $\ell_j$. By changing $w$ we can arrange that the $\ell_j$ do not have a common divisor different from $1$. Thus the $T$-action factors through a free $S^1$-action with weights $\ell := \left( \ell_1 , \ldots , \ell_n \right)$. The quotient of this free $S^1$-action is the weighted projective space $\CP^{n-1}_\ell$. By the functoriality property of invariant push-forward we obtain $( \pi / T_1 )_*$ as the composition of $T_1$-invariant push-forward $H_T^*(S^{2n-1}) \ra H_{T_0}^*(\CP^{n-1}_\ell)$ with $T_0$-equivariant integration $H_{T_0}^*(\CP^{n-1}_\ell) \ra H_{T_0}^*(\pt)$. Note that $T_0$ acts trivially on $\CP^{n-1}_\ell$, whence $H_{T_0}^*(\CP^{n-1}_\ell) \cong S(\LT_0^*) \otimes H^*(\CP^{n-1}_\ell)$ and this last $T_0$-equivariant integration is just evaluation on the fundamental cycle of $\CP^{n-1}_\ell$ in the second factor.

We need a $T$-equivariant $T_1$-connection $A$ on $S^{2n-1}$. Let $\alpha_\ell \in \Omega^1(S^{2n-1})$ be a connection form on the principal $S^1$-bundle $S^{2n-1} \ra \CP^{n-1}_\ell$. Here we identify $\mathrm{Lie}(S^1) \cong \R$ such that for $x \in \R$ we have $\exp(x) = e^{2 \pi i x}$. It follows that
\begin{equation} \label{def:T_equivariant_T_1_connection}
A := \frac{ \alpha_\ell }{ \langle w , e_1 \rangle } \otimes e_1
\end{equation}
is such a $T$-equivariant $T_1$-connection.

A short computation now shows that the image $\mathbbm{x}_A$ of $\mathbbm{x}$ under the Cartan map $c_A$ is given by
\begin{equation} \label{eqn:xx_A}
\mathbbm{x}_A \;=\; \prod_{j = 1}^m \left( x_j - \langle x_j , e_1 \rangle \frac{ \mathrm{d}\alpha_\ell - w }{ \langle w , e_1 \rangle } \right).
\end{equation}
The first step of integration $H_T^*(S^{2n-1}) \ra H_{T_0}^*(\CP^{n-1}_\ell)$ now simply consists of dividing by $\langle w , e_1 \rangle$ and interpreting $\mathrm{d}\alpha_\ell$ as a form on $\CP^{n-1}_\ell$. Evaluation on the fundamental cycle of $\CP^{n-1}_\ell$ in the second step then amounts to expanding the above expression for $\mathbbm{x}_A$ as a polynomial in $\mathrm{d}\alpha_\ell$ and picking the coefficient in front of ${\mathrm{d}\alpha_\ell}^{n-1}$. On can express this as taking the residue of a certain rational complex function. We leave the details to the reader and refer to \cite{JW} for the detailed computation.
\end{proof}

\begin{rem}
Formula \ref{eqn:T_1_invariant_integration} remains true even without the assumption of this section that the weights $w_j$ are colinear. To prove this we need the technique of localization. So we will return to this in the example at the end of the next section.
\end{rem}

\begin{rem} \label{rem:z}
The above computations yield the following recipe how to evaluate $( \pi / T_1 )_*$ on any $T$-equivariant differential form $\alpha \in \Omega_T^*(S^{2n-1})$ that is $T_1$-basic and has the following polynomial presentation
$$
\alpha \;=\; \sum_{k \ge 0} r_k \cdot \left( \frac{ \mathrm{d}\alpha_\ell - w }{ \langle w , e_1 \rangle } \right)^k
$$
for some $r_k \in S(\LT^*)$ and with $\alpha_\ell$ from \ref{def:T_equivariant_T_1_connection}. Note that the property of being $T_1$-basic imposes constraints on the $r_k$ in terms of $w$ and $e_1$. But we do not need to know them explicitly. It suffices to know that $\alpha = \alpha_A$. So in our computations we can replace $\mathbbm{x}_A$ by the above sum, because by \ref{eqn:xx_A} it is of exactly this form. We obtain the generalized formula
\begin{eqnarray} \label{eqn:gen_T_1_invariant_integration}
\left( \pi / T_1 \right)_* \left[ \alpha \right] (\xi) & = & \frac{1}{2 \pi i} \oint \frac{ \sum_{k \ge 0} r_k(\xi) \cdot z^k }{ \prod_{j = 1}^n \langle w_j , \xi + z e_1 \rangle } \; \dz \nonumber\\
& = & \Res_{z = - \frac{ \langle w , \xi \rangle }{ \langle w , e_1 \rangle } } \left\{ \frac{ \sum_{k \ge 0} r_k(\xi) \cdot z^k }{ \prod_{j = 1}^n \langle w_j , \xi + ze_1 \rangle } \right\},
\end{eqnarray}
that will be used later on.
\end{rem}

\section{Localization}

\label{chap:localization}

The abstract localization theorem of Atiyah and Bott in \cite{AB} gives rise to a formula for integrating equivariant $n$-forms over an $n$-dimensional manifold $X$ that is equipped with an action of a torus $T$. In the notation of section \ref{chap:invariant_integration} this integration is the restriction of the $T$-equivariant push-forward of $X \ra \{ \pt \}$ to degree $n$. The integration formula expresses such an integral in terms of data associated to the fixed point set of the $T$-action.

This fits neatly with our observation at the beginning of section \ref{sec:H_invariant_push_forward}: Such an integral vanishes as soon as there is a subgroup $H$ of positive dimension acting locally freely, i.e.~if there are no fixed points. But in that case $H$-invariant integration is defined and one expects a generalized localization formula for this adapted integration operation. In fact if $H$ acts freely one can lift the localization formula for the $T/H$-action on $X/H$ to obtain a formula for $H$-invariant integration over $X$. Since localization only works in a truely equivariant context, we have to assume that $T/H$ still has positive dimension, i.~e.~the codimension of $H$ in $T$ has to be at least one.

We will generalize the Atiyah-Bott localization formula to $H$-invariant push-forward in the case of a regular $H$-action and call the result \emph{relative Atiyah-Bott localization}. Before we can do this we start with some general facts and constructions concerning torus actions. Then we review the construction and properties of equivariant Thom forms. We finish the section with the computation of an explicit example that generalizes the one in section \ref{sec:example_invariant_integration} and that will be used later on.

\subsection{Torus actions}

\label{sec:torus_actions}

Suppose the $k$-dimensional torus $T$ acts on a closed $n$-dimensional manifold $X$. We denote the infinitesimal action at a point $p \in X$ by
$$
L_p \;:\; \LT \ra T_pX \;;\; \xi \mto X_\xi(p),
$$
with the infinitesimal vector field $X_\xi$ defined as in \ref{def:X_xi}, set $\LT_p := \ker \left( L_p \right) \subset \LT$ and denote the isotropy subgroup at a point $p \in X$ by $S_p \subset T$. Note that $\LT_p = \mathrm{Lie}(S_p)$.

Since $T$ is abelian there is only a finite number of isotropy groups $S_p$. This follows by the slice theorem (see for example tom Dieck \cite[Theorem 5.11]{tomD}). Hence the set of all subspaces in $\LT$ that occur as $\LT_p$ for some $p \in X$ is also finite. Let $k_p := \dim(\LT_p)$ and set
$$
k_0 \;:=\; \max_{p \in X} \, k_p \quad \mathrm{and} \quad k_1 \;:=\; k - k_0.
$$
We fix a $k_1$-dimensional subtorus $T_1 \subset T$ such that its Lie algebra $\LT_1 \subset \LT$ satisfies
$$
\LT_1 \cap \LT_p \;=\; \left\{ 0 \right\} \quad \mbox{for all} \quad p \in X.
$$
So $T_1$ is a subtorus of maximal dimension that acts locally freely on $X$ and $T_1$-invariant integration is the operation that we want to study. Invariant integration with respect to any larger subgroup is not possible and invariant integration with respect to any smaller subgroup would be zero. Of course there are many possible choices for $T_1 \subset T$ with the above property.

\begin{dfn} \label{def:relative_fixedpoint_set}
Given any closed subgroup $H \subset T$ we define the \emph{$H$-relative fixed point set} $F_H$ as
$$
F_H \;:=\; \left\{ p \in X \;|\;  L_p(\LT) = L_p(\LH) \right\}.
$$
Here $\LH \subset \LT$ denotes the Lie algebra of $H$.
\end{dfn}

Elements of $F_H$ represent fixed points of the $T / H$-action on the quotient $X / H$. So $F_H$ is the subset to which we would like to localize $H$-invariant integration. One prerequisite is the following

\begin{prop} \label{prop:F_H_submanifold}
Suppose the closed subgroup $H \subset T$ acts locally freely on $X$. Then the set $F_H \subset X$ is a smooth (but not necessarily connected) submanifold.
\end{prop}

\begin{proof}
It suffices to show that every $p \in F_H$ has an open neighbourhood $U \subset X$ such that $F_H \cap U$ is a submanifold in $U$. For $U$ we pick a tubular neighbourhood of the $T$-orbit through $p$. By the local slice theorem such a tube can be chosen to be equivariantly diffeomorphic to
$$
U \;\cong\; T \x_{S_p} V.
$$
Here $V$ is a linear space on which the isotropy group $S_p$ acts by orthogonal transformations. The orbit of the $T$-action through $p$ is identified with the points $[g,0] \in  U$ with $g \in T$ and $0 \in V$. Now for any point $q \in X$ we have
$$
\dim L_q(\LT) \;=\; \dim\LT - \dim\LT_q \;=\; k - k_q,
$$
and
$$
\dim L_q(\LH) \;=\; \dim\LH - \dim(\LH \cap \LT_q) \;=\; \dim\LH - \dim\LH_q \;=\; l - l_q,
$$
where we introduced $\LH_q := \LH \cap \LT_q$ and $\dim\LH =: l$ and $\dim\LH_q =: l_q$. So we have
$$
q \in F_H \quad \iff \quad k \;=\; k_q + l  - l_q,
$$
which is equivalent to saying that the intersection $\LT_q \cap \LH \subset \LT$ is transversal. The assumption of a regular $H$-action implies that $\LH \cap \LT_q = \{ 0 \}$ for all $q \in X$. Hence we get the refined equivalence
$$
q \in F_H \quad \iff \quad \LT \;=\; \LH \oplus \LT_q.
$$
For a point $q = [g,v] \in U$ we have $S_q = S_{p,v} \subset S_p$, where $S_{p,v}$ denotes the isotropy subgroup at $v$ for the orthogonal $S_p$-action on $V$. We write
$$
\LT_v \;:=\; \mathrm{Lie}(S_{p,v}) \subset \LT_p
$$
instead of $\LT_q$. Now we consider the subset
$$
W \;:=\; \left\{ v \in V \;|\; \LT = \LH \oplus \LT_v \right\}.
$$
Since $p \in F_H$ we have $\LT = \LH \oplus \LT_p$ and since $\LT_v \subset \LT_p$ for all $v \in V$ we in fact have $W = \left\{ v \in V \;|\; \LT_v = \LT_p \right\}$. This shows that $W$ is a linear and $S_q$-invariant subspace of $V$. Thus we finally obtain
$$
F_H \cap U \;\cong\; T \x_{S_p} W,
$$
which is a submanifold of $U \cong T \x_{S_p} V$.
\end{proof}

\begin{rem}
This proposition is not true without the assumption of $H$ acting locally freely. Consider the action of the two-torus $T = S^1 \x S^1$ on $\C \x \C$, where the first $S^1$ acts by complex multiplication on the first component and trivially on the second, and the second $S^1$ acts vice versa. Let $H \cong S^1$ be the diagonal in $T$. Then we claim that
$$
F_H \;=\; \left\{ ( z_1 , z_2 ) \;|\; z_1 \cdot z_2 = 0 \right\} \subset \C \x \C.
$$
In fact at points $p = ( z_1 , z_2 )$ with $ z_1 \cdot z_2 \ne 0$ the $T$-isotropy is trivial, so $p \not\in F_H$. But as soon as one component of $p$ is zero the tangent space to one $S^1$ factor in $T = S^1 \x S^1$ is contained in $\LT_p$. So $\LT_p$ intersects transversally with $\LH$ and hence $p \in F_H$.
\end{rem}

\subsection{Equivariant Thom forms}

\label{sec:equivariant_thom_forms}

The construction in this section does not require the acting group to be abelian. So for the moment we consider an arbitrary compact Lie group $G$ acting smoothly on a closed and oriented $n$-dimensional manifold $X$, preserving the orientation.

Given a $G$-equivariant differential form $\alpha$ one can integrate all the form-parts over $X$ to obtain an element in $S(\LT^*)$. If we restrict to $\alpha \in \Omega_G^n(X)$ then we in fact obtain $\int_X \alpha \in \R$, because only forms of degree $n$ contribute and so the polynomial remainder is of degree zero. This is the operation that in the following will be called \emph{integration} of an equivariant form and is denoted as above.

Suppose $i : Z \hookrightarrow X$ is a closed $G$-invariant submanifold of codimension $m$. Then given any tubular neighbourhood $U \subset X$ of $Z$ there exists an equivariant Thom form $\tau$ with support in $U$, i.~e.~an element $\tau \in \Omega_G^m(X)$ such that
\begin{itemize}
\item $\mathrm{d}_G\,\tau = 0$,
\item $\mathrm{supp}(\tau) \subset U$,
\item $\D \int_X \alpha\wedge\tau \;=\; \int_Z i^*\alpha$ for all closed forms $\alpha \in \Omega^*(X)$.
\end{itemize}
See Cieliebak and Salamon \cite[Chapter 5]{CS} for an explicit construction or Guillemin and Sternberg \cite[Chapter 10]{GS} for a canonical algebraic discussion. We give a short summary in order to see what happens in presence of a normal subgroup $H$ that acts locally freely.

The key ingredient is an $\mathrm{SO}(m)$-equivariant \emph{universal Thom form} on $\R^m$. This is a $\mathrm{d}_{\mathrm{SO}(m)}$-closed form $\rho \in \Omega_{\mathrm{SO}(m)}^m(\R^m)$ with compact support and integral equal to $1$. This universal Thom form is transplanted to the normal bundle $N$ of $Z \subset X$ by the following construction. Let $P$ be the principal $\mathrm{SO}(m)$-bundle of oriented orthonormal frames of $N$, such that
$$
N \;\cong\; P \x_{\mathrm{SO}(m)} \R^m.
$$
Of course we use a $G$-invariant metric on $X$ to define $P$. The $G$-action on $Z$ then extends to $P$ and commutes with the $\mathrm{SO}(m)$-action. We pull back $\rho$ to $P \x \R^m$ and use a $G \x \mathrm{SO}(m)$-equivariant $\mathrm{SO}(m)$-connection $A$ to get the $\mathrm{SO}(m)$-basic form $\rho_A$ by virtue of the Cartan map \ref{def:c_A}. Thus $\rho_A$ descends to a $\mathrm{d}_G$-closed form $\tau$ on $N$ with compact support and fiber integral equal to $1$. Finally $N$ can be identified equivariantly with an arbitrarily small tubular neighbourhood of $Z$.

Now given a closed normal subgroup $H \lhd G$ that acts with at most finite isotropy on $X$ we can in fact take a $G \x \mathrm{SO}(m)$-equivariant $H \x \mathrm{SO}(m)$-connection $A$ in the above construction. In the end we obtain a Thom form $\tau \in \Omega_G^m(X)$ that is $H$-basic. If we denote by $l := \dim(H)$ the dimension of $H$ then we obtain the following result.

\begin{prop} \label{prop:H_invariant_Thom_integration}
Let $\alpha \in \Omega_{H-\mathrm{bas}}^{n-m-l}(X)$ be closed and $\tau$ be a $G$-equivariant Thom form for $Z \subset X$ as above. Then
$$
\int_{Z/H} i^*\alpha \;=\; \int_{X/H} \alpha \wedge \tau.
$$
\end{prop}

\begin{proof}
The left hand side is just the $H$-invariant integral of an ordinary $H$-basic form in the sense of section \ref{sec:G_invariant_integration} and hence a real number. The right hand side is the $H$-invariant integral of a $G$-equivariant form. But since $\tau$ has degree $m$ its polynomial parts have degree less than $m$ and vanish upon integration, because $\alpha$ is assumed to have degree $n-m-l$. Hence we only have to look at $\tau^{[m]}$, i.~e.~the form part of $\tau$ in degree $m$.

It is shown by Cieliebak and Salamon \cite[Theorem 5.3]{CS} that the difference of any two Thom forms is $\mathrm{d}_G$-exact. Hence the right hand side does not depend on the particular choice for $\tau$ and we may assume that $\tau$ is $H$-basic. Next we observe that $\tau^{[m]}$ for an $H$-basic $G$-equivariant Thom form $\tau$ is in fact an $H$-equivariant Thom form for $Z \subset X$ if we just consider the $H$-action. Hence the result follows from Cieliebak and Salamon \cite[Corollary 6.3]{CS}, where the corresponding result is proven for $G$-invariant integration in the case of a regular $G$-action.
\end{proof}

We want to extend this rule to $H$-invariant integration of elements $[ \alpha ] \in H_G^*(X)$ with values in $H_K^*(\pt) = S(\LK^*)$ with $K := G/H$. As shown in proposition \ref{prop:gen_Cartan} we can represent the given class by a form
$$
\alpha \;\in\; C_K \left( \Omega_{H-\mathrm{bas}}^*(X) \right) \;=\; \left[ S(\LK^*) \otimes \Omega_{H-\mathrm{bas}}^*(X) \right]^K.
$$
Now since $\mathrm{d}_G \alpha = 0$ we have $\mathrm{d} \left( \alpha^{[n-m-l]} \right) = 0$ by lemma \ref{lem:basic_forms}. Hence the above lemma applies and since $H$-invariant integration is $K$-equivariant the formula extends to $\mathrm{d}_G$-closed elements of $C_K \left( \Omega_{H-\mathrm{bas}}^*(X) \right)$. So we conclude that
\begin{equation} \label{eqn:H_invariant_Thom_integration}
 \int_{X/H} \alpha \wedge \tau \;=\; \int_{Z/H} i^*\alpha \;\in\; S(\LK^*)
\end{equation}
for all $[ \alpha ] \in H_G^*(X)$, where $\tau$ is any $G$-equivariant Thom form for the $G$-invariant submanifold $Z \subset X$.

\subsection{Relative Atiyah-Bott localization}

\label{sec:relative_atiyah_bott_localization}

If one wants to use equation \ref{eqn:H_invariant_Thom_integration} in order to compute an integral $\int_{X/H} \beta$ then there are two obstacles: We can only integrate forms $\beta$ that can be written as $\alpha \wedge \tau$ and hence have support localized around some submanifold with Thom form $\tau$. And then one needs to 'divide' $\beta$ by $\tau$ in order to obtain $\alpha$. Generally this division will not be possible in $H_G^*(X)$ --- one needs to work in a localized ring.

In the case of a torus action both these steps of localization can be made explicit, so we return to the setup of section \ref{sec:torus_actions}. The construction is analogous to the one by Guillemin and Sternberg \cite[Chapter 10]{GS}. We first consider the case that the torus $T_1$, which acts locally freely and with respect to which we want to integrate invariantly, has codimension $1$ in the whole torus $T$. This corresponds to looking at an $S^1$-action in the non-relative case.

\subsubsection*{Geometric localization}

Denote by $Z_i$ the connected components of the $T_1$-relative fixed point set $F_{T_1}$. Let $U_i$ be sufficiently small disjoint $T_1$-invariant tubular neighbourhoods of the $Z_i$. We introduce the notation
$$
X' \;:=\; X \setminus F_{T_1}.
$$
Recall that $\dim(T_1) = k_1$, so the $T_1$-invariant integral over $X$ is nonzero only for forms of degree at least $n - k_1$.

\begin{lem} \label{lem:geometric_localization}
For any $[ \alpha ] \in H_T^d(X)$ with $d \ge n - k_1$ there exists a collection of $\mathrm{d}_T$-closed forms $\alpha_i$ with support in $U_i$ such that
$$
\int_{ X / T_1 } \alpha \;=\; \sum_i \int_{ X / T_1 } \alpha_i
$$
and furthermore $\alpha$ and $\alpha_i$ agree on some open neighbourhood of $Z_i$.
\end{lem}

\begin{proof}
First observe that $T$ acts locally freely on $X'$. In fact if $T$ had a whole $1$-dimensional subgroup fixing some point $p \in X$ this would imply $p \in F_{T_1}$, because $T_1$ has only codimension $1$ in $T$. This implies $H_T^d(X') = 0$, because any $T$-equivariant form on $X'$ can be made $T$-basic. But since $d \ge n - k_1 > \dim(X) - \dim(T)$ there are no nonzero $T$-basic forms in degree $d$ by lemma \ref{lem:basic_forms}. So if we restrict $\alpha$ to $X'$ we obtain
$$
\alpha \;=\; \mathrm{d}_T \gamma' \quad \mbox{for some} \quad \gamma' \in \Omega_T^{d-1}(X').
$$
For every $i$ we fix a smooth $T$-invariant function $\rho_i$ with support in $U_i$ such that $\rho_i = 1$ on some open neighbourhood of $Z_i$. Then we introduce
$$
\gamma \;:=\; \gamma' - \sum_i \rho_i \cdot \gamma'.
$$
Note that $\gamma$ extends over $F_{T_1}$ to an element of $\Omega_T^{d-1}(X)$. Next we set
$$
\beta \;:=\; \alpha - \mathrm{d}_T \gamma
$$
and obtain $\int_{X/T_1} \alpha = \int_{X/T_1} \beta$. Now $\beta$ vanishes outside the disjoint union of all the $U_i$, because there we have $\gamma = \gamma'$. Hence we can write
$$
\beta \;=:\; \sum_i \alpha_i
$$
with $\mathrm{d}_T$-closed forms $\alpha_i$ having support in $U_i$. Finally on a neighbourhood of each component $Z_i$ we have $\rho_i = 1$ and hence $\gamma = 0$. Hence we indeed obtain $\alpha_i = \alpha$ on these neighbourhoods.
\end{proof}

\subsubsection*{Algebraic localization}

We now only consider one component $Z$ of $F_{T_1}$ and a $\mathrm{d}_T$-closed form $\alpha$ with compact support in a small $T$-invariant tubular neighbourhood $U$ of $Z$. We denote by $i : Z \ra U$ the inclusion and by $\pi : U \ra Z$ the projection and choose a Thom form $\tau$ for $Z$ with support in $U$. We denote by
$$
i^*[\tau] \;=:\; e_T(Z) \in H_T^m(Z)
$$
the equivariant Euler class of $Z$. More precisely this is the $T$-equivariant Euler class of the normal bundle $N$ to $Z \subset X$. If $m = \codim(Z)$ is odd then $e_T(Z) = 0$, so we assume that $m = 2 \ell$ is even. From the proof of proposition \ref{prop:F_H_submanifold} we read off that on the connected component $Z$ we can choose a  $T_0 \subset T$ that acts trivially on $Z$ and such that $\LT = \LT_0 \oplus \LT_1$. By the discussion preceding equation \ref{eqn:trivial_H_action} we conclude that
$$
H_T^*(Z) \;=\; S(\LT_0^*) \otimes H_{T_1}^*(Z).
$$
Recall that $T_1$ acts locally freely. So we can represent $e_T(Z)$ by a sum
\begin{equation} \label{eqn:euler_class_sum}
e_T(Z) \;=\; f_\ell + f_{\ell-1} \cdot \theta_1 + \ldots + f_1 \cdot \theta_{\ell-1} + \theta_\ell
\end{equation}
with $f_j \in S^j(\LT_0^*)$ and closed differential forms $\theta_j \in \Omega_{T_1-\mathrm{bas}}^{2j}(Z)$. Since all the terms involving some $\theta_j$ are nilpotent we can formally invert this sum if we assume that $f_\ell \ne 0$. We define
$$
\Theta \;:=\; f_\ell - e_T(Z)
$$
and observe that $\Theta^r = 0$ for $r-1  := \lfloor \dim(Z) / 2 \rfloor$. We set
$$
\beta \;:=\; f_\ell^{r-1} + f_\ell^{r-2} \cdot \Theta + \ldots + f_\ell \cdot \Theta^{r-1}
$$
and compute
\begin{eqnarray*}
e_T(Z) \cdot \beta & = & \left( f_\ell - \Theta \right) \cdot \left( f_\ell^{r-1} + f_\ell^{r-2} \cdot \Theta + \ldots + f_\ell \cdot \Theta^{r-1} \right) \\
& = & f_\ell^r.
\end{eqnarray*}
Formally we can write this as $\frac{1}{e_T(Z)} = \frac{\beta}{f_\ell^r}$. Note that $\beta$ defines an element in $H_T^*(Z)$. Thus we can invert the Euler class $e_T(Z)$ if we localize $H_T^*(Z)$ at the monomial $f_\ell^r$. Without working in a localized ring we can write
\begin{eqnarray} \label{eqn:algebraic_localization}
f_\ell^r \cdot \int_{U/T_1} \alpha & = & \int_{U/T_1} \alpha \wedge \pi^* \beta \wedge \pi^* e_T(Z) \nonumber\\
& = & \int_{U/T_1} \alpha \wedge \pi^* \beta \wedge \tau \\
& = & \int_{Z/T_1} i^* \alpha \wedge \beta. \nonumber
\end{eqnarray}
We write this integration formula more suggestively as
$$
\int_{U/T_1} \alpha \;=\; \int_{Z/T_1} \frac{ i^* \alpha }{ e_T(Z) }
$$
but have in mind that the meaning of this formula is given by equation \ref{eqn:algebraic_localization} above. Thus we have proved the following relative localization theorem.

\begin{thm} \label{thm:relative_localization}
Suppose the subtorus $T_1 \subset T$ acts locally freely on $X$ with relative fixed point set $F_{T_1}$. Then for any $[ \alpha ] \in H_T^*(X)$ we have
$$
\int_{X/T_1} \alpha \;=\; \sum_i \int_{Z_i/T_1} \frac{ i_{Z_i}^* \alpha }{ e_T(Z_i) } \; \in S(\LT_0^*).
$$
Here the sum is over all connected components $Z_i$ of $F_{T_1}$, $i_{Z_i}$ denotes the inclusion into $X$ and $e_T(Z_i)$ is the $T$-equivariant Euler class of the normal bundle to the submanifold $Z_i \subset X$.
\end{thm}

\subsubsection*{Higher codimension}

We now extend the relative localization formula \ref{thm:relative_localization} to the case of $T_1$ having higher codimension in $T$. The procedure is analogous to the extension of the $S^1$-localization formula to arbitrary tori in \cite[Section 10.9]{GS}. The trouble is that the geometric localization does not work well enough if the torus $T$ has positive dimensional isotropy on points $p \in X \setminus F_{T_1}$. But in fact the auxiliary forms $\alpha_i$ in lemma \ref{lem:geometric_localization} do not appear in the final localization formula. So if we want to compute the value of
$$
\int_{X/T_1} \alpha (\xi) \quad \mathrm{for} \quad \xi \in \LT
$$
we can use different local forms $\alpha_i$ for different $\xi$. First suppose that for a given element $\xi \in \LT$ the space $\widetilde{\LT} := \LT_1 \oplus \R\xi$ is the Lie algebra of a closed subgroup $\widetilde{T} \subset T$ such that $T_1 \subset \widetilde{T}$ has codimension $1$. The inclusion $\widetilde{T} \ra T$ induces the restriction map $\LT^* \ra \widetilde{\LT}^*$, which extends to a map on equivariant cohomology $H_T^*(X) \ra H_{\widetilde{T}}^*(X)$. And in fact for any $[ \alpha ] \in H_T^*(X)$ the computation of $\int_{X/T_1} \alpha (\xi)$ factors through this map. Hence we can apply our relative localization formula from \ref{thm:relative_localization} to obtain
$$
\int_{X/T_1} \alpha (\xi) \;=\; \sum_i \int_{Z_i/T_1} \frac{ i_{Z_i}^* \alpha }{ e_T(Z_i) } \, (\xi) \; \in \R.
$$
But now the set of $\xi \in \LT$, for which we derived this formula, is dense in $\LT$. Hence by continuity it actually holds for all $\xi$ and so the localization theorem above indeed holds independently of the codimension of $T_1$ in $T$.

\subsubsection*{The equivariant Euler class}

From the explicit construction of an equivariant Thom form for any $G$-equivariant vector bundle $E \ra X$ in Cieliebak, Mundet and Salamon \cite[Chapter 5]{CMS} one can read off two things: If the rank of $E$ is odd the Euler class vanishes, and if the rank is even then one gets the following representative for the Euler class $e_G(E)$.

\begin{prop}[{\cite[Lemma 4.3]{CS}}] \label{prop:equivariant_Euler_class}
Suppose $E$ is a rank $n$ complex vector bundle and the action of $G$ is complex linear on the fibers. Fix a $G$-invariant Hermitian metric on $E$ and let $P \ra X$ denote the unitary frame bundle of $E$. Let $A \in \Omega^1(P,\LU(n))$ be a $G$-invariant $U(n)$-connection form on $P$. Then the $G$-equivariant Euler class of $E$ is represented by the $\mathrm{d}_G$-closed form $e_G(E)$ that is given by
$$
e_G(E)(\xi) \;=\; \det\left( \frac{i}{2\pi} F_A + \frac{i}{2\pi} A(X_\xi) \right),
$$
where $F_A \in \Omega(P,\LU(n))$ denotes the curvature of $A$.
\end{prop}

In fact assume that $E = X \x \C^n$ is a trivial bundle with a diagonal $G$-action and the action on the fibers is given by a unitary representation
$$
\rho \;:\; G \ra U(n).
$$
Then the frame bundle of $E$ is the product bundle $P = X \x U(n)$ and with the obvious $G$-invariant and flat connection $A$ we obtain
$$
e_G(E) \;=\; \det\left( \frac{i}{2\pi} \dot{\rho} \right).
$$
This computation is contained in the proof of {\cite[Lemma 4.3]{CS}}. Note that the $G$-action on $X$ does not enter the formula. Now suppose the group $G$ is a torus $T$ and $\rho$ is given by
$$
\rho \;:\; T \ra \mathrm{U}(n) \;;\; \exp(\xi) \mto \mathrm{diag}\left( e^{-2 \pi i \langle w_j , \xi \rangle} \right)_{j = 1 , \ldots , n}
$$
for a given collection of weight vectors $w_j \in \Lambda^*$. Then we obtain the representative
\begin{equation} \label{eqn:torus_euler_class}
e_T(E) \;=\; \prod_{j=1}^n w_j.
\end{equation}

\subsection{Example}

\label{sec:example_localization}

We are now able to extend the computation from section \ref{sec:example_invariant_integration} to more general torus actions on $S^{2n-1}$ that do not factor through an $S^1$-action. We consider the $T$-action given by \ref{def:torus_action} on the unit sphere $S^{2n-1} \subset \C^n$. We assume that all weights $w_j$ are nonzero and that no two weights are a negative multiple of each other. But the weights no longer have to be colinear. We take a primitive element $e_1 \in \Lambda$ such that
$$
\langle w_j , e_1 \rangle \ne 0 \quad \mathrm{for\ all}\ j \in \{ 1 , \ldots , n \}.
$$
As before we denote by $T_1 \subset T$ the subtorus generated by $e_1$, by $\LT_1 \subset \LT$ its Lie algebra, by $T_0 := T/T_1$ the quotient group and by $\LT_0 := \LT/\LT_1$ its Lie algebra. In order to apply our localization technique to the $T_1$-invariant push-forward
$$
\left( \pi / T_1 \right)_* \;:\; H_T^*\left( S^{2n-1} \right) \ra H_{T_0}^*(\pt) \cong S\left( \LT_0^* \right)
$$
we have to determine the $T_1$-relative fixed point set
$$
F_{T_1} \;:=\; \left\{ z \in S^{2n-1} \;|\; L_z(\LT) = L_z(\LT_1) \right\}
$$
with the infinitesimal action $L_z$ at the point $z \in \C^n$ given by
$$
L_z \;:\; \LT \ra T_zS^{2n-1} \;;\; \xi \mto \left( 2 \pi i \langle w_j , \xi \rangle \cdot z_j \right)_{j = 1 , \ldots , n} \subset \C^n
$$
if we identify $T_zS^{2n-1} \subset T_z\C^n \cong \C^n$. For that purpose we have to arrange the $w_j$ into sets of colinear weights. Denote by $(w_\nu)_{\nu \in \{ 1 , \ldots , N\}}$ the collection of pairwise different weights, that are also not a multiple of any other weight $w_j$. Then for every $\nu$ there are unique integers $(\ell_{\nu_i})_{i \in \{ 1 , \ldots , n_\nu \}}$ such that the set of all $w_j$ equals
$$
\left\{ \ell_{1_1} w_1 , \ldots , \ell_{1_{n_1}} w_1 , \ldots , \ell_{\nu_1} w_\nu , \ldots , \ell_{\nu_{n_\nu}} w_\nu  , \ldots , \ell_{N_1} w_N , \ldots , \ell_{N_{n_N}} w_N \right\}.
$$
We can assume $\ell_{\nu_1} = 1$ for all $\nu = 1 , \ldots , N$ and furthermore by our assumptions on the $w_j$ we have all $\ell_{\nu_i} > 0$. Note that $\sum_{\nu = 1}^N n_\nu = n$. We introduce
$$
V_\nu \;:=\; \left\{ ( z_1 , \ldots , z_n ) \in \C^n \;|\; z_j = 0 \;\mathrm{if}\; w_j \nparallel w_\nu \right\} \subset \C^n,
$$
the linear subspace of $\C^n$ on which the $T$-action does factor through the weighted $S^1$-action given by the vector $w_\nu$ and the $\ell_\nu := (\ell_{\nu_i})_{i \in \{ 1 , \ldots , n_\nu \}}$. So we have $\dim_\C V_\nu = n_\nu$ and we can write $\C^n \cong \oplus_{\nu = 1}^N V_\nu$. In $V_\nu$ we consider the unit sphere
$$
S_\nu \;:=\; V_\nu \cap S^{2n-1}.
$$

\begin{lem}
The $T_1$-relative fixed point set $F_{T_1}$ is given by the disjoint union
$$
F_{T_1} \;=\; \bigsqcup_{\nu = 1}^N S_\nu.
$$
\end{lem}

\begin{proof}
The $V_\nu$ intersect only in $0 \in \C^n$ and since $0 \not\in S^{2n-1}$ the $S_\nu$ are indeed disjoint. Now if $z \in S_\nu$ then $L_z(\LT)$ is the real one-dimensional space spanned by $\left( i \ell_{\nu_j}z_j \right) \in T_zS^{2n-1}$. The same is true for $L_z(\LT_1)$ since $\langle w_\nu , e_1 \rangle \ne 0$ for any $\nu$. Hence $\sqcup_{\nu = 1}^N S_\nu \subset F_{T_1}$.

If on the other hand $z \in S^{2n-1}$ is not contained in any $S_\nu$ then it has two nonzero components $z_l , z_m \ne 0$ with $w_l \nparallel w_m$. Thus there exists an element $\xi \in \LT$ with $\langle w_l , \xi \rangle = 0$ but $\langle w_m , \xi \rangle \ne 0$. This implies $L_z(\xi) \not\in L_z(\LT_1)$, because by assumption $\langle w_l , e_1 \rangle$ and $\langle w_m , e_1 \rangle$ are both nonzero. Hence $z \not\in F_{T_1}$.
\end{proof}

Next we have to describe the equivariant normal bundles $N_\nu$ to $S_\nu$ in $S^{2n-1}$. They can explicitly be identified as
$$
N_\nu \;\cong\; S_\nu \x \bigoplus_{\nu' \ne \nu} V_{\nu'}.
$$
Hence by \ref{eqn:torus_euler_class} the $T$-equivariant Euler class of $N_\nu$ is given by
$$
e_T(N_\nu) \;=\; \prod_{w_j \nparallel w_\nu} w_j.
$$
As before we take elements $x_j \in \LT^*$ for $j = 1 , \ldots , m$ and denote the product by
$$
\mathbbm{x} \;:=\; \prod_{j = 1}^m x_j \in S(\LT^*).
$$
We write $\pi_\nu : S_\nu \ra \{ \pt \}$ for the projection of $S_\nu$ onto a point and obtain from the relative localization theorem \ref{thm:relative_localization} the following formula:
\begin{equation} \label{eqn:pi_T_1_localized}
\left( \pi / T_1 \right)_* \left[ \mathbbm{x} \right] \;=\; \sum_{\nu = 1}^N \left( \pi_\nu / T_1 \right)_* \left[ \frac{ \mathbbm{x} }{ e_T \left( N_\nu \right) } \right]
\end{equation}
Now every single summand can be computed by using the technique that was used to prove proposition \ref{prop:T_1_invariant_integration}, because with the $T$-action restricted to $S_\nu \subset V_\nu$ we are precisely in the setting of the example in section \ref{sec:example_invariant_integration}.

\begin{lem}
With the above notation we obtain for every $\nu \in \{ 1 , \ldots , N \}$
$$
\left( \pi_\nu / T_1 \right)_* \left[ \frac{ \mathbbm{x} }{ e_T \left( N_\nu \right) } \right] (\xi) \;=\; \frac{1}{2 \pi i} \oint_\nu \frac{ \prod_{j = 1}^m \langle x_j , \xi + ze_1 \rangle }{ \prod_{j = 1}^n \langle w_j , \xi + ze_1 \rangle } \; \dz \; .
$$
The integral is around a circle in the complex plane enclosing the pole at
$$
z = z_\nu := - \frac{ \langle w_\nu , \xi \rangle }{ \langle w_\nu , e_1 \rangle}
$$
but no other pole of the integrand.
\end{lem}

\begin{proof}
We use the $T$-equivariant $T_1$-connection $A$ from \ref{def:T_equivariant_T_1_connection} with $\ell = \ell_\nu$ and $w = w_\nu$ and obtain analogously to \ref{eqn:xx_A}
\begin{eqnarray} \label{eqn:e_T_A}
e_T \left( N_\nu \right)_A & = & \prod_{w_j \nparallel w_\nu} w_{j,A} \nonumber\\
& = & \prod_{w_j \nparallel w_\nu} \left( w_j - \langle w_j , e_1 \rangle \frac{ \mathrm{d}\alpha_\ell - w_\nu }{ \langle w_\nu , e_1 \rangle } \right).
\end{eqnarray}
In the expansion of this product the term not involving any power of $\mathrm{d}\alpha_\ell$ is the nonzero element
$$
\mathbbm{w} \;:=\; \prod_{w_j \nparallel w_\nu} \left( w_j - \frac{ \langle w_j , e_1 \rangle }{ \langle w_\nu , e_1 \rangle } \cdot w_\nu \right)  \in S( \LT_0^* ).
$$
Hence if we repeat the steps following equation \ref{eqn:euler_class_sum} in the abstract discussion we obtain
\begin{equation} \label{eqn:e_inversion}
\frac{ \mathbbm{w}^n }{ e_T \left( N_\nu \right)_A } \;=\; \beta
\end{equation}
with
$$
\beta \;:=\; \mathbbm{w}^{n-1} + \mathbbm{w}^{n-2} \cdot \Theta + \ldots + \mathbbm{w} \cdot \Theta^{n-1} \quad ; \quad \Theta \;:=\; \mathbbm{w} - e_T \left( N_\nu \right)_A,
$$
because $\left( \mathrm{d}\alpha_\ell \right)^n = 0$. This formula shows that $\beta$ is $T_1$-basic. Now by \ref{eqn:e_T_A} we have  $e_T \left( N_\nu \right)_A$ given as a polynomial in $\frac{ \mathrm{d}\alpha_\ell - w_\nu }{ \langle w_\nu , e_1 \rangle }$. This implies that we can also write
$$
\beta \;=\; \sum_{k = 0}^{n-1} r_k \cdot \left( \frac{ \mathrm{d}\alpha_\ell - w_\nu }{ \langle w_\nu , e_1 \rangle } \right)^k
$$
for some elements $r_k \in S(\LT^*)$. Now since $\mathbbm{w} \in S( \LT_0^* )$ we have $\mathbbm{w}_A = \mathbbm{w}$ and hence
\begin{eqnarray*}
\mathbbm{w}^n \cdot \left( \pi_\nu / T_1 \right)_* \left[ \frac{ \mathbbm{x} }{ e_T \left( N_\nu \right) } \right] & = & \left( \pi_\nu / T_1 \right)_* \left[ \frac{ \mathbbm{x}_A \cdot \mathbbm{w}^n }{ e_T \left( N_\nu \right)_A } \right] \\
& = &  \left( \pi_\nu / T_1 \right)_* \left[ \mathbbm{x}_A \cdot \beta \right].
\end{eqnarray*}
Now the form $\mathbbm{x}_A \cdot \beta$ is precisely of the type that we discussed in remark \ref{rem:z}. So by formula \ref{eqn:gen_T_1_invariant_integration} we obtain
$$
\left( \pi_\nu / T_1 \right)_* \left[ \mathbbm{x}_A \cdot \beta \right] (\xi) \;=\;  \Res_{z = - \frac{ \langle w_\nu , \xi \rangle }{ \langle w_\nu , e_1 \rangle } } \left\{ \frac{ \D \prod_{j = 1}^m \langle x_j , \xi + ze_1 \rangle \cdot \sum_{k = 0}^{n-1} r_k(\xi) \cdot z^k }{ \D \prod_{w_j \parallel w_\nu} \langle  w_j , \xi + ze_1 \rangle } \right\}.
$$
By equations \ref{eqn:e_T_A} and \ref{eqn:e_inversion} and the definition of the coefficients $r_k$ we get the identity
$$
\sum_{k = 0}^{n-1} r_k (\xi) \cdot z^k \;=\; \frac{ \mathbbm{w}^n (\xi) }{ \D \prod_{w_j \nparallel w_\nu} \langle  w_j , \xi + ze_1 \rangle }
$$
for all $z$ away from the zeros of the denominator. These zeros are all different from $z_\nu := - \frac{ \langle w_\nu , \xi \rangle }{ \langle w_\nu , e_1 \rangle }$. So we can plug this identity into the above residue formula to get
\begin{eqnarray*}
\left( \pi_\nu / T_1 \right)_* \left[ \mathbbm{x}_A \cdot \beta \right] (\xi) & = & \Res_{z = z_\nu} \left\{ \frac{ \D \prod_{j = 1}^m \langle x_j , \xi + ze_1 \rangle \cdot \mathbbm{w}^n (\xi) }{ \D \prod_{w_j \parallel w_\nu} \langle  w_j , \xi + ze_1 \rangle \cdot \prod_{w_j \nparallel w_\nu} \langle  w_j , \xi + ze_1 \rangle } \right\} \\
& = & \mathbbm{w}^n (\xi) \cdot \Res_{z = z_\nu } \left\{ \frac{ \prod_{j = 1}^m \langle x_j , \xi + ze_1 \rangle }{ \prod_{j = 1}^n \langle  w_j , \xi + ze_1 \rangle } \right\}.
\end{eqnarray*}
Now for all $\xi$ with $\mathbbm{w} (\xi) \ne 0$ this implies
$$
\left( \pi_\nu / T_1 \right)_* \left[ \frac{ \mathbbm{x} }{ e_T \left( N_\nu \right) } \right] (\xi) \;=\; \Res_{z = z_\nu } \left\{ \frac{ \prod_{j = 1}^m \langle x_j , \xi + ze_1 \rangle }{ \prod_{j = 1}^n \langle  w_j , \xi + ze_1 \rangle } \right\}.
$$
Since $\mathbbm{w} \ne 0$ the set of such $\xi$ is dense in $\LT$ and the result actually holds for all $\xi$ by continuity. The claimed identity finally follows by the residue theorem.
\end{proof}

Note the beauty of this formula: The integrand does not depend on $\nu$ --- only the point $z_\nu$ around which we have to integrate does. Hence we can perform the sum in \ref{eqn:pi_T_1_localized} by just integrating around all $z_\nu$, which also happen to be all poles of the integrand. This proves the following generalization of proposition \ref{prop:T_1_invariant_integration}:

\begin{prop} \label{prop:example_localization}
Suppose the torus $T$ acts on $S^{2n-1} \subset \C^n$ via a collection of $n$ weight vectors $w_j \in \Lambda^* \setminus \{ 0 \}$ as
$$
\exp( \xi ).z \;:=\; \left( e^{- 2 \pi i \langle w_1 , \xi \rangle} \cdot z_1 , \ldots ,  e^{- 2 \pi i \langle w_n , \xi \rangle} \cdot z_n \right)
$$
Assume that no two $w_j$ are negative multiples of each other. Let $e_1 \in \LT$ be such that $\langle w_j , e_1 \rangle \ne 0$ for all $j$. Then for any $\mathbbm{x} \;=\; \prod_{j = 1}^m x_j \in S(\LT^*)$ and $\xi \in \LT$ we have
$$
\left( \pi / T_1 \right)_* \left[ \mathbbm{x} \right] (\xi) \;=\; \frac{1}{2 \pi i} \oint \frac{ \prod_{j = 1}^m \langle x_j , \xi + z e_1 \rangle }{ \prod_{j = 1}^n \langle w_j , \xi + z e_1 \rangle } \; \dz \; .
$$
The integral is around a circle in the complex plane enclosing all poles of the integrand.
\end{prop}

\section{Moduli problems}

\label{chap:moduli_problems}

Cieliebak, Mundet and Salamon \cite{CMS} introduced the notion of a \emph{$G$-moduli problem} and its associated \emph{Euler class}. Roughly speaking a $G$-moduli problem consists of an equivariant Hilbert space bundle $\mcE \ra \mcB$ with an equi\-variant Fredholm section $\mcS$ and the Euler class is the map
$$
\chi^{\mcB,\mcE,\mcS} \;:\; H_G^*(\mcB) \ra \R
$$
obtained by $G$-invariant integration over the zero set $\mcM := \mcS^{-1}(0)$. The point is that this map can be defined even if the section $\mcS$ is not transversal and cannot be made so by an equivariant perturbation.

We review this technique because we will use it frequently later on. In fact we give a generalized construction and build in the results from section \ref{sec:H_invariant_push_forward} on $H$-invariant push-forward to define a relative $G/H$-equivariant Euler class for $H$-regular $G$-moduli problems. Furthermore we can study fibered moduli problems, which for instance appear in the discussion of the wall crossing formula in section \ref{sec:toric_wall_crossing}.

\subsection{Definitions}

Throughout $G$ denotes a compact oriented Lie group.

\begin{dfn}
A \emph{$G$-moduli problem} is a triple $(\mcB,\mcE,\mcS)$ with the following properties:
\begin{itemize}
\item $\mcB$ is a Hilbert manifold with a smooth $G$-action.
\item $\mcE$ is a Hilbert space bundle over $\mcB$, equipped with a smooth $G$-action such that the projection $\pi : \mcE \ra \mcB$ is $G$-equivariant and the induced action between fibers of $\mcE$ is by isometries.
\item $\mcS : \mcB \ra \mcE$ is a smooth $G$-equivariant Fredholm section of constant Fredholm index. The determinant bundle $\det(\mcS) \ra \mcB$ is oriented, $G$ acts by orientation preserving isometries on $\det(\mcS)$, and the zero set
$$
\mcM \;:=\; \left\{ x \in \mcB \;|\; \mcS(x) = 0 \right\}
$$
is compact.
\end{itemize}
A $G$-moduli problem with finite-dimensional spaces $\mcB$ and $\mcE$ is called \emph{oriented} if $\mcB$ and $\mcE$ are oriented and the $G$-actions preserve orientations. Given a closed subgroup $H \subset G$ a $G$-moduli problem is called \emph{$H$-regular} if the induced $H$-action on $\mcM$ is regular. A $G$-regular $G$-moduli problem is simply called \emph{regular}.
\end{dfn}

Given any trivialization of $\mcE$ around a point $x \in \mcB$ one obtains the \emph{vertical differential} $\mcD_x$ by composing the differential $\mathrm{d}\mcS(x) : T_x\mcB \ra T_{\mcS(x)}\mcE$ of $\mcS$ with the projection $T_{\mcS(x)}\mcE \ra \mcE_x$ onto the fiber $\mcE_x$ of $\mcE$ over $x$:
$$
\mcD_x \;:\;  T_x\mcB \ra \mcE_x
$$
The Fredholm property of $\mcS$ asserts that $\mcD_x$ is a Fredholm operator for all $x$ in a small neighbourhood of $\mcM$ and that the index of $\mcD_x$ is independent of $x$. The \emph{index} of a $G$-moduli problem is defined to be
$$
\Ind(\mcS) \;:=\; \Ind(\mcD_x) - \dim(G).
$$

\begin{dfn} \label{def:morphism}
A \emph{morphism} between two $G$-moduli problems $(\mcB,\mcE,\mcS)$ and $(\mcB',\mcE',\mcS')$ is a pair $(f,F)$ with the following properties.
\begin{itemize}
\item $f : \mcB_0 \ra \mcB'$ is a smooth $G$-equivariant embedding of a neighbourhood $\mcB_0 \subset \mcB$ of $\mcM$ into $\mcB'$.
\item $F : \mcE|_{\mcB_0} \ra \mcE'$ is a smooth and injective bundle homomorphism over $f$.
\item The sections $\mcS$ and $\mcS'$ satisfy
$$
\mcS' \circ f \;=\; F \circ \mcS \quad \mathrm{and} \quad \mcM' \;=\; f(\mcM).
$$
\item For all $x \in M$ the linearized operator $\mathrm{d}f(x) : T_x\mcB \ra T_{f(x)}\mcB'$ and the linear operator $F_x : \mcE_x \ra \mcE'_{f(x)}$ induce isomorphisms
$$
\mathrm{d}f(x) \;:\; \ker \mcD_x \ra \ker \mcD'_{f(x)} \quad \mathrm{and} \quad F_x \;:\; \coker \mcD_x \ra \coker \mcD'_{f(x)}
$$
and the resulting isomorphism from $\det(\mcS)$ to $\det(\mcS')$ is orientation preserving.
\end{itemize}
\end{dfn}

To any regular $G$-moduli problem there exists its \emph{Euler class}, which is a homomorphism
$$
\chi^{\mcB,\mcE,\mcS} \;:\; H_G^*(\mcB) \ra \R.
$$
As shown in \cite{CMS} the Euler class is uniquely determined by the following two properties.
\begin{description}
\item[(Functoriality)] If $(f,F)$ is a morphism from $(\mcB,\mcE,\mcS)$ to $(\mcB',\mcE',\mcS')$ then
$$
\chi^{\mcB,\mcE,\mcS} \circ f_G^* \;=\; \chi^{\mcB',\mcE',\mcS'}.
$$
\item[(Thom class)] If $(B,E,S)$ is a finite-dimensional, oriented, regular $G$-moduli problem and $\tau \in \Omega_G^*(E)$ is an equivariant Thom form supported in an open neighbourhood $U \subset E$ of the zero section such that $U \cap E_x$ is convex for every $x \in B$, $U \cap \pi^{-1}(K)$ has compact closure for every compact set $K \subset B$, and $S^{-1}(U)$ has compact closure, then
$$
\chi^{\mcB,\mcE,\mcS} (\alpha) \;=\; \int_{B/G} \alpha \wedge S^*\tau
$$
for every $\alpha \in H_G^*(B)$.
\end{description}
The integral over $B/G$ is $G$-invariant integration as explained in section \ref{chap:invariant_integration}. We will review the definition of $\chi^{\mcB,\mcE,\mcS}$ in the next section, where we explain the generalization to $H$-regular problems.

\begin{dfn} \label{def:homotopy}
A \emph{homotopy of regular $G$-moduli problems} is a $G$-equi\-va\-ri\-ant Hilbert space bundle $\mcE \ra [0,1] \x \mcB$ and a smooth $G$-equivariant section $\mcS$ therein, such that $( \mcB , \mcE|_{\{t\}\x\mcB} , \mcS|_{\{t\}\x\mcB} )$ is a regular $G$-moduli problem for every $t \in [0,1]$, and the set
$$
\mcM \;:=\; \left\{ (t,x) \in [0,1] \x \mcB \;|\; \mcS(t,x) = 0 \right\}
$$
is compact.
\end{dfn}

A homotopy is an example for the more general notion of a \emph{cobordism} of $G$-moduli problems, which is defined in the obvious way:

\begin{dfn} \label{def:cobordism}
Two regular $G$-moduli problems $\left( \mcB_i , \mcE_i , \mcS_i \right)$, $i \in \{0,1\}$, are called \emph{cobordant} if there exist a $G$-equivariant Hilbert space bundle $\widetilde{\mcE} \ra \widetilde{\mcB}$ over a Hilbert manifold $\widetilde{\mcB}$ with boundary and a smooth oriented $G$-equivariant Fredholm section $\widetilde{\mcS} : \widetilde{\mcB} \ra \widetilde{\mcE}$ such that the zero set $\widetilde{\mcM} := \widetilde{\mcS}^{-1}(0)$ is compact, $G$ acts with finite isotropy on $\widetilde{\mcB}$, and
$$
\del \widetilde{\mcB} \;=\; \mcB_0 \cup \mcB_1 \quad , \quad \mcE_i \;=\; \widetilde{\mcE}|_{\mcB_i} \quad , \quad \mcS_i \;=\; \widetilde{\mcS}|_{\mcB_i}.
$$
Moreover, $\det(\widetilde{\mcS})$ carries an orientation which induces the orientation of $\det(\mcS_1)$ over $\mcB_1$ and the opposite of the orientation of $\det(\mcS_0)$ over $\mcB_0$. Here an orientation of $\det(\widetilde{\mcS})$ induces an orientation of the determinant bundle of $\mcS := \widetilde{\mcS}|_{\del \widetilde{\mcB}}$ via the natural isomorphism $\det(\widetilde{\mcS})|_{\del \widetilde{\mcB}} \cong \R v \otimes \det(\mcS)$ for an outward pointing normal vector field $v$ along $\del \widetilde{\mcB}$.
\end{dfn}

It is shown in \cite{CMS} that the Euler class satisfies the following property:
\begin{description}
\item[(Cobordism)] If $\left( \mcB_0 , \mcE_0 , \mcS_0 \right)$ and $\left( \mcB_1 , \mcE_1 , \mcS_1 \right)$ are cobordant $G$-moduli problems then
$$
\chi^{\mcB_0,\mcE_0,\mcS_0}(\iota_0^*\alpha) \;=\; \chi^{\mcB_1,\mcE_1,\mcS_1}(\iota_1^*\alpha)
$$
for every $\alpha \in H_G^*(\widetilde{\mcB})$, where $\iota_0 : \mcB_0 \ra \widetilde{\mcB}$ and $\iota_1 : \mcB_1 \ra \widetilde{\mcB}$ are the inclusions.
\end{description}

\subsection{The $G/H$-equivariant Euler class}

Suppose $(\mcB,\mcE,\mcS)$ is an $H$-regular $G$-moduli problem for some closed normal subgroup $H \lhd G$. We denote the quotient group by $K := G / H$ and its Lie algebra by $\LK$. In this section we will define the $G/H$-equivariant Euler class
$$
\chi^{\mcB,\mcE,\mcS}_{G,H} \;:\; H_G^*(\mcB) \ra H_K^*(\pt) = S(\LK^*).
$$
This is a generalization of the Euler class $\chi^{\mcB,\mcE,\mcS}$ from \cite{CMS} because for a $G$-regular $G$-moduli problem we will have
$$
\chi^{\mcB,\mcE,\mcS}_{G,G} \;=\; \chi^{\mcB,\mcE,\mcS}.
$$
The construction is analogous to the old one. The procedure of finite-dimensional reduction does not depend on the involved isotropy groups. Hence it suffices to define $\chi^{\mcB,\mcE,\mcS}_{G,H}$ for finite-dimensional, oriented, $H$-regular $G$-moduli problems.

A \emph{Thom structure} for such a problem is a pair $(U,\tau)$ with the following properties:
\begin{enumerate}
\item $U \subset \mcE$ is a $G$-invariant open neighbourhood of the zero section that intersects each fiber in a convex set and such that $U \cap \mcE|_K$ has compact closure for every compact subset $K \subset \mcB$.
\item $\mcS^{-1}(U)$ has compact clusure.
\item $\tau \in \Omega_G^*(\mcE)$ is an equivariant Thom form for the zero section in $\mcE$ with support in $U$.
\end{enumerate}
By the results cited in section \ref{sec:equivariant_thom_forms} it is clear that for every finite-dimensional, oriented $G$-moduli problem $(\mcB,\mcE,\mcS)$ and for every subset $U \subset \mcE$ that satisfies the first two points, there exists a form $\tau$ such that $(U,\tau)$ is a Thom structure. This also does not involve the isotropy of the $G$-action.

Now since $\mcS^{-1}(0) = \mcM$ is compact we find an open neighbourhood $\mcB_0$ of $\mcM$ in $\mcB$ that has compact closure and such that $H$ acts with finite isotropy on all of $\mcB_0$. Next we choose $U \subset \mcE$ small enough so that $\mcS^{-1}(U)$ lies in $\mcB_0$ and such that $U$ satisfies property (1). We then find a corresponding Thom form $\tau$. For a class $[ \alpha ] \in H_G^*(\mcB)$ we finally define
$$
\chi^{\mcB,\mcE,\mcS}_{G,H} ( [ \alpha ] ) \;:=\; \int_{\mcB/H} \alpha \wedge \mcS^* \tau \in S(\LK^*).
$$
Integration over $\mcB/H$ is $H$-invariant push-forward of the map $\mcB \ra \{\pt\}$ as defined in section \ref{sec:H_invariant_push_forward}. It is well defined because by construction the form $\mcS^*\tau$ has compact support. Independence of all involved choices follows as in \cite[Chapter 8]{CMS} if we use the fact that $G$-equivariant Thom forms behave well for $H$-invariant integration as shown in proposition \ref{prop:H_invariant_Thom_integration}.

\begin{rem}
There is one open question about the $G/H$-equivariant Euler class. The Euler class $\chi^{\mcB,\mcE,\mcS}$ from \cite{CMS} is rational, i.~e.~it satisfies the following property:
\begin{description}
\item[(Rationality)] If $\alpha \in H_G^*(\mcB;\Q)$ then $\chi^{\mcB,\mcE,\mcS} (\alpha) \in \Q$.
\end{description}
Does the corresponding property also hold for $\chi^{\mcB,\mcE,\mcS}_{G,H}$? One would need a generalization of the technique of representing rational classes by weighted branched submanifolds (see \cite[Chapter 9 and 10]{CMS}) to the case of non-regular $G$-actions. Such a theory would be interesting on its own account.
\end{rem}

\subsection{Fibered moduli problems}

\label{sec:fibred_G_moduli_problems}

Suppose the base space $\mcB$ of a moduli problem is a fiber bundle
$$
f \;:\; \mcB \ra \mcB_0.
$$
Furthermore assume that $f$ is $G$-equivariant, that $H$ acts trivially in $\mcB_0$, and that we have a $K$-moduli problem $\left( \mcB_0 , \mcE_0 , \mcS_0 \right)$ with $K := G/H$ and some closed normal subgroup $H \lhd G$. Now consider the pull-back
$$
\left(\,\mcB \,,\, \mcE := f^*\mcE_0 \,,\, \mcS := f^*\mcS_0 \,\right).
$$
If the fibers of $f$ are compact then the moduli space $\mcM := \mcS^{-1}(0)$ will be compact and $\left( \mcB , \mcE , \mcS \right)$ in fact is a $G$-moduli problem. This is the simplest instance of a \emph{fibered moduli problem}. Since this is also the only form in which we will actually apply the notion of a fibered moduli problem, we restrict to this case. We comment on more general settings below.

\begin{prop} \label{prop:fibred_Euler_class}
In the above setup suppose that $\left( \mcB , \mcE , \mcS \right)$ is $H$-regular. Then we have the identity
$$
\chi^{\mcB,\mcE,\mcS}_{G,H} \;=\; \chi^{\mcB_0,\mcE_0,\mcS_0}_{K,\{e\}} \circ \left( f / H \right)_* \;:\; H_G^*(\mcB) \ra H_K^*(\pt) = S(\LK^*).
$$
If in addition $\left( \mcB_0 , \mcE_0 , \mcS_0 \right)$ is $K$-regular, then $\left( \mcB , \mcE , \mcS \right)$ is $G$-regular and we have
$$
\chi^{\mcB,\mcE,\mcS}_{G,G} \;=\; \chi^{\mcB_0,\mcE_0,\mcS_0}_{K,K} \circ \left( f / H \right)_* \;:\; H_G^*(\mcB) \ra \R.
$$
\end{prop}

\begin{proof}
The first step is to show that we can choose finite-dimensional reductions that feature the same fibered picture. We omit this discussion and deal with finite-dimensional problems only. We will only apply this result in settings that are already finite-dimensional.

We choose a Thom structure $(U_0,\tau_0)$ for $\left( \mcB_0 , \mcE_0 , \mcS_0 \right)$. Then
$$
\left( U , \tau \right) \;:=\; \left( f^{-1}U_0 , f^*\tau_0 \right)
$$
is a Thom structure for $\left( \mcB , \mcE , \mcS \right)$. By property $(5)$ in proposition \ref{prop:invariant_integration} we have
\begin{eqnarray*}
\left( f / H \right)_* \left( \alpha \wedge S^*\tau \right) & = & \left( f / H \right)_* \left( \alpha \wedge S^*f^*\tau_0 \right) \\
& = & \left( f / H \right)_* \left( \alpha \wedge f^*{S_0}^*\tau_0 \right) \\
& = & \left( f / H \right)_* \alpha \wedge {S_0}^*\tau_0.
\end{eqnarray*}
Hence the claimed identities follow by the functoriality properties $(4)$ in proposition \ref{prop:invariant_integration} and the definition of the Euler classes.
\end{proof}

\begin{rem}
Another instance of a fibered moduli problem would be a \emph{parametrized} problem, that is a triple $\left( \mcB , \mcE , \mcS \right)$ with an equivariant fibration $f : \mcB \ra \mcB_0$ such that $H$ acts trivially on $\mcB_0$ and the restriction of $\left( \mcB , \mcE , \mcS \right)$ to any fiber of $f$ is an $H$-regular $G$-moduli problem. For such a parametrized problem to be a moduli problem one needs the total moduli space $\mcM$ to be compact. It does not suffice if only the base $\mcB_0$ is compact. A homotopy of moduli problems would be an example. But to really get computational simplifications in the spirit of proposition \ref{prop:fibred_Euler_class} from such a fibered setting, one would also need a good behaviour of the section $\mcS$ with respect to the fibration. The section $\mcS$ should split into components that either depend only on the point in the base or only on the point in the fiber.

The moduli problem associated to the vortex equations (see section \ref{chap:vortex_invariants}) has the property that $\mcB$ is a fiber bundle. The base $\mcB_0$ is the space of connections $\mcA(P)$ on a principal bundle $P$ modulo based gauge transformations $\mcG_0(P)$. The quotient group $G = \mcG(P) / \mcG_0(P)$ of all gauge transformations by the based ones acts trivially on connections, but locally freely on $\mcB$. But the section $\mcS$ does not behave well with respect to this fibration. One can interprete our deformation of the vortex equations in section \ref{sec:vortex_deformation} as the attempt to improve this situation and exploit the fiber form of $\mcB$.
\end{rem}

\section{Toric manifolds}

\label{chap:toric_manifolds}

We present toric manifolds as symplectic reduction of complex linear torus actions but adapt the notation to our purposes. In particular we want to have a setup in which we can easily change the multiplicities of the characters of a given torus action --- because this is what happens in section \ref{chap:vortex_invariants} when we look at the moduli space of the vortex equations associated to such a torus action.

Toric manifolds are the basic model for moduli problems as presented in section \ref{chap:moduli_problems}. The associated Euler class is the integral of elements in the image of the Kirwan map over the toric manifold. Knowing these integrals suffices to determine the whole cohomology ring of the toric manifold. In section \ref{sec:toric_wall_crossing} we study the effect it has on this Euler class if we change the chamber of the regular value $\tau$ at which we reduce. With the methods on invariant integration at hand we can derive an explicit formula that allows to compute these wall crossing numbers. Then by a sequence of wall crossings that connects the chamber of $\tau$ with the chamber outside of the image of the moment map we can completely determine the Euler class.

This wall crossing strategy to evaluate integrals over symplectic quotients is not new. It is used by Guillemin and Kalkman in \cite{GK} and Martin in \cite{Mar} to derive a formula for such integrals. Because of the results in \cite{JK} such formulae are generally referred to as \emph{Jeffrey-Kirwan localization}. We discuss the relation of our work with general Jeffrey-Kirwan localization at the end of this section.

\subsection{Torus actions on Hermitian vector spaces}

\label{sec:torus_actions_on_hermitian_vector_spaces}

Let $V$ be a Hermitian vector space with Hermitian form given by
$$
\left( v , w \right) \;=\; g \left( v , w \right) - i \; \omega \left( v , w \right)
$$
with inner product $g$ (given by the real part $\Re$ of the Hermitian form) and symplectic form $\omega$ (given by minus the imaginary part $\Im$ of the Hermitian form). The induced norm is denoted by $\left| v \right|^2 = g \left( v , v \right)$. Let $T$ be a $k$-dimensional torus with Lie algebra $\LT$. We denote the integral lattice in $\LT$ by
$$
\Lambda \;=\; \left\{ \xi \in \LT \;|\; \exp \left( \xi \right) = \unit \in T \right\}
$$
and its dual by
$$
\Lambda^* \;=\; \left\{ w \in \LT^* \;|\; \forall \xi \in \Lambda : \left\langle w,\xi \right\rangle \in \Z \right\}.
$$
As before $\langle \;,\; \rangle$ denotes the pairing between $\LT$ and its dual space $\LT^*$. Let $\rho \;:\; T \ra \mathrm{U}(1)$ be a character given by
$$
\exp \left( \xi \right) \mto e^{-2 \pi i \langle w , \xi \rangle}
$$
for some nonzero weight vector $w \in \Lambda^* \setminus \{ 0 \}$. Now $\mathrm{U}(1)$ acts by scalar multiplication on the complex vector space $V$. Hence $\rho$ induces an action of the torus $T$ on $V$. This action is Hamiltonian with respect to the symplectic form $\omega$ on $V$. A moment map for this action is given by
\begin{align*}
\mu \;:\; V & \;\ra\; \quad \LT^* \\
v & \;\mto\; \pi \left| v \right|^2 \cdot w.
\end{align*}
By definition a moment map satisfies the equation
$$
\mathrm{d} \langle \mu , \xi \rangle = - \iota_\xi \omega
$$
for all $\xi \in \LT$. The sign-conventions are such that $\omega - \mu$ is a $\mathrm{d}_T$-closed equivariant differential form.

We now consider a finite collection of characters $\rho_\nu$ given by weights $w_\nu \in \Lambda^*$ for $\nu=1,\ldots,N$. Given a corresponding collection of Hermitian vector spaces $V_\nu$ we get a diagonal action $\rho$ of the torus $T$ on the direct sum $V = \bigoplus_{\nu=1}^N V_\nu$. In fact every complex linear torus action decomposes in such a way. This action is again Hamiltonian and a moment map is given by the sum of the moment maps $\mu_\nu$ for the single $\rho_\nu$-induced $T$-actions on the $V_\nu$:
\begin{align*}
\mu = \sum_{\nu=1}^N \mu_\nu \;:\; V = \bigoplus_{\nu=1}^N V_\nu & \;\ra\; \quad \LT^* \\
\left( v_\nu \right) \quad & \;\mto\; \pi \sum_{\nu=1}^N \left| v_\nu \right|^2 \cdot w_\nu
\end{align*}

\begin{dfn} \label{def:proper}
A collection of characters $\rho_\nu$ is called \emph{proper} if there exists an element $\xi \in \LT$ such that $\left\langle w_\nu , \xi \right\rangle < 0$ for all $\nu \in \left\{ 1 , \ldots , N \right\}$, i.e.~all weight vectors are contained in the same connected component of $\LT^*$ with one hyperplane through the origin removed.
\end{dfn}

This notion is motivated by the following observation. The moment map $\mu$ associated to a collection of finite-dimensional vector spaces $V_\nu$ is proper (in the usual sense) if and only if the collection of characters $\rho_\nu$ is proper (in the above sense). Hence properness of the moment map is preserved if we change the multiplicities of the characters in the torus action, i.e.~the dimensions of the spaces $V_\nu$. We can even allow $\dim V_\nu = 0$. The following notation is taken from Guillemin, Ginzburg and Karshon \cite{GGK}.

\begin{dfn}
Given a collection of characters $\rho_\nu$ an element $\tau \in \LT^*$ is called \emph{regular} (for this collection) if the following holds: If $\tau = \sum_{j \in J} a_j w_j$ for some positive coefficients $a_j > 0$ and some index set $J \subset \left\{ 1 , \ldots , N \right\}$, then the weight vectors $\left( w_j \right)_{j \in J}$ span $\LT^*$. An element $\tau \in \LT^*$ is called \emph{super-regular} if the following holds: If $\tau = \sum_{j \in J} a_j w_j$ for some positive coefficients $a_j > 0$ and some index set $J \subset \left\{ 1,\ldots,N \right\}$, then the weight vectors $\left( w_j \right)_{j \in J}$ generate $\Lambda^*$ over $\Z$.
\end{dfn}

Here we observe that an element $\tau \in \LT^*$ is a regular value of the moment map $\mu$ if and only if $\tau$ is regular in the above sense. Again this does not depend on the spaces $V_\nu$. A super-regular element $\tau$ is in particular regular but has even nicer properties. The moment map is $T$-invariant, hence we get induced $T$-actions on every level set $\mu^{-1}(\tau)$. Recall that we call an action \emph{regular} if all isotropy subgroups are finite.

\begin{lem}
If $\tau \in \LT^*$ is regular, then the $T$-action on $\mu^{-1}(\tau)$ is regular. If $\tau$ is super-regular, then $T$ acts freely on $\mu^{-1}(\tau)$.
\end{lem}

\begin{proof}
Let $\tau$ be regular and $v \in \mu^{-1}(\tau)$. Hence $\tau = \pi \sum_{\nu = 1}^N \left| v_\nu \right|^2 \cdot w_\nu$ and the collection $(w_j)_{j \in J}$ with $J = \left\{ \left. j \in \left\{ 1 , \ldots , N \right\} \;\right|\; v_j \ne 0 \right\}$ spans $\LT^*$. So the $1$-form $\left\langle \mathrm{d}\mu(v) , \xi \right\rangle = 2 \pi \sum_{\nu = 1}^N g( v_\nu , \cdot ) \, \langle w_\nu , \xi \rangle$ can only vanish for $\xi = 0$. Now since $\mu$ is a moment map $\langle \mathrm{d}\mu(v) , \xi \rangle = - \iota_\xi \omega$ and this implies that $X_\xi$ can only vanish for $\xi = 0$. Thus the torus acts with finite isotropy on $v$.

If $\exp(\xi) \in T$ fixes a point $v \in \mu^{-1}(\tau)$ then for all $j$ with $v_j \ne 0$ we must have $\langle w_j , \xi \rangle \in \Z$. If $\tau$ is super-regular this implies $\langle w , \xi \rangle \in \Z$ for all $w \in \Lambda^*$, hence $\xi \in \Lambda$ and $\exp(\xi) = \unit$.
\end{proof}

So if we have a collection of characters $\left( \rho_\nu \right)$ and a super-regular element $\tau \in \LT^*$ we can associate a \emph{toric manifold}
$$
X_{V,\tau} = \mu^{-1}(\tau) / T
$$
to any collection of spaces $V_\nu$. If $\tau$ is only regular we get orbifold singularities in the quotient. This quotient is compact if and only if $\left( \rho_\nu \right)$ is a proper collection. Smoothness and compactness do not depend on the choice of Hermitian vector spaces $V_\nu$. They can even be zero-dimensional. If we denote by $n_\nu$ the complex dimension of $V_\nu$ the real dimension of $X_{V,\tau}$ is given by
$$
\dim X_{V,\tau} = 2n - 2k \quad \mathrm{with} \quad n := \sum_{\nu = 1}^N n_\nu.
$$
The notation $X_{V,\tau}$ indicates that we think of the characters $\rho_\nu$ and hence the weight vectors $w_\nu \in \LT^*$ to be fixed, whereas the Hermitian vector spaces $V_\nu$ and the level $\tau$ can vary.

\begin{rem}
If the collection of weight vectors $w_\nu$ does not span the whole $\LT^*$ then there are no regular elements $\tau$ in the image of the moment map $\mu$. Hence any associated toric manifold $X_{V,\tau}$ would be empty. So we can restrict to torus actions that are given by collections of weight vectors that do span $\LT^*$.
\end{rem}

\subsection{Toric manifolds as moduli problems}

\label{sec:toric_manifolds_as_moduli_problems}

Consider a diagonal torus action given by a proper collection $\rho_\nu$ and a regular level $\tau \in \LT^*$ as above. We do not require super-regularity, so the quotient $X_{V,\tau}$ can have singularities. Using the notation from section \ref{chap:moduli_problems} this setup gives rise to a finite-dimensional $T$-moduli problem with base $\mcB = V$, bundle $\mcE = V \x \LT^*$ and section $\mcS = \mu - \tau$. Regularity of $\tau$ implies regularity for the moduli problem and transversality of $\mcS$ to the zero section. The complex orientation on $V$ and any choice of orientation on the torus $T$ gives an orientation of the finite-dimensional $T$-moduli problem $\left( \mcB , \mcE , \mcS \right)$ and hence we get an associated Euler class
$$
\chi^{V,\tau} : S^m \left( \LT^* \right) \ra \R
$$
with
$$
m = \frac{1}{2} \cdot \dim X_{V,\tau} = n - k.
$$
This Euler class is defined as follows. There is a natural identification of $S \left( \LT^* \right)$ with the equivariant cohomology $H_T^*(V)$. Now an equivariant class can be restricted to the $T$-invariant submanifold $\mu^{-1}(\tau)$ to get an element in $H_T^*( \mu^{-1}(\tau) )$. By regularity of $\tau$ the torus action on $\mu^{-1}(\tau)$ has at most finite isotropy. Hence $T$-invariant integration can be applied to obtain a real number. Here we again use the orientation on $T$ that was chosen above. Formally
$$
\chi^{V,\tau}(\alpha) = \int_{X_{V,\tau}} i_\tau^*(\alpha)
$$
with the inclusion $i_\tau : \mu^{-1}(\tau) \ra V$ and the induced pull-back $i_\tau^*$ on equivariant cohomology, and with integration over $X_{V,\tau}$ understood as $T$-invariant integration over $\mu^{-1}(0)$. The map
$$
i_\tau^* : H_T^*( V ) \ra H_T^*( \mu^{-1}(\tau) )
$$
is known as the \emph{Kirwan map} and due to Kirwan \cite{Ki} is the following result:

\begin{thm}[F.~Kirwan] \label{thm:Kirwan}
The Kirwan map is surjective.
\end{thm}

Hence any integral of a cohomology class over $X_{V,\tau}$ can be expressed in terms of this Euler class $\chi^{V,\tau}$.

\subsection{Orientation of toric manifolds}

\label{sec:orientation_of_toric_manifolds}

We digress a little to explain the orientation of a toric manifold in more detail. This is useful for later purposes because this is the finite-dimensional model case for the orientation of the vortex moduli space that we will discuss in section \ref{sec:orientations}.

There is a natural orientation on $X_{V,\tau} = \mu^{-1}(\tau) / T$ that does not depend on the choice of an orientation for $T$. Instead we fix the following complex structure on the space $\LT^* \oplus \LT^*$:
$$
\LT^* \oplus \LT^* \ra \LT^* \oplus \LT^* \; ; \; ( a , b ) \mto ( -b , a ).
$$
We write
$$
L_v : \LT \ra T_vV \; ; \; \xi \mto X_\xi(v)
$$
for the infinitesimal action at the point $v \in V$ and consider the adjoint map $L_v^* : T_vV \ra \LT^*$, where we identify $T_vV$ with its dual via the fixed metric on $V$. Given a point $[v] \in X_{V,\tau}$ we can now identify the tangent space $T_{[v]}X_{V,\tau}$ with the kernel of the map
$$
D_v := \left( \mathrm{d}\mu(v) , L_v^* \right) : T_vV \ra \LT^* \oplus \LT^*.
$$
The first component restricts to the tangent space along the level set of the moment map. The second component fixes a complement to the tangent space along the $T$-orbit. If we identify $V \cong T_vV$ a short computation shows that
$$
D_v ( z ) = \left( 2 \sum_{\nu = 1}^N \Re ( z_\nu , v_\nu ) \cdot w_\nu \; , \; 2 \sum_{\nu = 1}^N \Im ( z_\nu , v_\nu ) \cdot w_\nu \right).
$$
This shows that $D_v$ is complex linear (with respect to the natural complex structure on $V$ and the above chosen complex structure on $\LT^* \oplus \LT^*$). And $D_v$ is even surjective if $\tau$ is regular. In any case the complex orientation on source and target of $D_v$ define an orientation on its kernel and hence define an orientation on $X_{V,\tau}$.

This is our preferred way to define the orientation of a toric manifold, because it extends to an infinite-dimensional setting: The determinant of a complex linear Fredholm operator between complex Banach spaces admits a natural orientation. But if we want to understand $X_{V,\tau}$ as the oriented moduli space of the finite-dimensional $T$-moduli problem $\left( \mcB , \mcE , \mcS \right)$ then we need to specify an orientation of the torus. In fact, if we denote the vertical differential at a point $x \in \mcB$ by $\mcD : T_x\mcB \ra \LT^*$ then we obtain
\begin{eqnarray*}
\det(\mcD) & = & \Lambda^{\mathrm{max}} \left( \ker\,\mcD \right) \otimes \Lambda^{\mathrm{max}} \left( \coker\,\mcD \right)\\
& = & \Lambda^{\mathrm{max}} \left( \ker\,\mcD \right) \otimes \Lambda^{\mathrm{max}} \left( \frac{\LT^*}{\im\,\mcD} \right)\\
& \cong & \Lambda^{\mathrm{max}} \left( \ker\,\mcD \right) \otimes \Lambda^{\mathrm{max}} \left( \im\,\mcD \right) \otimes \Lambda^{\mathrm{max}} \left( \im\,\mcD \right) \otimes \Lambda^{\mathrm{max}} \left( \frac{\LT^*}{\im\,\mcD} \right)\\
& \cong & \Lambda^{\mathrm{max}} \left( \ker\,\mcD \right) \otimes \Lambda^{\mathrm{max}} \left( \frac{T_x\mcB}{\ker\,\mcD} \right) \otimes \Lambda^{\mathrm{max}} \left( \im\,\mcD \right) \otimes \Lambda^{\mathrm{max}} \left( \frac{\LT^*}{\im\,\mcD} \right)\\
& \cong & \Lambda^{\mathrm{max}} \left( T_x\mcB \right) \otimes \Lambda^{\mathrm{max}} \left( \LT^* \right).
\end{eqnarray*}
In the third line we use the canonical identification $\Lambda^{\mathrm{max}} (W) \otimes \Lambda^{\mathrm{max}} (W) \cong \R$ for any finite-dimensional vector space $W$. And in the last line we use the canonical identification $\Lambda^{\mathrm{max}} (W) \cong \Lambda^{\mathrm{max}} (U) \otimes \Lambda^{\mathrm{max}} \left( \frac{W}{U} \right)$ for any linear subspace $U \subset W$.

Now an orientation of $T$ is the same as an orientation of $\LT$. And if we identify $\LT$ with its dual via the choice of an inner product then this also fixes the orientation of $\LT^*$. The choice of the inner product does not affect the result, because the space of inner products in contractible. Hence the orientation of the determinant bundle of the $T$-moduli problem as well as the definition of the Euler class (because of the invariant integration that needs a specified orientation of the torus) depend on the choice of orientation for $T$. So if we both times use the same orientation then the resulting homomorphism $\chi^{V,\tau}$ does not depend on this choice. But it does depend on the conventions that we use for the ordering of basis vectors of the involved spaces.

A careful look at the construction reveals that we reproduce the preferred complex orientation from above via the following conventions: Take an oriented basis of $T_vV$ such that the image under $\mathrm{d}\mu(v)$ of the first $k$ vectors is an oriented basis of $\LT^*$ and the remaining vectors lie in the kernel of $\mathrm{d}\mu(v)$. The orientation of the latter determines the orientation of $T_v \left( \mu^{-1}(\tau) \right)$. Now take a local slice $(U_v,\varphi_v,T_v)$ at $v$ for the $T$-action on $\mu^{-1}(\tau)$ (see section \ref{sec:G_invariant_integration} for the notation) using the product orientation on $T \x_{T_v} U_v$ defined by first taking an oriented basis of $\LT$ and then the standard orientation on $U_v \subset \R^{2m}$.

\subsection{Wall crossing}

\label{sec:toric_wall_crossing}

Let $\left( \rho_\nu \right)_{\nu = 1 , \ldots , N}$ be a proper collection of characters with weight vectors $w_\nu \in \Lambda^*$ that span $\LT^*$. Fix a Hermitian vector space $V = \bigoplus_{\nu = 1}^N V_\nu$. For an index set $I \subset \left\{ 1 , \ldots , N \right\}$ we define the cone
$$
W_I = \left\{ \left. \sum_{i \in I} c_i w_i \in \LT^* \;\right|\; c_i \ge 0 \right\}.
$$

\begin{dfn} \label{def:wall}
A cone $W_I$ is called a \emph{wall} if
\begin{itemize}
\item $W_I \subset \LT^*$ has codimension one, i.e.~the family $(w_i)_{i \in I}$ has rank $(k-1)$.
\item The index set $I$ is complete, i.e.~it contains all indices $i$ with $w_i \in W_I$.
\end{itemize}
\end{dfn}

The set of non-regular elements in $\LT^*$ is precisely given by the union of all walls. The walls divide $\LT^*$ into open connected components of regular elements, called \emph{chambers}. If one element in a chamber is super-regular then super-regularity holds for all elements in that chamber.

Let $\tau_0 \in \LT^*$ be an element in a wall $W_I$. We assume that any other cone $W_J$ containing $\tau_0$ is completely contained in $W_I$, so that $\tau_0$ does not lie in the intersection of two different walls. Hence $\tau_0$ is contained in the closure of exactly two regular chambers and we want to compute the difference between the Euler classes $\chi^{V,\tau}$ for values of $\tau$ in either of them.

\subsubsection*{A cobordism}

Pick an element $\eta\in\LT^*$ transverse to $W_I$ and define
$$
\tau_t \;:=\; \tau_0 + t\eta.
$$
Fix a small enough $\eps > 0$ such that $\tau_t$ is regular for all $t \in \left[ -\eps , \eps \right] \setminus \left\{ 0 \right\}$. Now the set $\left\{ \left( v , t \right) \in V \x \left[ -\eps , \eps \right] \;|\; \mu(v) = \tau_t \right\}$ provides a smooth compact cobordism between the manifolds $\mu^{-1} \left( \tau_{-\eps} \right)$ and $\mu^{-1} \left( \tau_{\eps} \right)$. If the $T$ action on (the $V$ factor of) this cobordism was regular for all times $t$ this would even establish a cobordism of the $T$-moduli problems associated to $\tau_{-\eps}$ and $\tau_{\eps}$ and the Euler classes $\chi^{V,\tau_{\pm\eps}}$ would coincide. But regularity fails because of the wall crossing at $t=0$. So we have to cut out a neighbourhood of the locus with singular action and compute the invariant integral over the newly created boundary component.

To describe the situation around the singular locus we need some preparation. Since $W_I$ has codimension one we can fix a primitive $e_1 \in \Lambda$ such that
\begin{eqnarray*}
\left\langle w_i , e_1 \right\rangle & = & 0 \quad \mathrm{for\ all} \; i \in I,\; \mathrm{and} \\
\left\langle \eta , e_1 \right\rangle \, & > & 0.
\end{eqnarray*}
Denote by $T_1 \subset T$ the subtorus generated by $e_1$ and by $\LT_1 \subset \LT$ its Lie algebra. The quotient group and its Lie algebra are denoted by
$$
T_0 = T/T_1 \quad \mathrm{and} \quad \LT_0 = \LT/\LT_1
$$
and as before we identify $\LT_0^* \cong \left\{ w \in \LT^* \;|\; \langle w,e_1 \rangle = 0 \right\}$. Now $T_1$ acts trivially on
$$
V^I \;:=\; \left\{ \left( v_1 , \ldots , v_N \right) \in V \;|\; v_\nu = 0\ \mathrm{for}\ \nu \not\in I \right\}.
$$
For a small number $\delta > 0$ we define
$$
W_\delta \;:=\; \left( V \x \left[ -\eps , \eps \right] \right) \setminus N_\delta,
$$
where $N_\delta$ is the $\delta$-neighbourhood of the set
$$
V^I \x \left[ -\rho , \rho \right] \subset V \x \left[ -\eps , \eps \right]
$$
with
$$
\rho \;:=\; \frac{ \pi \cdot \max_{\nu \not\in I} \left| \left\langle w_\nu , e_1 \right\rangle \right| }{ \left\langle \eta , e_1 \right\rangle} \cdot \delta.
$$
The reason for this choice will become evident below. We choose $\delta$ small enough such that $N_\delta$ is contained in the interior of $V \x \left[ -\eps , \eps \right]$. Hence $W_\delta$ is a smooth connected oriented manifold with $\mcC^1$-boundary
$$
\partial W_\delta \;=\; (V \x \{ -\eps \}) \sqcup (V \x \{ \eps \}) \sqcup \partial N_\delta.
$$
If we orient $\partial W_\delta$ as the boundary of $W_\delta \subset V \x [ -\eps , \eps ]$ and $\partial N_\delta$ as part of this boundary then
$$
\partial W_\delta \;\cong\; (-V) \sqcup V \sqcup \partial N_\delta.
$$
On $W_\delta$ we have the $T$-action on the $V$-factor and we consider the equivariant map
$$
\Phi (v,t) \;:=\; \mu(v) - \tau_t.
$$
Denote by $M_\delta := \Phi^{-1}(0)$ the zero set of $\Phi$. Note that $\Phi (v,t) = 0$ implies that
\begin{equation} \label{eqn:t}
t \;=\; \frac{\pi}{ \langle \eta , e_1 \rangle} \sum_{\nu \not\in I} \left| v_\nu \right|^2 \langle w_\nu , e_1 \rangle.
\end{equation}
Now on $\partial N_\delta$ we have $\sum_{\nu \not\in I} \left| v_\nu \right|^2 \le \delta$ and hence the intersection of $M_\delta$ with the boundary component $\partial N_\delta$ is contained in the cylindrical part
$$
Z_\delta \;:=\; \left\{ (v,t) \in \partial N_\delta \,|\, |t| \le \rho \right\} = \left\{ v \in V \,\left|\, \sum_{\nu \not\in I} \left| v_\nu \right|^2 = \delta \right. \right\} \x [ -\rho,\rho ].
$$
We introduce the notation
$$
V_\delta^I \;:=\; \left\{ v \in V \,\left|\, \sum_{\nu \not\in I} \left| v_\nu \right|^2 = \delta \right. \right\}
$$
and
$$
S_\delta \;:=\; \left\{ (v_\nu)_{\nu \not\in I} \,\left|\, \sum_{\nu \not\in I} \left| v_\nu \right|^2 = \delta \right. \right\}
$$
and observe that $Z_\delta = V_\delta^I \x [ -\rho , \rho ]$ and $V_\delta^I \cong V^I \x S_\delta$.

\begin{prop}
For all $\eta$ in a dense and open subset of $\LT^*$ the value $0 \in \LT^*$ is regular for $\Phi : W_\delta \ra \LT^*$ and also for $\Phi|_{\partial W_\delta}$. The $T$-action on $\Phi^{-1}(0) = M_\delta$ is regular.
\end{prop}

\begin{proof}
Pick $(v,t) \in M_\delta$. For $t \ne 0$ the element $\tau_t$ is a regular value for the moment map $\mu : V \ra \LT^*$ and hence $\mathrm{d}\Phi_{(v,t)}$ is onto. This also holds for the differential of the restriction of $\Phi$ to the boundary components $V \x \{-\eps\}$ and $V \x \{\eps\}$. For $t = 0$ we have $\mu(v) = \tau_0$. Now the image of $\mathrm{d}\Phi_{(v,t)}$ is spanned by $\eta$ and the collection $(w_j)_{j \in J}$ with $J := \left\{ j \in \{ 1 , \ldots , N \} \,|\, v_j \ne 0 \right\}$. Suppose that $(w_j)_{j \in J}$ does not already span all of $\LT^*$. Then $\tau_0 \in W_J$ implies that $(w_j)_{j \in J \cap I}$ has rank $k-1$. Otherwise $\tau_0$ would be contained in two different walls, which we assumed not to be the case. So since $\eta$ is transversal to $W_I$ we also get surjectivity of $\mathrm{d}\Phi_{(v,t)}$. So for the first statement of the proposition it remains to show regularity for $\Phi_{|\partial N_\delta}$ and in fact only for $\Phi|_{Z_\delta}$ since $M_\delta \cap \partial N_\delta \subset Z_\delta$ as we remarked above. We define
$$
\varphi \;:=\; \Phi|_{Z_\delta} : Z_\delta \ra \LT^*.
$$
Let $(v,t) \in Z_\delta$. The tangent space $T_{(v,t)} Z_\delta$ is given by pairs $(x,s) \in V \x \R$ such that $\sum_{\nu \not\in I} g( v_\nu , x_\nu ) = 0$. Now
$$
\mathrm{d}\varphi_{(v,t)}(x,s) \;=\; 2 \pi \sum_{\nu = 1}^N g( v_\nu , x_\nu ) \cdot w_\nu - s \eta.
$$
So with $J := \left\{ j \in \{ 1 , \ldots , N \} \,|\, v_j \ne 0 \right\}$ the image of $\mathrm{d}\varphi_{(v,t)}$ is given by
$$
\mcI \;:=\; \left\{ \left. \sum_{j \in J} c_j w_j - s \eta \in \LT^* \,\right|\, \sum_{j \in J \setminus I} c_j = 0, \; c_j , s \in \R \right\}.
$$
Next we note that $(w_j)_{j \in J}$ spans $\LT^*$: For $t \ne 0$ this follows by regularity of $\tau_t$. For $t = 0$ we see as above that $(w_j)_{j \in J \cap I}$ has rank $k-1$. But now $\sum_{\nu \not\in I} \left| v_\nu \right|^2 = \delta$ implies that there is an index $l \in J \setminus I$ and hence $w_l$ spans the remaining dimension.

Given any $\tau \in \LT^*$ we thus find numbers $d_j \in \R$ with $\sum_{j \in J} d_j w_j = \tau$. In case $d := \sum_{j \in J \setminus I} d_j = 0$ this would already prove $\tau \in \mcI$. Else we pick $e_j \in \R$ with $\sum_{j \in J} e_j w_j = \eta$. Then if $e := \sum_{j \in J \setminus I} e_j \ne 0$ we can set
$$
c_j \;:=\; d_j - \frac{d}{e} e_j \quad , \quad s \;:=\; - \frac{d}{e}
$$
to obtain $\tau \in \mcI$. So for proving $\mcI = \LT^*$ it suffices to ensure by a suitable choice of $\eta$ that the equation $\sum_{j \in J} e_j w_j = \eta$ has at least one solution with $\sum_{j \in J \setminus I} e_j \ne 0$. Set
$$
\mcJ \;:=\; \left\{ J_0 \subset \{ 1 , \ldots , N \} \,|\, (w_j)_{j \in J_0} \ \mbox{is a basis for} \ \LT^* \right\}.
$$
For every $J_0 \in \mcJ$ we introduce
$$
E(J_0) \;:=\; \left\{ \sum_{j \in J_0} e_j w_j \,\left|\, \sum_{j \in J_0 \setminus I} e_j = 0 \right. \right\} \subset \LT^*
$$
and we choose
$$
\eta \in \LT^* \setminus \left( \mathrm{span}(w_i)_{i \in I} \cup \bigcup_{J_0 \in \mcJ} E(J_0) \right).
$$
This ensures that $\eta$ is indeed transversal to $W_I$. And to obtain the needed solution of $\sum_{j \in J} e_j w_j = \eta$ we pick a $J_0 \in \mcJ$ with $J_0 \subset J$. We then get a unique solution of $\sum_{j \in J_0} e_j w_j = \eta$ and setting all $e_j$ for $j \in J \setminus J_0$ equal to zero we get the desired solution of $\sum_{j \in J} e_j w_j = \eta$ with $\sum_{j \in J \setminus I} e_j \ne 0$. This finishes the proof of the first statement in the proposition.

The isotropy subgroups of the $T$-action at points $(v,t) \in M_\delta$ with $t \ne 0$ are finite by regularity of $\tau_t$. Now suppose $(v,0) \in M_\delta$ is fixed by a whole $1$-parameter family $\{ \exp(te) \in T \,|\, t \in \R \}$ for some $e \in \Lambda \setminus \{ 0 \}$. Again we set $J := \left\{ j \in \{ 1 , \ldots , N \} \,|\, v_j \ne 0 \right\}$ and observe
$$
j \in J \quad \iff \quad \langle w_j , e \rangle = 0.
$$
But similarly as above we see that $(w_j)_{j \in J} $ spans all of $\LT^*$, which would imply $e = 0$. This gives the desired contradiction.
\end{proof}

So if we pick a generic direction $\eta$ for the wall crossing we get a smooth cobordism $M_\delta$ with regular $T$-action. If we fix an orientation of $\LT^*$ we get induced orientations on $\mu^{-1}(\tau_{\pm \eps})$ and $M_\delta$. We equip the boundary components of $M_\delta$ with the induced boundary orientations and obtain
$$
\partial M_\delta \;\cong\; (-\mu^{-1}(\tau_{-\eps})) \sqcup \mu^{-1}(\tau_\eps) \sqcup \varphi^{-1}(0).
$$
By Stokes' formula for invariant integration we thus obtain the following result.

\begin{prop} \label{prop:wall_crossing_1}
If we orient $\varphi^{-1}(0)$ as part of the boundary of $M_\delta$ then
\begin{equation} \label{eqn:wall_crossing_1}
\chi^{V,\tau_\eps} - \chi^{V,\tau_{-\eps}} \;=\; - \int_{\varphi^{-1}(0) / T} \pi^* \;:\; H_T^*(V) \cong S(\LT^*) \ra \R.
\end{equation}
where $\pi : Z_\delta \ra V$ denotes the projection and integration over $\varphi^{-1}(0) / T$ is understood as $T$-invariant integration.
\end{prop}

\begin{rem}
This wall crossing formula also holds for non-generic $\eta$ if we use the cobordism property for Euler classes and interprete the integral on the right hand side as the Euler class of a regular $T$-moduli problem over $Z_\delta$.
\end{rem}

\subsubsection*{The reduced problem}

The task is now to put the above formula for the wall crossing into a more computable shape. In fact it is possible to deform the right hand side of \ref{eqn:wall_crossing_1} into the Euler class of a reduced toric moduli problem. Note that the $(w_i)_{i \in I}$ and $\tau_0$ can be viewed as elements in $\LT_0^*$ since they vanish on $e_1$.

\begin{lem}
The element $\tau_0 \in \LT_0^*$ is regular for the collection of characters of $T_0$ given by the $(w_i)_{i \in I}$.
\end{lem}

\begin{proof}
Given an index set $J \subset I$ and numbers $a_j > 0$ with $\sum_{j \in J} a_j w_j = \tau_0$ we observe that the $(w_j)_{j \in J}$ must have rank $k-1$ because otherwise $\tau_0$ would be contained in two different walls, which we assumed not to be the case.
\end{proof}

But in general $\tau_0$ will not be super-regular for the action of the reduced torus $T_0$. We introduce
\begin{eqnarray*}
\bar{w}_\nu & := & w_\nu - \frac{ \left\langle w_\nu , e_1 \right\rangle }{ \left\langle \eta , e_1 \right\rangle } \cdot \eta \\
\mu_0(v) & := & \pi \sum_{i \in I} \left| v_i \right|^2 \cdot w_i \\
R(v) & := & \pi \sum_{\nu \not\in I} \left| v_\nu \right|^2 \cdot \bar{w}_\nu \\
\varphi_s(v) & := & \mu_0(v) - \tau_0 + s \cdot R(v) , \quad \mathrm{for} \quad s \in [0,1]
\end{eqnarray*}
and observe that $\langle \varphi_s(v) , e_1 \rangle = 0$ for all $v$ and $s$. Hence we can --- and always will --- consider $\varphi_s$ as a map to $\LT_0^*$.

\begin{lem} \label{lem:phi_phi_0_diffeo}
$\varphi^{-1}(0)$ is $T$-equivariantly diffeomorphic to $\varphi_1^{-1}(0)$.
\end{lem}

\begin{proof}
The projection $\pi : Z_\delta = V_\delta^I \x [ -\delta , \delta ] \ra V_\delta^I$ restricts to a diffeomorphism on $\varphi^{-1}(0) \subset Z_\delta$ because the inverse is explicitly given by
$$
v = (v_\nu) \;\mto\; \left( (v_\nu) \quad,\quad t = \frac{\pi}{\langle \eta,e_1 \rangle} \sum_{\nu \not\in I} \left| v_\nu \right|^2 \langle w_\nu,e_1 \rangle \right)
$$
as can be seen from equation \ref{eqn:t}.
\end{proof}

Now $T_1$ by construction acts trivially on $V^I$. Hence the $T$-action on $V$ descends to a $T_0$-action on $V^I$ that is given by the weight vectors $(w_i)_{i \in I}$. Furthermore $\mu_0$ is a moment map for this $T_0$-action on $V^I$ and $\tau_0$ is a regular level. If we restrict to those $v \in V$ with $\sum_{\nu \not\in I} \left| v_\nu \right|^2 = \delta$ and choose $\delta$ small enough then $R(v)$ is as small as we want and hence also $\tau_0 - s \cdot R(v)$ will be regular for the $T_0$-action for all $s \in [0,1]$. So if we consider $\varphi_s$ as a map
$$
\varphi_s : V_\delta^I \cong V^I \x S_\delta \ra \LT_0^*
$$
we get that $0 \in \LT_0^*$ is a regular value of $\varphi_s$ for all $s \in [0,1]$. This shows
$$
\int_{\varphi_0^{-1}(0) / T} \;=\; \int_{\varphi_1^{-1}(0) / T} \;:\; H_T^*(V_\delta^I) \ra \R.
$$
Together with the preceding lemma we obtain the following refinement of proposition \ref{prop:wall_crossing_1}:

\begin{cor} \label{cor:wall_crossing_2}
We orient $V_\delta^I \cong V^I \x S_\delta$ by the complex orientation on $V^I$ and the standard orientation (via taking the outward pointing vector first) of the sphere $S_\delta$. We orient $\LT_0^*$ via the fixed orientation of $\LT^*$ as the kernel of the map ${e_1}^* : \LT^* \ra \R$. With the induced orientation on $\varphi_0^{-1}(0)$ we obtain
\begin{equation} \label{eqn:wall_crossing_2}
\chi^{V,\tau_\eps} - \chi^{V,\tau_{-\eps}} \;=\; \int_{\varphi_0^{-1}(0) / T} i^* \;:\; H_T^*(V) \cong S(\LT^*) \ra \R.
\end{equation}
with the inclusion $i : V_\delta^I \ra V$.
\end{cor}

\begin{proof}
We loose the minus sign from formula \ref{eqn:wall_crossing_1} by a change of orientation. All the rest is clear.
\end{proof}

\begin{rem}
Without the genericity assumption on the direction $\eta$ one shows more generally that the $\varphi_s$ for $s \in [0,1]$ define a homotopy of $T$-moduli problems. And lemma \ref{lem:phi_phi_0_diffeo} can be rephrased as a morphism between the moduli problems associated to $\varphi$ on $Z_\delta$ and $\varphi_1$ on $V_\delta^I$. The above result then remains true if we interprete the right hand side of \ref{eqn:wall_crossing_2} as the Euler class of the $T$-moduli problem associated to $\varphi_0$ on $V_\delta^I$ as explained below.
\end{rem}

Observe that $\varphi_0 : V_\delta^I \ra \LT_0^*$ defines a fibered $T$-moduli problem of the kind that we considered in section \ref{sec:fibred_G_moduli_problems}: We have the reduced $T / T_1 = T_0$-moduli problem on $V^I$ given by the section $\varphi_0|_{V^I} \;:\; V^I \ra \LT_0^*$. This is $T_0$-regular, because $\tau_0$ as an element of $\LT_0^*$ is regular for the weights $w_i \in \LT_0^*$ for $i \in I$. We write $\chi^{V^I,\tau_0}$ for the associated Euler class. The pull-back of the moduli problem under the projection
$$
\pi \;:\; V_\delta^I \;\cong\; V^I \x S_\delta \;\ra\; V^I
$$
yields the $T$-regular moduli problem $\left( V_\delta^I , V_\delta^I \x \LT_0^* , \varphi_0 \right)$, since $\varphi_0$ does not depend on the point in $S_\delta$ and $T_1$ acts locally freely. Hence by proposition \ref{prop:fibred_Euler_class} we can compute the right hand side of \ref{eqn:wall_crossing_2} by composing $\chi^{V^I,\tau_0}$ with invariant push-forward $\left( \pi / T_1 \right)_*$. So we finally obtain
\begin{equation*}
\chi^{V,\tau_\eps} - \chi^{V,\tau_{-\eps}} \;=\; \chi^{V^I,\tau_0} \circ \left( \pi / T_1 \right)_* \circ i^* \;:\; S(\LT^*) \ra \R.
\end{equation*}
With the genericity assumption on the direction $\eta$ this also follows immediately from the functoriality property $(4)$ in Proposition \ref{prop:invariant_integration} without referring to fibered moduli problems.

\subsubsection*{Localization}

Recall that we derived an explicit formula for the $T_1$-invariant integration $\left( \pi / T_1 \right)_* \circ i^* \;:\; S(\LT^*) \ra S(\LT_0^*)$ over an odd-dimensional sphere like $S_\delta$ in the localization example in section \ref{sec:example_localization}. Note that the weights for the $T$-action on $S_\delta$ are those $w_\nu$ with $\nu \not\in I$. Since the collection of the $w_\nu$ is assumed to be proper no two weights are negative multiples of each other. Also recall that every weight $w_\nu$ appears with multiplicity $n_\nu$ since $\dim_\C V_\nu = n_\nu$. So using proposition \ref{prop:example_localization} we can summarize and reformulate the wall crossing formula as follows.

\begin{thm} \label{thm:wall_crossing}
Let $\mathbbm{x} = \prod_{j = 1}^m x_j \in S^*(\LT^*)$. Then with the notation and conventions from above we have
$$
\chi^{V,\tau_\eps} (\mathbbm{x}) - \chi^{V,\tau_{-\eps}} (\mathbbm{x}) \;=\; \chi^{V^I,\tau_0} (\mathbbm{x}_0)
$$
with 
$$
\mathbbm{x}_0 (\xi) \;:=\; \frac{1}{2 \pi i} \oint \frac{ \prod_{j = 1}^m \langle x_j , \xi + z e_1 \rangle }{ \prod_{\nu \not\in I} \langle w_j , \xi + z e_1 \rangle^{n_\nu} } \; \dz \; .
$$
For every $\xi$ the integral is around a circle in the complex plane enclosing all poles of the integrand.
\end{thm} 

\subsection{Jeffrey-Kirwan localization}

\label{chap:Jeffrey_Kirwan_localization}

The preceding discussion on wall crossing and the computation of integrals over toric manifolds can be put into a much broader context. We will briefly outline the general theory. Suppose the torus $T$ acts on the compact manifold with boundary $M$ such that the action on the boundary $\partial M$ is regular. Let $\alpha$ be a $\mathrm{d}_T$-closed form on $M$. If the torus action was regular on all of $M$ then by Stokes' formula
$$
\int_{\partial M / T} \alpha \;=\; \int_{M / T} \mathrm{d} \alpha \;=\; 0.
$$
But $T$-invariant integration over $M$ is not defined as soon as there are points $p \in M$ with isotropy subgroups of positive dimension. As in section \ref{sec:torus_actions} we denote the isotropy subgroups by $S_p$ and their Lie algebras by $\LT_p$. Consider the set
$$
Z \;:=\; \left\{ p \in M \;|\; \dim( \LT_p ) \ge 1 \right\}.
$$
This set is $T$-invariant, but in general there is no reason why $Z$ should be a submanifold of $M$. But let us for the moment assume that
$$
Z \;=\; \bigsqcup_\nu Z_\nu
$$
is the disjoint union of connected and $T$-invariant submanifolds $Z_\nu$ such that all $p \in Z_\nu$ have isotropy $S_p = S_\nu$ for some fixed subgroup $S_\nu \subset T$.

We pick sufficiently small disjoint tubular neighbourhoods $U_\nu$ of the $Z_\nu$ and equivariantly identify the boundary of $U_\nu$ with the unit sphere bundle $SN_\nu$ in the normal bundle $N_\nu$ of $Z_\nu$ in $M$ with respect to an invariant metric. Now we obtain
\begin{eqnarray*}
&  \D \int_{\partial \left( M \setminus \sqcup_\nu U_\nu \right) / T} \alpha & \;=\; \quad \quad 0 \\
\iff \quad & \D \int_{\partial M / T} \alpha & \;=\; \quad \sum_\nu \int_{SN_\nu / T} \alpha.
\end{eqnarray*}
If we denote the projection $SN_\nu \ra Z_\nu$ by $\pi_\nu$ then we are precisely in the situation required to apply proposition \ref{prop:invariant_integration}: The group $S_\nu$ acts locally freely on the sphere bundle $SN_\nu$ but trivially on the base $Z_\nu$. If we introduce the notation
$$
T_\nu \;:=\; T / S_\nu
$$
then we obtain
\begin{equation*}
\D \int_{\partial M / T} \alpha \;=\; \sum_\nu \int_{Z_\nu / T_\nu} \left( \pi_\nu / S_\nu \right)_* \alpha.
\end{equation*}
This is the basic identity from which the Jeffrey-Kirwan localization formulae of Guillemin and Kalkman \cite{GK} and Martin \cite{Mar} and our results in section \ref{chap:toric_manifolds} can be deduced. In all three cases the manifold $M$ is obtained as the preimage of a certain path $\tau_t$ in $\LT^*$ under the moment map $\mu : X \ra \LT^*$ on some Hamiltonian $T$-space $X$ such that $\tau_t$ connects the point of reduction $\tau = \tau_0$ with the complement of the image of $\mu$. Hence
$$
\partial M / T \;=\; \mu^{-1}(\tau) / T \;=\; X /\!/ T (\tau)
$$
is the symplectic quotient of $X$. Compactness of $M$ is for example given if the moment map is proper.

The first step now is to ensure that the above assumption on the set $Z$ is satisfied. This is done by a suitable choice for the path $\tau_t$. In our case we have to demand that the path does not hit the intersection of two walls. Then the components $Z_\nu$ are the fixed point sets of certain one-dimensional subtori $S_\nu \subset T$. This can in fact be arranged for any Hamiltonian $T$-space $X$ (see Guillemin and Kalkman \cite{GK}).

Next one has to compute $\left( \pi_\nu / S_\nu \right)_* \alpha$. We obtain the corresponding formula in \ref{thm:wall_crossing} by applying our relative version of the usual Atiyah-Bott localization formula. This also works in the context studied in \cite{GK} and gives an alternative explanation for their \emph{residue operations}. In addition we can avoid to work with orbifolds.

Finally one has to iterate this procedure for the $T_\nu$-invariant integrals over the $Z_\nu$. Indeed our detailed exposition in section \ref{sec:toric_wall_crossing} is mainly in order to show that these \emph{reduced problems} are in fact again of the same form as the initial one and the combinatorial data can immediately be read off. This explicit correspondence has no analogue in the more general setting of Guillemin and Kalkman \cite{GK} and is the reason for why one can actually evaluate these iterated residues in the toric case, as it is done by Cieliebak and Salamon \cite{CS}.

\section{Vortex invariants}

\label{chap:vortex_invariants}

We now turn to vortex equations and the computation of vortex invariants for toric manifolds. First we summarize the usual vortex setup in the case of a linear torus action on $\C^N$. We then deform these equations as described in the introduction and prove that this deformation yields a homotopy of regular $T$-moduli problems. Then we study the deformed picture and gather consequences.

Throughout we freely use the notion and properties of moduli problems and associated Euler classes from section \ref{chap:moduli_problems}. For the definitions and facts about the Sobolev spaces that we use we refer to K.~Wehrheim \cite[Appendix B]{KW}. The gauge-theoretic results that we use originate in the work of Uhlenbeck \cite{Uhl} and of Donaldson and Kronheimer \cite{DK}, but we also use the book \cite{KW} for references.

\subsection{Setup}

\label{sec:vortex_setup}

Let $T$ be a $k$-dimensional torus. Using the notation from section \ref{chap:toric_manifolds} we consider a torus action $\rho$ on $\C^N$ that is given by $N$ weight vectors $w_\nu\in\Lambda^*$. We assume that this collection of weights is proper and spans $\LT^*$ so that we get non-empty regular quotients $X_{\C^N,\tau}$ for regular elements $\tau$ in the image of the moment map
$$
\mu : \C^N \ra \LT^* \; ; \; z \mto \pi \sum_{\nu = 1}^N \left| z_\nu \right|^2 \cdot w_\nu.
$$
We fix an element $\kappa \in \Lambda \cong H_2(\BT;\Z)$ and a principal $T$-bundle $\pi : P \ra \Sigma$ over a connected Riemann surface $\Sigma$ with characteristic vector $\kappa$. Explicitly we take $P := f^*\ET$ with a classifying map $f : \Sigma \ra \BT$ satisfying $f_*[\Sigma] = \kappa$. Since $T$ is connected $\BT$ is $1$-connected. Hence up to homotopy $f$ is uniquely determined by this property. Thus $P$ is uniquely determined by the choice for $\kappa$.

A $T$-equivariant map $u : P \ra \C^N$ is the same as a section in the Hermitian bundle $\mcV := P \x_\rho \C^N$ or a collection of sections $(u_1 , \ldots , u_N)$ in the line bundles $\mcL_\nu := P \x_{\rho_\nu} \C$ associated to the actions $\rho_\nu$ of $T$ on $\C$ given by the weights $w_\nu$. We use the convention that principal bundles carry right actions, hence we let $g \in T$ act on $P \x \C$ by
$$
g(x,z) := \left( xg^{-1} , \rho_\nu(g)z \right)
$$
and equivariance of $u : P \ra \C^N$ means that $u(xg^{-1}) = \rho(g)u(x)$. So for $p = \pi(x) \in \Sigma$ we set
$$
u_\nu(p) := \left[ x , \pi_\nu \circ  u(x) \right] \in \mcL_\nu
$$
with the projection $\pi_\nu : \C^N \ra \C$ onto the $\nu$-th component. This yields a well defined section of $\mcL_\nu$ since $\pi_\nu \circ \rho(g) = \rho_\nu(g) \circ \pi_\nu$. By construction the degree of $\mcL_\nu$ is given by $d_\nu := \left\langle w_\nu , \kappa \right\rangle$ and $\mcV = \bigoplus_{\nu = 1}^N \mcL_\nu$.

Let $A \in \Omega^1(P,\LT)$ be a connection form on $P$ and $u$ an equivariant map as above. The symplectic vortex equations for the pair $(u,A)$ at a parameter $\tau \in \LT^*$ now take the form
$$
(*) \quad \left\{ \quad
\begin{array}{lr@{\quad = \quad}l}
(\mathrm{I}) & \bar{\del}_A u & 0 \\
(\mathrm{II}) & *F_A + \mu(u) & \frac{\kappa}{\vol(\Sigma)} + \tau.
\end{array}
\right.
$$
In the first equation we have the linear Cauchy-Riemann operator $\bar{\del}_A$ on $\mcV$ that is associated to the covariant derivative $\mathrm{d}_A u := \mathrm{d}u - X_A(u)$. It is given by
$$
\bar{\del}_A u = \mathrm{d}_A u + i \circ \mathrm{d}_A u \circ j_\Sigma
$$
with a fixed complex structure $j_\Sigma$ on $\Sigma$ and the standard complex structure $i$ on $\C^N$. For the second equation we also fix a metric (respectively a volume form $\dvol_\Sigma$) on $\Sigma$ to get the Hodge $*$-operator and the volume $\vol(\Sigma)$ and we also have to choose an inner product on the Lie algebra $\LT$ to identify it with its dual space $\LT^*$. Note that in general the curvature $F_A$ is a two-form on $\Sigma$ with values in the bundle $P \x_\Ad \LT$, which in this case is trivial since $T$ is abelian. Hence $*F_A$ is just a map $\Sigma \ra \LT$. And by $T$-invariance of the moment map $\mu$ the composition $\mu(u)$ also descends to a map $\Sigma \ra \LT$.

We consider the gauge group $\mcG(P)$ of the principal bundle $P$ and fix a point $x_0 \in P$ to get the based gauge group
$$
\mcG_0(P) := \left\{ \gamma \in \mcG(P) \;|\; \gamma(x_0) = e \right\}.
$$
An element $\gamma \in \mcG(P)$ is a smooth $T$-invariant map $\gamma : P \ra T$ that acts on pairs $(u,A)$ by $\gamma(u,A) := ( \gamma^*u , \gamma^*A )$ with
$$
\left( \gamma^*u \right) (x) := u\left( x\gamma(x) \right) = \rho\left( \gamma(x)^{-1} \right) u(x)
$$
and for $v \in T_xP$
$$
\left( \gamma^*A \right)_x(v) := A_x(v) + {\gamma(x)^{-1}}_* \left( \mathrm{d}\gamma \right)_x (v),
$$
in short $\gamma^*A = A + \gamma^{-1} \mathrm{d}\gamma$. Now the based gauge group $\mcG_0(P)$ acts freely on the configuration space of pairs $(u,A)$. In fact $\gamma^*A = A$ implies that $\gamma : P \ra T$ is a constant map and since $P$ is connected the condition $\gamma(x_0) = e \in T$ for based gauge transformations $\gamma \in \mcG_0(P)$ implies that $\gamma(x) = e$ for all $x \in P$. Using suitable Sobolev completions the quotient of this action becomes a smooth Hilbert manifold $\mcB$. Explicitly we consider
\begin{equation} \label{def:mcB}
\mcB := \frac{ W^{k,2} \left( \Sigma , \mcV \right) \x \mcA^{k,2}(P) }{ \mcG_0^{k+1,2}(P) }.
\end{equation}
By the local slice theorem (see K.~Wehrheim \cite[Theorem 8.1]{KW}) the quotient $\mcA^{k,2}(P) / \mcG_0^{k+1,2}(P)$ is a smooth Hilbert manifold and the projection from $\mcB$ onto the connection-part makes $\mcB$ to a Hilbert space bundle over this quotient. Taking $k \ge 3$ ensures that all objects are $\mcC^2$. In the end it does not matter which $k \ge 3$ we take.

The map $\mcG(P) / \mcG_0(P) \ra T \;;\; [\gamma] \mto \gamma(x_0)$ is a diffeomorphism. Hence the action of the whole gauge group $\mcG(P)$ descends to a $T$-action on the quotient $\mcB$ given by 
$$
g \left[ u , A \right] := \left[ \rho\left( g^{-1} \right) u , A \right] \quad \mathrm{with} \quad g \in T , \left[ u , A \right] \in \mcB.
$$
Here we use that for an abelian Lie group the gauge group actually splits into the product $T \x \mcG_0(P)$ and that the constant gauge transformations $g \in T \subset \mcG(P)$ act trivially on connections.

The equations $(*)$ now give rise to a $T$-moduli problem with base $\mcB$ and total space
\begin{equation} \label{def:mcE}
\mcE := \frac{ W^{k,2} \left( \Sigma , \mcV \right) \x \mcA^{k,2}(P) \x \mcF }{ \mcG_0^{k+1,2}(P) }
\end{equation}
with fiber
\begin{equation} \label{def:mcF}
\mcF := W^{k-1,2} \left( \Sigma , \Lambda^{0,1} T^*\Sigma \otimes_\C \mcV \right) \oplus W^{k-1,2} \left( \Sigma , \LT \right)
\end{equation}
and section
$$
\mcS : \mcB \ra \mcE \; ; \; \left[ u , A \right] \mto 
\left[ u , A , 
\left(
\begin{array}{c}
\bar{\partial}_A u \\
*F_A + \mu(u)  - \frac{\kappa}{\vol(\Sigma)} - \tau
\end{array}
\right)
\right].
$$
On the fiber $\mcF$ the gauge group $\mcG(P)$ acts only on sections in the bundle $\mcV$. Thus a short computation shows that the section $\mcS$ is $T$-equivariant and well defined, i.e.~it descends to the $\mcG_0(P)$-quotients.

The real index of this moduli problem ist $(N-k) \cdot \chi(\Sigma) + 2\sum_{\nu = 1}^N d_\nu$ and it is regular if $\tau$ is regular. For the computation of the index, the regularity and the compactness properties we refer to Cieliebak, Gaio, Mundet and Salamon \cite{CGMS}. The definition of the orientation on $\det(\mcS)$ will be discussed together with all further orientation issues in section \ref{sec:orientations}. The Euler class of this $T$-moduli problem then gives rise to the vortex invariant. We denote it by
$$
\Psi^{\rho,\tau}_{\kappa,\Sigma} : S^*(\LT^*) \otimes H^*(\mcA/\mcG_0) \ra \R.
$$
In case $\Sigma = S^2$ has genus zero we omit the surface and just write $\Psi^{\rho,\tau}_\kappa$. This map is constructed as follows: Elements in $S^*(\LT^*) \otimes H^*(\mcA/\mcG_0)$ are pulled back to $H_T^*(\mcB)$ via the equivariant evaluation map $\mcB \ra \C^N$; $[u,A] \mto u(p_0)$ and the projection $\mcB \ra \mcA/\mcG_0$, and are then integrated $T$-invariantly over the compact space of solutions $\mcM := \mcS^{-1}(0)$. If the space $\mcM$ is not cut out transversally this integration is understood in terms of the Euler class.

\begin{rem}
In case of genus zero the space $\mcA(P)/\mcG_0(P)$ is in fact contractible. This follows for example from lemma \ref{lem:jacobian_torus} below. Hence the vortex invariant reduces to a map
$$
\Psi^{\rho,\tau}_\kappa : S^*(\LT^*) \ra \R.
$$
\end{rem}

\subsection{Deformation}

\label{sec:vortex_deformation}

The objects in the curvature equation $(\mathrm{II})$ of $(*)$ are considered as elements of $W^{k-1,2}(\Sigma,\LT)$. Using the fixed volume form $\dvol_\Sigma$ we can split this vector space into the direct sum
$$
W^{k-1,2}(\Sigma,\LT) = \underline{ W^{k-1,2}(\Sigma,\LT) } \oplus \LT
$$
with
$$
\underline{ W^{k-1,2}(\Sigma,\LT) } := \left\{ F \in W^{k-1,2}(\Sigma,\LT) \;|\; \int_\Sigma F \cdot \dvol_\Sigma = 0 \right\},
$$
the space of maps with mean value equal to zero. For any map $F : \Sigma \ra \LT$ we denote by
$$
\overline{F} := \frac{ \int_\Sigma F \cdot \dvol_\Sigma }{ \vol(\Sigma) } \in \LT
$$
its mean value and by
$$
\underline{F} := F - \overline{F}
$$
its projection onto the space of maps of zero mean value. Observe that by Chern-Weil theory we have
$$
\overline{*F_A} = \frac{ \int_\Sigma *F_A \cdot \dvol_\Sigma }{ \vol(\Sigma) } = \frac{\kappa}{\vol(\Sigma)}
$$
for any connection $A$. For later use we also compute
\begin{equation} \label{eqn:overline_mu}
\overline{\mu(u)} = \frac{\pi}{\vol(\Sigma)} \cdot \sum_{\nu = 1}^N \left\| u_\nu \right\|^2_{L^2(\Sigma,\mcL_\nu)} \cdot w_\nu.
\end{equation}
In this splitting the vortex equations become
$$
(*) \quad \left\{ \quad 
\begin{array}{lr@{\quad = \quad}l}
(\mathrm{I}) & \bar{\partial}_A u & 0 \\
(\mathrm{II}) & *F_A + \mu(u) & \frac{\kappa}{\vol(\Sigma)} + \overline{\mu(u)} \\
(\mathrm{II\/I}) & \overline{\mu(u)} & \tau.
\end{array}
\right.
$$
We regroup and introduce a parameter $\eps \in [0,1]$ in the second equation to obtain
$$
(*)_\eps \quad \left\{ \quad 
\begin{array}{lr@{\quad = \quad}l}
(\mathrm{I}) & \bar{\partial}_A u & 0 \\
(\mathrm{II})_\eps & *F_A - \frac{\kappa}{\vol(\Sigma)} & \eps \left( \overline{\mu(u)} - \mu(u) \right) \\
(\mathrm{II\/I}) & \overline{\mu(u)} & \tau.
\end{array}
\right.
$$

\begin{rem} \label{rem:epsilon_omega}
One can interprete the equations $(*)_\eps$ as the vortex equations for the same Hamiltonian group action, but with rescaled symplectic form $\eps \cdot \omega$ on the target manifold (hence the moment map for the action gets multiplied by $\eps$) and with rescaled parameter $\eps \cdot \tau$. So as long as $\eps \ne 0$ nothing particular will happen.
\end{rem}

We claim, and will prove it in the remainder of this section, that the deformed equations with $\eps = 0$ yield the same invariants as the usual vortex equations. We consider the Hilbert space bundle $[0,1] \x \mcE \ra [0,1] \x \mcB$ with $\mcB$ as in \ref{def:mcB} and $\mcE$ as in \ref{def:mcE}, but we write the fiber from \ref{def:mcF} as
$$
\mcF := W^{k-1,2} \left( \Sigma , \Lambda^{0,1} T^*\Sigma \otimes_\C \mcV \right) \oplus \underline{W^{k-1,2} \left( \Sigma , \LT \right)} \oplus \LT.
$$
In this bundle the equations $(*)_\eps$ give rise to the section
\begin{eqnarray*}
\mcS : [0,1] \x \mcB & \ra & [0,1] \x \mcE \\
\left( \eps , \left[ u , A \right] \right)
& \mto &
\left( \eps ,
\left[ u , A ,
\left(
\begin{array}{c}
 \bar{\partial}_A u \\
*F_A - \frac{\kappa}{\vol(\Sigma)} + \eps \cdot \mu(u) - \eps \cdot \overline{\mu(u)} \\
\overline{\mu(u)} - \tau
\end{array}
\right)
\right]
\right).
\end{eqnarray*}
Our claim now follows from the following theorem in combination with the cobordism property of the Euler class.

\begin{thm} \label{thm:homotopy}
The triple $\left\{ [0,1] \x \mcB , [0,1] \x \mcE , \mcS \right\}$ defines a homotopy of regular $T$-moduli problems if $\tau$ is regular.
\end{thm}

There are three things to check that are not obvious:
\begin{description}
\item[(Orientation)] There is a natural orientation on all determinant bundles $\det\left( \mcS|_{ \{ \eps \} \x \mcB } \right)$ with $\eps \in [0,1]$ that fit together to define an orientation on $\det(\mcS)$ once we fix an orientation for the interval $[0,1]$.
\item[(Regularity)] The $T$-isotropy subgroup at every point $[u,A] \in \mcB$ with $(u,A)$ solving $(*)_\eps$ for some $\eps \in [0,1]$ is finite.
\item[(Compactness)] The space of solutions $\left\{ (\eps,[u,A]) \; | \; (u,A) \; \mathrm{solves} \; (*)_\eps \right\}$ in $[0,1] \x \mcB$ is compact.
\end{description}

\subsubsection*{Orientation}

We refer to the argument from Cieliebak, Gaio, Mundet and Salamon \cite[chapter 4.5]{CGMS} for the orientation in the original vortex moduli problem and argue that the additional parameter $\eps$ only induces a compact perturbation of the original operators and hence their determinants are canonically identified. We postpone the detailed discussion to section \ref{sec:orientations}.

\subsubsection*{Regularity}

The idea for this proof is taken from \cite[Remark 4.3]{CGMS}. An element $g \in T$ fixes a map $u : P \ra \C^N$ if
$$
g \in \bigcap_{x \in P} \Iso_{u(x)}(T) =: S.
$$
Now $V_S := \left\{ z \in \C^N \; | \; S \subset \Iso_z(T) \right\}$ is a linear subspace, so $\mu(V_S)$ is a convex cone in $\LT^*$ based at zero. By construction $u$ maps into $V_S$, hence we obtain $\int_\Sigma \mu(u) \cdot \dvol_\Sigma \in \mu(V_S)$ and hence also
$$
\overline{\mu(u)} = \frac{ \int_\Sigma \mu(u) \cdot \dvol_\Sigma }{ \vol(\Sigma) }  \in \mu(V_S).
$$
Equation $(\mathrm{II\/I})$ then implies $\tau \in \mu(V_S)$. Thus there exists a $z_0 \in \mu^{-1}(\tau) \cap V_S$. By definition of $V_S$ we have $S \subset \Iso_{z_0}(T)$. By assumption $\tau$ is regular, which implies that $T$ acts with finite isotropy on $z_0 \in \mu^{-1}(\tau)$, hence $S$ is finite.

\subsubsection*{Compactness}

Let $(\eps_j,u_j,A_j) \in [0,1] \x W^{k,2} \left( \Sigma , \mcV \right) \x \mcA^{k,2}(P)$ be a sequence of triples with $(u_j,A_j)$ solving $(*)_{\eps_j}$. We want to prove the existence of a convergent subsequence in the quotient by $\mcG_0^{k+1,2}(P)$. By Sobolev embedding it suffices to exhibit a uniform $W^{k+1,2}$-bound on all $(u_j,A_j)$ after modifying them by suitable gauge transformations. We do this in seven steps.

\subsubsection*{Step 1:}

The curvature of the connections $A_j$ is uniformly bounded, i.~e.~there exists a constant $C$ with
$$
\| F_{A_j} \|_{L^\infty} \le C
$$
for all $j$.

\begin{proof} The computation is analogous to the one in Cieliebak, Gaio and Salamon \cite[Proposition 3.5]{CGS}. There an $L^\infty$-bound on the maps $u$ of solutions $(u,A)$ to the usual vortex equations $(*)$ is shown. By the curvature equation $(\mathrm{II})$ this yields such a uniform bound on the curvatures. In our case for a solution $(\eps,u,A)$ of $(*)_\eps$ we only get an $L^\infty$-bound on $\eps \cdot \mu(u)$, which does not give a uniform bound on the maps $u$, but which is sufficient to also uniformly bound the curvatures via $(\mathrm{II})_\eps$.

We fix a cover of $\Sigma$ by a finite number of holomorphic charts $U_k$ and equivariant trivializations $P|_{U_K} \cong U_k \x T$. Thus we get induced trivializations $\mcV|_{U_k} \cong U_k \x \C^N$ via the local lifts $U_k \ra P|_{U_k} \; ; \; z \mto (z , \unit)$. In local holomorphic coordinates $(s,t)$ on $U_k$ we can write a connection $A$ as $\varphi(s,t) \ds + \psi(s,t) \dt$ with functions $\varphi, \psi : U_k \ra \LT$. Then the local expression for the curvature is
$$
F_A = ( \del_s \psi - \del_t \varphi ) \ds \wedge \dt,
$$
because the torus is abelian. The volume form $\dvol_\Sigma$ is locally given by $\lambda^2 \ds \wedge \dt$ with a positive function $\lambda$. The infinitesimal action of an element $\xi \in \LT$ at $z \in \C^N$ is given by $X_\xi(z) = \dot{\rho}(\xi) \cdot z$ with the skew-Hermitian matrix
$$
\dot{\rho}(\xi) := 2 \pi i \cdot \mathrm{diag}( \langle w_\nu , \xi \rangle ).
$$
With this notation the equations $(\mathrm{I})$ and $(\mathrm{II})_\eps$ in local holomorphic coordinates $(s,t)$ on a chart $U_k$ are equivalent to
\begin{eqnarray}
(\mathrm{I})^{loc} \quad \; \del_s u - \dot{\rho}(\varphi) u + i \left( \del_t u - \dot{\rho}(\psi) u \right) & = & 0,  \label{equ:I_loc} \\ 
(\mathrm{II})_\eps^{loc} \qquad \qquad \quad \frac{ \del_s \psi - \del_t \varphi }{ \lambda^2 } - \frac{\kappa}{\vol(\Sigma)} & = & \eps \left( \overline{\mu(u)} - \mu(u) \right). \quad \label{equ:II_eps_loc}
\end{eqnarray}
We introduce $\nabla_s := \del_s - \dot{\rho}(\varphi)$ and $\nabla_t := \del_t - \dot{\rho}(\psi)$ and denote the standard real inner product on $\C^N$ by $\langle \, , \, \rangle$. Then $\langle u , \nabla_s u \rangle = \langle u , \del_s u \rangle$, because $\dot{\rho}$ is skew-symmetric. With this observation a short computation shows that
$$
\frac{ \Delta |u|^2 }{2} := \frac{1}{2} \left( \del_s \del_s + \del_t \del_t \right) \langle u , u \rangle = \left| \nabla_s u \right|^2 + \left| \nabla_t u \right|^2 + \left\langle u , \nabla_s \nabla_s u + \nabla_t \nabla_t u \right\rangle.
$$
Now equation \ref{equ:I_loc} reads $\nabla_s u = - i \nabla_t u$ and hence
$$
\nabla_s \nabla_s u + \nabla_t \nabla_t u = \nabla_s \left( - i \nabla_t u \right) + \nabla_t \left( i \nabla_s \right) u = i \left( \nabla_t \nabla_s u - \nabla_s \nabla_t u \right).
$$
Furthermore
\begin{eqnarray*}
\nabla_t \nabla_s u & = & \left( \del_t - \dot{\rho}(\psi) \right) \left( \del_s - \dot{\rho}(\varphi) \right) u \\
& = & \del_t \del_s u - \dot{\rho}(\del_t \varphi) u - \dot{\rho}(\varphi) \del_t u - \dot{\rho}(\psi) \del_s u + \dot{\rho}(\psi) \dot{\rho}(\varphi) u
\end{eqnarray*}
and correspondingly with $(s,\varphi)$ and $(t,\psi)$ interchanged. Combining all these equations we obtain the following estimate:
$$
\frac{ \Delta |u|^2 }{2} \ge \left\langle u , i \left( \dot{\rho}( \del_s \psi ) - \dot{\rho}( \del_t \varphi ) \right) u \right\rangle
$$
Here we used that the partial derivatives $\del_s$ and $\del_t$ as well as diagonal matrices commute. If we write 
$$
\xi_\eps := \lambda^2 \left( \frac{\kappa}{\vol(\Sigma)} + \eps \overline{\mu(u)} - \eps \mu(u) \right)
$$
then equation \ref{equ:II_eps_loc} reads $\del_s \psi - \del_t \varphi = \xi_\eps$ and hence
\begin{eqnarray*}
\frac{ \Delta |u|^2 }{2} & \ge & \left\langle u , i \dot{\rho}( \xi_\eps ) u \right\rangle \\
& = & - \left\langle u , 2 \pi \cdot \mathrm{diag}( \langle w_\nu , \xi_\eps \rangle ) \cdot u \right\rangle \\
& = & -  \sum_\nu 2 \pi \cdot \langle w_\nu , \xi_\eps \rangle \cdot |u_\nu|^2 \\
& = & -2 \left\langle \pi \sum_\nu |u_\nu|^2 \cdot w_\nu , \xi_\eps \right\rangle \\
& = & -2 \left\langle \mu(u) , \xi_\eps \right\rangle.
\end{eqnarray*}
Now let $x \in P$ be a point at which the function $|u|^2$ attains its maximum. Then $0 \ge \Delta |u|^2(x)$ and hence
\begin{eqnarray*}
0 & \le & \left\langle \mu(u(x)) , \xi_\eps(x) \right\rangle \\
& = & \lambda^2 \cdot \left\langle \mu(u(x)) , \frac{\kappa}{\vol(\Sigma)} + \eps \tau - \eps \mu(u(x)) \right\rangle \\
& \le & \lambda^2 \cdot \left| \mu(u(x)) \right| \cdot \left| \frac{\kappa}{\vol(\Sigma)} + \eps \tau \right| - \lambda^2 \cdot \eps \cdot \left| \mu(u(x)) \right|^2,
\end{eqnarray*}
where we useed the third equation $\overline{\mu(u)} = \tau$ for solutions of $(*)_\eps$. This implies
$$
\eps \cdot \left| \mu(u(x)) \right| \le \left| \frac{\kappa}{\vol(\Sigma)} \right| + |\tau| =: c.
$$
Now since $\mu$ is proper this gives a bound on $\eps \cdot |u(x)|^2$, which by definition of $x$ holds for $\eps \cdot |u(y)|^2$ for all $y \in P$. This in turn gives a uniform bound for $\eps \cdot |\mu(u)|$. Explicitly: Since the collection of $w_\nu$ is proper, by definition \ref{def:proper} there exists an element $\eta \in \LT$ with $|\eta| = 1$ and $\langle w_\nu , \eta \rangle > 0$ for all $\nu$. We denote $m := \min \left\{ \langle w_\nu , \eta \rangle \; | \; \nu = 1,\ldots,N \right\} > 0$ and compute
$$ 
\eps \cdot |u(x)|^2 \le \eps \cdot \sum_{\nu=1}^N |u_\nu(x)|^2 \cdot \frac{\langle w_\nu , \eta \rangle}{m} = \eps \cdot \frac{\langle \mu(u(x)) , \eta \rangle}{\pi m} \le \eps \cdot |\mu(u(x))| \cdot \frac{|\eta|}{\pi m} \le \frac{c}{\pi m}.
$$
Hence we obtain
\begin{equation} \label{equ:eps_mu_u}
\eps \cdot |\mu(u)| \le \eps \cdot \pi \sum_{\nu=1}^N |u_\nu|^2 \cdot \max\{ |w_\nu| \} \le \frac{c \cdot \max\{ |w_\nu| \}}{m}
\end{equation}
for any triple $(\eps,u,A)$ with $(u,A)$ solving $(*)_{\eps}$. Now by $(\mathrm{II})_\eps$ we obtain a uniform bound on the curvatures as claimed.
\end{proof}

\subsubsection*{Step 2:}

Fix a real number $q > 1$ and a smooth reference connection $A_0$ on $P$. Then there exist gauge transformations $\gamma_j \in \mcG^{2,q}(P)$ and a constant $C$ such that
$$
\| \gamma_j^*A_j - A_0 \|_{W^{1,q}} \le C
$$
for all $j$.

\begin{proof} By step $1$ the curvatures $F_{A_j}$ are uniformly $L^q$-bounded. Hence the claim follows by Uhlenbeck compactness (see K.~Wehrheim \cite[Theorem A]{KW}).
\end{proof}

\subsubsection*{Step 3:}

There is a uniform $W^{2,2}$-bound on the maps $\gamma_j^*u_j$, i.~e.~there exists a constant $C$ with
$$
\| \gamma_j^*u_j \|_{W^{2,2}} \le C
$$
for all $j$.

\begin{proof} By equation $(\mathrm{II\/I})$ there is a uniform $L^2$-bound, because
$$
\| u \|_{L^2}^2 = \sum_{\nu = 1}^N \| u_\nu \|_{L^2}^2 \le \frac{ \sum_{\nu = 1}^N \| u_\nu \|_{L^2}^2 \cdot \langle w_\nu , \eta \rangle }{ m } = \frac{ \langle \tau \cdot \vol(\Sigma) , \eta \rangle }{ \pi m } =: {c_0}^2.
$$
We fix a smooth partition of unity $(f_k)$ subordinate to the fixed finite cover $(U_k)$. It now suffices to give uniform $W^{2,2}$-bounds for maps $f \cdot u$ having compact support in one chart $U$ with holomorphic coordinates $(s,t)$. We write $\| \cdot \|_{l,p}$ for $\| \cdot \|_{W^{l,p}(U)}$ and $\bar{\del} := \del_s + i \del_t$. We compute
$$
\| \del_s(fu) \|_{0,2}^2 + \| \del_t(fu) \|_{0,2}^2 = \| \bar{\del}(fu) \|_{0,2}^2 \le \left( c_f \cdot c_0 + \| \bar{\del} u \|_{0,2} \right)^2.
$$
Here $c_f$ is a constant depending only on the function $f$. Now equation \ref{equ:I_loc} implies $\bar{\del} u = ( \dot{\rho}(\varphi) + i \dot{\rho}(\psi) ) u =: B \cdot u$ and by step $2$ the entries of the diagonal matrix $B$ are $L^\infty$-bounded. Again using the uniform $L^2$-bound on $u$ this gives a uniform $W^{1,2}$-bound on $f \cdot u$ and hence a bound
$$
\| u \|_{W^{1,2}} \le c_1.
$$
Now consider the same computation as above with $(fu)$ replaced by $\del_s(fu)$. We obtain the estimate
$$
\| \del_s\del_s(fu) \|_{0,2}^2 + \| \del_t\del_s(fu) \|_{0,2}^2 \le \left( c_f \cdot (c_0+c_1) + \| \bar{\del} \del_s u \|_{0,2} \right)^2.
$$
The $s$-derivative of \ref{equ:I_loc} gives $\bar{\del} \del_s u = \del_s B \cdot u + B \cdot \del_s u$. We get an $L^2$-bound on $B \cdot \del_s u$ via the $L^\infty$-bound on $B$ and the $W^{1,2}$-bound on $u$. And for the $L^2$-bound on $\del_s B \cdot u$ note that the $W^{1,2}$-bound on $u$ gives rise to an $L^4$-bound on $u$, as well as a $W^{1,4}$-bound on the connection $A$ (from step $2$ if we take $q \ge 4$) gives an $L^4$-bound on $\del_s B$, and hence we get an $L^2$-bound on the product.

Together with the corresponding estimate from replacing $(fu)$ by $\del_t(fu)$ we finally get the desired $W^{2,2}$ bound.
\end{proof}

From this point the proof continues along the lines of Cieliebak, Gaio, Mundet and Salamon \cite[Theorem 3.2]{CGMS}. There compactness for solutions of the usual vortex equations is shown under the additional assumption that the first derivatives of the maps $u$ satisfy a uniform $L^\infty$-bound.

We replace the sequence $(\eps_j,u_j,A_j)$ by $(\eps_j,\gamma_j^*u_j,\gamma_j^*A_j)$, so we have a uniform $W^{1,q}$-bound on $(u_j,A_j)$ modulo $\mcG^{2,q}(P)$ with $q \ge 4$. Hence we get a weak $W^{1,q}$-limit $(\eps,u,A)$ for some subsequence that we again denote by $(\eps_j,u_j,A_j)$. It follows that $(u,A)$ satisfies $(*)_\eps$ and the next step is to improve the regularity of this limit solution.

\subsubsection*{Step 4:}

Fix a real number $q > 2$. Given a solution $(u,A)$ of $(*)_\eps$ of class $W^{1,q}$ there exists a gauge transformation $\gamma \in \mcG^{2,q}(P)$ such that $\gamma^*(u,A)$ is smooth.

\begin{proof} This follows exactly as in \cite[Theorem 3.1]{CGMS}. The additional factor of $\eps \in [0,1]$ in the curvature equation does not affect the argument.
\end{proof}

So if we replace the previous gauge transformations $\gamma_j$ by $\gamma_j \cdot \gamma$ then the limit solution $(u,A)$ is actually smooth. Note that the product of two $W^{2,q}$ gauge transformations on $P$ is again of class $W^{2,q}$.

Next we want to apply the local slice theorem \cite[Theorem 8.1]{KW} to put the connections $A_j$ into Coulomb gauge relative to $A$. We have a uniform bound on $\| A_j - A \|_{W^{1,q}}$ and since the $A_j$ converge strongly in $\mcC^0$ to $A$ a suitable subsequence lies in a sufficiently small $L^\infty$-neighbourhood of $A$ so that we can apply the local slice theorem. We obtain a sequence $\beta_j \in \mcG^{2,q}(P)$ such that
$$
\mathrm{d}_A^* ( \beta_j^* A_j - A ) = 0
$$
and a constant $C$ with
$$
\| \beta_j^* A_j - A \|_{L^\infty} \le C \cdot \| A_j - A \|_{L^\infty}
$$
and
$$
\| \beta_j^* A_j - A \|_{W^{1,q}} \le C \cdot \| A_j - A \|_{W^{1,q}}
$$
for all $j$.

By the second inequality a subsequence of $\beta_j^*A_j$ will again have a weak $W^{1,q}$-limit which by the first inequality has to be equal to $A$. But we also have to control $\beta_j^*u_j$.

\subsubsection*{Step 5:}

The sequence $\beta_j \in \mcG^{2,q}(P)$ is uniformly $W^{2,q}$-bounded.

\begin{proof}
By the second of the above inequalities we get a uniform bound on $\| \beta_j^{-1} \mathrm{d}\beta_j \|_{W^{1,q}}$ which in turn gives a uniform $W^{2,q}$-bound on $\beta_j$. See K.~Wehrheim \cite[Appendices A, B]{KW} for details.
\end{proof}

Thus replacing $(u_j,A_j)$ with $\beta_j^*(u_j,A_j)$ yields an equivalent sequence with uniform $W^{1,q}$-bounds and connections in relative Coulomb gauge to $A$ modulo $\mcG^{2,q}(P)$ with $q \ge 4$. In fact a uniform $W^{1,q}$-bound on the gauge transformations $\beta_j$ would suffice to keep the $\beta_j^*u_j$ bounded.

\subsubsection*{Step 6:}

For every $l \ge 1$ there is a number $C_l$ with the property that
$$
\| u_j \|_{W^{l,2}} \le C_l
$$
and
$$
\| A_j - A \|_{W^{l,2}} \le C_l
$$
for all $j$.

\begin{proof}
We already have the result for $l = 1$ and we get the estimates for higher $l$ by induction, because the gauge condition $\mathrm{d}_A^*(A_j-A) = 0$ together with the curvature equation $(\mathrm{II})_\eps$ and equation $(\mathrm{I})$ form a complete elliptic system.
\end{proof}

Finally we do not want these uniform estimates on $(u_j,A_j)$ modulo the action of gauge transformations $\gamma_j \in \mcG^{2,q}(P)$, but modulo based gauge transformations of class $W^{k+1,2}$. Arranging the $\gamma_j$ to be based is easy in our abelian case: We can just multiply every $\gamma_j$ by the constant gauge transformation $\gamma_j(x_0)^{-1}$, while keeping uniform bounds. To get the desired regularity we need

\subsubsection*{Step 7:}

Suppose $A \in \mcA^{k,2}(P)$ and $\gamma \in \mcG^{2,2}(P)$ are such that $\gamma^*A$ is again of class $W^{k,2}$ and let $k \ge 2$. Then in fact $\gamma \in \mcG^{k+1,2}(P)$.

\begin{proof}
By definition $\gamma^*A = \gamma^{-1}A\gamma + \gamma^{-1}d\gamma$, hence $d\gamma = \gamma (\gamma^*A) - A \gamma$. Now the product of $W^{2,2}$ with $W^{k,2}$ for sections over a $2$-dimensional manifold $\Sigma$ is again of class $W^{2,2}$ if $k \ge 2$. Hence $\gamma$ is in fact of class $W^{3,2}$ and the result follows by induction.
\end{proof}

This completes the proof for compactness. We finish the deformation discussion with a list of important remarks.

\begin{rem} The above arguments also show that the solution space and hence the associated invariants do not depend on the choice of $k$ in the definition of the configuration space $\mcB$: All solutions are gauge equivalent to smooth ones.
\end{rem}

\begin{rem}
The corresponding deformation is already used by Cieliebak, Gaio, Mundet and Salamon \cite[chapter 9]{CGMS} for the computation of vortex invariants of weighted projective spaces. But the claimed property of beeing a homotopy of regular moduli problems is not further justified.
\end{rem}

\begin{rem} \label{rem:general_Lie_groups}
It is essential for the deformation that the Lie group $T$ is abelian. The above proof for regularity does not extend to non-abelian groups. Furthermore the bundle $P \x_T \LT$ will in general not be trivial and sections therein cannot be integrated over $\Sigma$. Hence one can not simply split the integrated equation $(\mathrm{II\/I})$ from the curvature equation. An extension to general Lie groups is possible if one restricts to trivial principal bundles $P$. This case was considered by Gonzalez and Woodward \cite{GW}.
\end{rem}

\begin{rem}
In step $3$ of the above proof for compactness it is essential that we work with the standard complex structure on $\C^n$. So while this deformation could in principle be considered for any Hamiltonian $T$-space, it is not clear that it gives a homotopy of $T$-moduli problems in general. We return to this question in section \ref{chap:generalizations}.
\end{rem}

\subsection{Computation of vortex invariants}

\label{sec:vortex_computation}

We will now show how the deformation result simplifies the computation of vortex invariants by studying the moduli problem associated to the equations $(*)_{\eps = 0}$. This moduli problem is given by $\mcB$ as in \ref{def:mcB}, the bundle $\mcE$ as in \ref{def:mcE} with fiber
$$
\mcF := W^{k-1,2} \left( \Sigma , \Lambda^{0,1} T^*\Sigma \otimes_\C \mcV \right) \oplus \underline{W^{k-1,2} \left( \Sigma , \LT \right)} \oplus \LT,
$$
and the section
$$
\mcS : \mcB \ra \mcE \; ; \; \left[ u , A \right] \mto 
\left[ u , A , 
\left(
\begin{array}{c}
\bar{\partial}_A u \\
*F_A - \frac{\kappa}{\vol(\Sigma)} \\
\overline{\mu(u)} - \tau
\end{array}
\right)
\right].
$$

\subsubsection*{The Jacobian torus}

\label{sec:vortex_Jacobian_torus}

We begin with the following fact:

\begin{lem} \label{lem:jacobian_torus}
Let $\pi : P \ra \Sigma$ be a principal $T^k$-bundle over the closed Riemann surface $\Sigma$ of genus $g$ (with fixed volume form $\dvol_\Sigma$) with characteristic vector $\kappa \in \Lambda \subset \LT$. Then the constant map $\bar{\kappa} := \kappa / \vol(\Sigma)$ is a regular value of the map
\begin{eqnarray*}
\Phi \quad : \quad \mcA(P) / \mcG_0(P) & \ra & \left\{ F : \Sigma \ra \LT \;|\; \int_\Sigma F \cdot \dvol_\Sigma = \kappa \right\} \\
\left[ A \right] \qquad & \mto & \qquad *F_A.
\end{eqnarray*}
The preimage $\Phi^{-1}(\bar{\kappa})$ is diffeomorphic to the \emph{Jacobian torus}
$$
\T := \left[ \frac{ H^1(\Sigma;\R) }{ H^1(\Sigma;\Z) } \right]^k.
$$
\end{lem}

\begin{proof}
To be precise in the statement of the lemma we have to introduce Sobolev completions in source and target of the map $\Phi$. But then by smoothness of the element $\bar{\kappa}$ and elliptic regularity one shows that all solutions lie in the quotient $\mcA(P) / \mcG_0(P)$ of smooth objects. For simplicity we omit this discussion and deal with smooth objects right away.

Given any $A \in \mcA(P)$ the element $*F_A \in \mcC^\infty(\Sigma,\LT)$ is defined as follows: By the properties of connections the differential $\mathrm{d}A \in \Omega^2(P,\LT)$ descends to a $2$-form $F_A := \pi_*\mathrm{d}A \in \Omega^2(\Sigma,\LT)$. This curvature-form is invariant under the action of the gauge group $\mcG(P)$ on connections. Then $*F_A \cdot \dvol_\Sigma := F_A$. Recall that by Chern-Weil theory the integral $\int_\Sigma F_A$ does not depend on the chosen connection $A$ and this value was our definition for the characteristic vector $\kappa$.

We first consider the case of a principal $S^1$-bundle (i.~e.~we set $k = 1$). We identify $\mathrm{Lie}(S^1) \cong \R$ and write connections just as usual $1$-forms. We start with an arbitrary connection $A$ and note that $F_A$ and $\bar{\kappa} \cdot \dvol_\Sigma$ are cohomologous since their difference evaluates to zero upon pairing with the fundamental class $[\Sigma]$. Hence there is an element $\alpha \in \Omega^1(\Sigma)$ such that
$$
F_A - \bar{\kappa} \cdot \dvol_\Sigma = \mathrm{d}\alpha.
$$ 
Now define $A_0 := A - \pi^*\alpha$. First observe that $A_0$ again is a connection form. Furthermore it satisfies
$$
*F_{A_0} = * \pi_* \mathrm{d} ( A - \pi^*\alpha ) = * ( F_A - \mathrm{d}\alpha ) = \bar{\kappa}.
$$
The same reasoning as above shows that
$$
\left\{ F \in \mcC^\infty(\Sigma) \;|\; \int_\Sigma F \cdot \dvol_\Sigma = \kappa \right\} = \left\{ \bar{\kappa} + *\mathrm{d}\alpha \;|\; \alpha \in \Omega^1(\Sigma) \right\}.
$$
And for every connection $A$ with $*F_A = \bar{\kappa}$ we see that the elements of the affine space $\left\{ A + \pi^*\alpha \;|\; \alpha \in \Omega^1(\Sigma) \right\}$ are connections with curvature $\bar{\kappa} + *\mathrm{d}\alpha$. This shows that indeed $\bar{\kappa}$ is a regular value of $\Phi$.

To express $\Phi^{-1}(\bar{\kappa})$ we first note that any $A$ with $*F_A = \bar{\kappa}$ differs from $A_0$ by $\pi^*\alpha$ for some $\alpha \in \Omega^1(\Sigma)$ with $\mathrm{d}\alpha = 0$. Hence we obtain
$$
\Phi^{-1}(\bar{\kappa}) = \left\{ A_0 + \pi^*\alpha \;|\; \mathrm{d}\alpha = 0 \right\} / \mcG_0(P).
$$
Now consider the normal subgroup $\mcG_{0,\mathrm{contr.}}(P) \lhd \mcG_0(P)$ consisting of all those based gauge transformations $\gamma : \Sigma \ra S^1$ that are homotopic to the constant map $\Sigma \mto 1$. Any such map can be written as $\gamma(z) = e^{if(z)}$ for some smooth map $f : \Sigma \ra \R$. The action of such a transformation on a connection $A$ is given by $\gamma^*A = A + \pi^*\mathrm{d}f$. Hence we obtain
$$
\left\{ A_0 + \pi^*\alpha \;|\; \mathrm{d}\alpha = 0 \right\} / \mcG_{0,\mathrm{contr.}}(P) \cong H^1(\Sigma;\R)
$$
and the quotient $\mcG_0(P) / \mcG_{0,\mathrm{contr.}}(P)$ can be identified with homotopy classes of maps $\Sigma \ra S^1$, which in turn is $H^1(\Sigma,\Z)$. This proves the lemma in the case $k = 1$. The general case now follows by splitting the $k$-torus into a product of $k$ circles and considering each component of the connections in the induced splitting $\LT \cong \oplus_{j=1}^k \mathrm{Lie}(S^1)$ separately.
\end{proof}

This observation gives rise to the following morphism of moduli problems. We define
$$
\mcB_\kappa := \frac{ W^{k,2} \left( \Sigma , \mcV \right) \x \left\{ A \in \mcA^{k,2}(P) \;|\; *F_A = \bar{\kappa} \right\} }{ \mcG_0^{k+1,2}(P) }.
$$
This is a smooth submanifold of $\mcB$. Over $\mcB_\kappa$ we consider the bundle $\mcE_\kappa$ with fiber
$$
\mcF_\kappa := W^{k-1,2} \left( \Sigma , \Lambda^{0,1} T^*\Sigma \otimes_\C \mcV \right) \oplus \LT
$$
and section
$$
\mcS_\kappa : \mcB_\kappa \ra \mcE_\kappa \; ; \; \left[ u , A \right] \mto 
\left[ u , A , 
\left(
\begin{array}{c}
\bar{\partial}_A u \\
\overline{\mu(u)} - \tau
\end{array}
\right)
\right].
$$
The inclusion $\mcF_\kappa \ra \mcF$ given by $( \delta , v ) \mto ( \delta , 0 , v )$ induces a bundle homomorphism $F : \mcE_\kappa \ra \mcE$ covering the inclusion $f : \mcB_\kappa \ra \mcB$. It is now straightforward to show the following fact.

\begin{lem} \label{lem:morphism}
The pair $(f,F)$ is a morphism between the $T$-moduli problems $( \mcB_\kappa , \mcE_\kappa , \mcS_\kappa )$ and $( \mcB , \mcE , \mcS )$ in the sense of definition \ref{def:morphism}.
\end{lem}

\begin{proof}
Only the fourth property of a morphism in definition \ref{def:morphism} needs an explanation. In fact we have not yet defined the orientation on $\det(\mcS_\kappa)$. We refer to section \ref{sec:orientations} for a detailed discussion of orientations.
\end{proof}

\subsubsection*{Moduli spaces}

We digress a little and just look at the moduli space $\mcM_\kappa = \mcS_\kappa^{-1}(0)$ without worrying about the fact that in general it will not be cut out transversally.

\subsubsection*{Genus zero}

In case $\Sigma = S^2$ is a surface of genus zero the Jacobian torus is just one point. We fix a smooth connection $A$ representing this point, i.~e.~with $*F_A = \bar{\kappa}$, and identify
$$
\mcB_\kappa \cong  W^{k,2} \left( \Sigma , \mcV \right).
$$
The two components of the section $\mcS_\kappa$ are now completely uncoupled. The first one picks out the holomorphic sections (with respect to the fixed connection $A$) in the Hermitian vector bundle $\mcV = \oplus_{\nu = 1}^N \mcL_\nu$. This is a finite-dimensional vector space $V$. The $L^2$ inner product on sections of $\mcV$ makes $V$ into a Hermitian vector space.

If we denote the space of holomorphic sections in the line bundle $\mcL_\nu$ by $V_\nu$ we obtain $V = \oplus_{\nu = 1}^N V_\nu$. The dimension of $V_\nu$ is the dimension of the kernel of the linear Cauchy-Riemann operator in the bundle $\mcL_\nu$ of degree $d_\nu = \langle w_\nu , \kappa \rangle$. By the Riemann-Roch theorem (see McDuff and Salamon \cite[Theorem C.1.10]{McS} for the precise statements that we refer to) this is given by
$$
n_\nu := \dim_\C V_\nu =  \max \left( 0 , 1+d_\nu \right).
$$
The remaining $T \cong \mcG / \mcG_0$-action on $\mcB$ carries over to a linear torus action $\rho$ on $V$ that on the component $V_\nu$ is given by the weight vector $w_\nu$. Thus we are precisely in the setting of a Hamiltonian torus action as described in section \ref{sec:torus_actions_on_hermitian_vector_spaces}. Hence the moment map for this action is given by
\begin{align*}
V = \bigoplus_{\nu=1}^N V_\nu & \;\ra\; \quad \LT^* \\
\left( u_\nu \right) \quad & \;\mto\; \pi \sum_{\nu=1}^N \left\| u_\nu \right\|^2_{L^2(\Sigma,\mcL_\nu)} \cdot w_\nu.
\end{align*}
If we rescale the Hermitian form on $V$ by $\vol(\Sigma)^{-\frac{1}{2}}$ and compare this to equation \ref{eqn:overline_mu} we see that the moment map is exactly $\overline{\mu}$. Thus the second component of the section $\mcS_\kappa$ just picks out the $\tau$-level of the moment map. So in the notation of section \ref{sec:torus_actions_on_hermitian_vector_spaces} we obtain $\mcM_\kappa / T \cong X_{V,\tau}$.

\begin{thm} \label{thm:genus_0_moduli_space}
Consider the linear torus action on $\C^N$ that is given by weights $(w_\nu)_{\nu = 1, \ldots N}$. Then the moduli space of the deformed genus zero vortex equations in degree $\kappa$ for a regular parameter $\tau$ is the toric manifold $X_{V,\tau}$ with $V = \oplus_{\nu = 1}^N V_\nu$ as above.

The difference between the original toric manifold $X_{\C^N,\tau}$ and this moduli space lies solely in the dimensions of the vector spaces $V_\nu$. These dimensions are completely determined by the value of $\kappa$.
\end{thm}

This observation makes the computation of the vortex invariants for genus zero surfaces extremely transparent. We will turn to this point in the next section.

\subsubsection*{Higher genus}

If the genus $g$ of $\Sigma$ is greater or equal to one then the dimension of the Jacobian torus $\T$ is positive and the base manifold $\mcB_\kappa$ is a bundle over $\T$ with fiber $W^{k,2} \left( \Sigma , \mcV \right)$ and the remaining torus acts only in the fibers. Now the above discussion applies to the restriction of the $T$-moduli problem $\left( \mcB_\kappa , \mcE_\kappa , \mcS_\kappa \right)$ to any such fiber, but the picture in each fiber may look very different: The dimension of the kernel of the operator $\bar{\del}_A$ depends on the point $[A] \in \T$. But we can circumvent this difficulty if we make restrictions on the degrees $d_\nu$. If we assume that for every $\nu$ we either have
$$
d_\nu > 2g - 2 \quad \mathrm{or} \quad d_\nu < 0
$$
then we get that the spaces $V_\nu$ of holomorphic sections in $\mcL_\nu$ have complex dimension
$$
n_\nu = \max ( 0 , 1 - g + d_\nu )
$$
independent of the connection: If $d_\nu < 0$ we have $n_\nu = 0$ and in fact the Riemann-Roch theorem tells us that the Cauchy-Riemann operator of a complex line bundle of negative degree is injective. Hence $\dim(\ker\,\bar{\del}_A) = 0$ for any connection $A$. If on the other hand $d_\nu - 2 g + 2 > 0$, then the Cauchy-Riemann operator is surjective and the dimension of the kernel agrees with the index ($1-g+d_\nu$). And in fact we have $n_\nu = 1 - g + d_\nu$ in that case, because
$$
d_\nu - 2 g + 2 > 0 \;\Ra\; 1 - g + d_\nu > g - 1 \ge 0.
$$
This proves the following generalization of theorem \ref{thm:genus_0_moduli_space} for general genus $g$.

\begin{thm} \label{thm:genus_g_moduli_space}
With the notation and the assumptions on the degrees $d_\nu$ from above the vortex moduli space $\mcM_\kappa / T$ is a fiber bundle over the Jacobian torus $\T$ with toric fiber $X_{V,\tau}$.
\end{thm}

\subsubsection*{Finite-dimensional reduction}

\label{sec:vortex_finite_dimensional_reduction}

From now on we concentrate on the case $\Sigma = S^2$ of genus zero. The next step is to reduce the moduli problem $\left( \mcB_\kappa , \mcE_\kappa , \mcS_\kappa \right)$ to a finite-dimensional one. The existence of a finite-dimensional reduction is the general tool to define the Euler class for an equivariant moduli problem. In our case this reduction is easy and we can give an explicit description, because of linearity.

As in the above discussion of the genus zero case we fix a smooth connection $A$ with $*F_A = \bar{\kappa}$ and write
$$
\mcB_\kappa \cong  W^{k,2} \left( \Sigma , \mcV \right) = \bigoplus_{\nu = 1}^N W^{k,2} \left( \Sigma , \mcL_\nu \right).
$$
For every $\nu \in \left\{ 1 , \ldots , N \right\}$ we denote by $\bar{\del}_{A,\nu}$ the linear Cauchy-Riemann operator on $W^{k,2} \left( \Sigma , \mcL_\nu \right)$ with respect to the fixed connection $A$. We set
$$
V_\nu := \ker\,\bar{\del}_{A,\nu} \quad \mathrm{and} \quad W_\nu := \coker\,\bar{\del}_{A,\nu}
$$
and identify $W_\nu$ with the $L^2$-orthogonal complement to the image of $\bar{\del}_{A,\nu}$ in $W^{k-1,2} \left( \Sigma , \Lambda^{0,1} T^*\Sigma \otimes_\C \mcL_\nu \right)$. Thus all $V_\nu$ and $W_\nu$ are Hermitian vector spaces with othogonal $T$-actions given by the respective weights $w_\nu$. As explained above it makes sense to rescale the natural $L^2$ inner products on these spaces by $\vol(\Sigma)^{-\frac{1}{2}}$. By the Riemann-Roch theorem we have
$$
n_\nu := \dim_\C( V_\nu ) = \left\{
\begin{array}{c@{\quad \mathrm{if} \quad}l}
0 & d_\nu < 0, \\
1 + d_\nu & d_\nu \ge 0,
\end{array}
\right.
$$
and hence
$$
m_\nu := \dim_\C( W_\nu ) = \left\{
\begin{array}{c@{\quad \mathrm{if} \quad}l}
- 1 - d_\nu & d_\nu < 0, \\
0 & d_\nu \ge 0.
\end{array}
\right.
$$
Recall that the degree $d_\nu$ of the complex line bundle $\mcL_\nu$ is given by $\langle w_\nu , \kappa \rangle$. Hence the dimensions of these spaces depend on $\kappa$. We now consider the following $T$-moduli problem
$$
B_\kappa := \bigoplus_{\nu = 1}^N V_\nu \quad , \quad F_\kappa := \left( \bigoplus_{\nu = 1}^N W_\nu \right) \oplus \LT \quad , \quad E_\kappa := B_\kappa \x F_\kappa
$$
with
$$
S_\kappa : B_\kappa \ra F_\kappa \; ; \; v \mto
\left(
\begin{array}{c}
0 \\
\overline{\mu}(v) - \tau
\end{array}
\right).
$$
This finite-dimensional moduli problem is oriented by the natural orientation of complex vector spaces and a choice of orientation on $\LT$.

\begin{lem} \label{lem:finite_dimensional_reduction}
The inclusions $B_\kappa \subset \mcB_\kappa$ and $F_\kappa \subset \mcF_\kappa$ induce a morphism between the $T$-moduli problems $\left( B_\kappa , E_\kappa , S_\kappa \right)$ and $\left( \mcB_\kappa , \mcE_\kappa , \mcS_\kappa \right)$.
\end{lem}

\begin{proof}
Everything apart from orientations is obvious. We deal with orientations in section \ref{sec:orientations}.
\end{proof}

\subsubsection*{The Euler class}

We are now ready to compute the Euler class $\chi^{B_\kappa,E_\kappa,S_\kappa}$. Recall from section \ref{chap:moduli_problems} that for an oriented finite-dimensional moduli problem the Euler class is defined as follows: Pick a $T$-equivariant Thom structure $(U,\theta)$ on $\left( B_\kappa , E_\kappa , S_\kappa \right)$ and for a $\mathrm{d}_T$-closed equivariant form  $\alpha \in \Omega_T^*(B_\kappa)$ set
$$
\chi^{B_\kappa,E_\kappa,S_\kappa}(\alpha) := \int_{B_\kappa / T} \alpha \wedge {S_\kappa}^*\theta.
$$

\begin{lem}
The Thom form $\theta \in \Omega_T^*(E_\kappa)$ can be chosen such that
$$
{S_\kappa}^*\theta = \left( \prod_{\nu = 1}^N w_\nu^{m_\nu} \right) \wedge {S_\kappa}^* {\pi_F}^* \theta_\LT.
$$
Here $\theta_\LT$ is an ordinary Thom form on $\LT$ and $\pi_F : E_\kappa = B_\kappa \x F_\kappa \ra F_\kappa$ is the projection onto the fiber of $E$.
\end{lem}

\begin{proof}
Since the bundle $E$ is trivial the Thom form $\theta$ can be written as
$$
\theta = {\pi_F}^* ( \theta_1 \wedge \ldots \wedge \theta_N \wedge \theta_\LT ),
$$
where $\theta_\nu$ is an equivariant Thom form on the linear space $W_\nu$. The section $S_\kappa$ is zero in all components $W_\nu$, hence the pull-back of ${\pi_F}^* \theta_\nu$ under $S_\kappa$ is by definition the equivariant Euler class of the trivial bundle $B_\kappa \x W_\nu$. Now the complex dimension of $W_\nu$ is $m_\nu$ and the torus $T$ acts on $W_\nu$ via the weight vector $w_\nu$. Hence by \ref{eqn:torus_euler_class} this Euler class is given by
$$
{S_\kappa}^* {\pi_F}^* \theta_\nu = w_\nu^{m_\nu}.
$$
The claimed result follows.
\end{proof}

\begin{prop}
With $V := B_\kappa =  \oplus_{\nu=1}^N V_\nu$ and the notation from section \ref{sec:toric_manifolds_as_moduli_problems} we have
$$
\chi^{B_\kappa,E_\kappa,S_\kappa}(\alpha) = \chi^{V,\tau} \left( \alpha \cdot \prod_{\nu = 1}^N w_\nu^{m_\nu} \right)
$$
for every $\alpha \in S(\LT^*)$.
\end{prop}

\begin{proof}
This follows immediately from the preceding lemma and the fact that $\chi^{V,\tau}$ is defined to be the Euler class of the $T$-moduli problem associated to the toric manifold that is described by the data $V_\nu$, $w_\nu$ and $\tau$.
\end{proof}

\begin{thm} \label{thm:computation_of_vortex_invariants}
With the notation from above the genus zero vortex invariant for a regular element $\tau$ is given by the formula
$$
\Psi^{\rho,\tau}_\kappa(\alpha) = \chi^{V,\tau} \left( \alpha \cdot \prod_{\nu = 1}^N w_\nu^{m_\nu} \right)
$$
for every $\alpha \in S(\LT^*)$.
\end{thm}

\begin{proof}
This follows immediately from the preceding proposition and the cobordism and the functoriality property of the Euler class. One only has to observe that the homotopy from our deformation and the subsequent morphisms induce the identity map on $S(\LT^*)$ and that the orientation conventions for the original vortex moduli problem and the moduli problem associated to a toric manifold coincide. We deal with the orientations in the following section.
\end{proof}

\subsection{Orientations}

\label{sec:orientations}

There are four issues on orientation that we still have to attend to:
\begin{itemize}
\item We need to define the orientation on the homotopy of moduli problems given by our deformation in section \ref{sec:vortex_deformation}.
\item We need to identify the orientation on the target moduli problem of the morphism in section \ref{sec:vortex_Jacobian_torus}.
\item We need to identify the orientation on the finite-dimensional reduction in section \ref{sec:vortex_finite_dimensional_reduction}.
\item We need to relate the resulting orientation to that of toric manifolds as discussed in section \ref{sec:orientation_of_toric_manifolds}.
\end{itemize}
We start with the first point, which is just a slight generalization of the definition of orientation on the original vortex moduli problem. Given a solution $[u,A] \in \mcB$ of $(*)_\eps$ for some $\eps$ we choose a local trivialization of the bundle $\mcE$ around this point to obtain the vertical differential
$$
\mcD \;:\; T_{[u,A]}\mcB \ra \mcF.
$$
This operator is Fredholm and its determinant $\det(\mcD)$ is the one-dimensional vector space
$$
\det(\mcD) \;=\; \Lambda^{\mathrm{max}} \left( \ker\,\mcD \right) \otimes \Lambda^{\mathrm{max}} \left( \coker\,\mcD \right).
$$
The tangent space $T_{[u,A]}\mcB$ can be identified with the kernel of the map
$$
\varphi \;:\; W^{k,2} \left( \Sigma , \mcV \right) \x W^{k,2} \left( \Sigma , T^*\Sigma \otimes \LT \right) \;\ra\; \widetilde{W}^{k-1,2} \left( \Sigma , \LT \right),
$$
which is the $L^2$-adjoint to the inclusion of the tangent space to the orbit of the based gauge group. Here $\widetilde{W}$ denotes the space of maps that vanish at the basepoint $x_0$. We write an additional summand $\LT$ into the target to get a map $\widetilde{\varphi}$ into $W^{k-1,2} \left( \Sigma , \LT \right)$ that corresponds to a local slice condition for the whole gauge group. Now after fixing an orientation on $\LT$ the determinant of $\mcD$ can be canonically identified with the determinant of the operator
$$
\widetilde{\mcD} \;:\; W^{k,2} \left( \Sigma , \mcV \right) \x W^{k,2} \left( \Sigma , T^*\Sigma \otimes \LT \right) \;\ra\; \mcF \oplus W^{k-1,2} \left( \Sigma , \LT \right)
$$
that is given by $\mcD \oplus \widetilde{\varphi}$.

Now source and target of $\widetilde{\mcD}$ carry complex structures. In the case $\eps = 1$ it is shown by Cieliebak, Gaio, Mundet and Salamon \cite[Chapter 4.5]{CGMS} that $\widetilde{\mcD}$ is a compact perturbation of a certain complex linear operator. But this is also true for any other value of $\eps$: The parameter $\eps$ only appears as a factor in front of some parts of this compact perturbation. Hence the natural orientation of the determinant for a complex linear Fredholm operator induces an orientation for $\det(\mcD)$. This shows that during the whole homotopy we can orient all determinant line bundles by connecting the corresponding operators to one fixed complex linear one in the space of Fredholm operators. This defines an orientation on our homotopy.

Let us be explicit about the complex structures that are used to define this orientation. On $W^{k,2} \left( \Sigma , \mcV \right)$ we get it from the fixed Hermitian structure on $\mcV$. On $W^{k,2} \left( \Sigma , T^*\Sigma \otimes \LT \right)$ we take the Hodge $*$-operator that is given by the fixed metric on $\Sigma$. On
$$
\mcF \oplus W^{k-1,2} \left( \Sigma , \LT \right) \;=\;  W^{k-1,2} \left( \Sigma , \Lambda^{0,1} T^*\Sigma \otimes \mcV \right) \oplus W^{k-1,2} \left( \Sigma , \LT \right) \oplus W^{k-1,2} \left( \Sigma , \LT \right)
$$
the complex structure is determined by fixing an order of the second and third component. Here we make the same choice as for the complex structure on $\LT \oplus \LT$ that we made in section \ref{sec:orientation_of_toric_manifolds} in order to define the orientation of toric manifolds.

Next we consider the morphism $(f,F)$ from lemma \ref{lem:morphism}. Recall that the source moduli problem $( \mcB , \mcE , \mcS )$ features the curvature equation $*F_A = \bar{\kappa}$ as part of the section $\mcS$, while the target moduli problem $( \mcB_\kappa , \mcE_\kappa , \mcS_\kappa )$ has this equation as part of the definition of $\mcB_\kappa$. The orientation of $\det(\mcS)$ is given as above and we can define the orientation on $\det(\mcS_\kappa)$ in just the same way: We split
$$
W^{k-1,2} \left( \Sigma , \LT \right) \;\cong\; \LT \oplus \underline{ W^{k-1,2} \left( \Sigma , \LT \right) }
$$
as in the introduction to section \ref{sec:vortex_deformation} and we can identify the kernel and co\-ker\-nel of the vertical differential $\mcD_\kappa$ with the kernel and cokernel of the same operator $\widetilde{\mcD}$ that we use to orient $\det(\mcS)$. Only the component $\underline{ W^{k-1,2} \left( \Sigma , \LT \right) }$ of the summand that belongs to $\mcF$ is no longer interpreted as part of the fiber in the moduli problem, but as defining equation for the tangent space to $\mcB_\kappa$.

With this orientation on $( \mcB_\kappa , \mcE_\kappa , \mcS_\kappa )$ it is clear that $(f,F)$ is orientation preserving. Note that in the above splitting of $W^{k-1,2} \left( \Sigma , \LT \right)$ we identify the subset of constant maps with $\LT$, which agrees with the identification of $\LT$ in the splitting
$$
W^{k-1,2} \left( \Sigma , \LT \right) \;\cong\; \LT \oplus \widetilde{W}^{k-1,2} \left( \Sigma , \LT \right)
$$
that we used to extend the map $\varphi$ to $\widetilde{\varphi}$.

Now in the process of finite-dimensional reduction in section \ref{sec:vortex_finite_dimensional_reduction} we first throw away all the data associated to the connections $A$ by fixing one. By the above remark we can describe the induced orientation on the remaining moduli problem by the complex structures on source and target of a map
$$
W^{k,2} \left( \Sigma , \mcV \right) \;\ra\; W^{k-1,2} \left( \Sigma , \Lambda^{0,1} T^*\Sigma \otimes_\C \mcV \right) \oplus \LT \oplus \LT,
$$
that are given as before. The inclusions into source and target of the finite-dimensional spaces $\bigoplus_{\nu = 1}^N V_\nu$ and $\bigoplus_{\nu = 1}^N W_\nu$ that describe the finite-dimensional reduction of the problem are complex linear. So if we define the orientation on $\left( B_\kappa , E_\kappa , S_\kappa \right)$ in lemma \ref{lem:finite_dimensional_reduction} by the complex orientations of
$$
\bigoplus_{\nu = 1}^N V_\nu \;\ra\; \bigoplus_{\nu = 1}^N W_\nu \oplus \LT \oplus \LT,
$$
then we indeed obtain a morphism.

Now finally we observe that this gives precisely the complex orientation on $X_{V,\tau}$ that we explained in section \ref{sec:orientation_of_toric_manifolds}.

\section{Generalizations}

\label{chap:generalizations}

In this section we address the question to what extend our deformation result from section \ref{chap:vortex_invariants} can be generalized. We consider general Hamiltonian $T$-spaces and we examine the notion of energy for solutions of our deformed equations. For the generalization to arbitrary Lie groups $G$ see remark \ref{rem:general_Lie_groups}.

\subsection{General Hamiltonian $T$-spaces}

We want to extend the deformation of the vortex equations from section \ref{sec:vortex_deformation} to more general Hamiltonian $T$-spaces $X$ than $\C^N$. As remarked above the given proof for compactness fails if the almost complex structure on $X$ is no longer constant. In general one will encounter the phenomenon of bubbling of holomorphic spheres. One could at least hope to retain the deformation result for manifolds $X$ that are symplectically aspherical, but this forces $X$ to be non-compact because of the following proposition.

\begin{prop}
Suppose $X$ is a closed symplectic manifold that carries a nontrivial Hamiltonian $S^1$-action. Then there exists an element $\alpha \in \pi_2(X)$ such that 
$$
\omega(\bar{\alpha}) \;>\; 0,
$$
where $\bar{\alpha}$ is the image of $\alpha$ under the map $\pi_2(X) \ra H_2(X;\Z)$.
\end{prop}

\begin{proof}
Consider a moment map $\mu : X \ra \R$ for the Hamiltonian $S^1$-action. Since $X$ is compact it has at least two critical points. Fix a compatible almost complex structure $J$ on $X$ and consider the gradient flow of $\mu$ with respect to the associated metric. Take a nontrivial flowline $\gamma$ that converges to one critical point $p$. Again due to compactness of $X$ the other end of $\gamma$ converges to another critical point $q$. Hence the closure of the $S^1$-orbit of $\gamma$ sweeps out the image of a continuous map $S^2 \ra X$. This map represents a class $\alpha$ with the desired property. See Ono \cite{Ono} for details and more general results.
\end{proof}

Now on a non-compact manifold $X$ to start the bubbling analysis we first need that the images of the maps $u$ stay in a compact subset. Else one could have a sequence of maps with diverging derivative on a sequence of points going off to infinity. But even in the linear case of $X \cong \C^N$ we do not get such an a priori bound: The analysis in step $1$ for the discussion of compactness in section \ref{sec:vortex_deformation} only gives a uniform bound on $\eps \cdot |u|^2$. We get the necessary $\mcC^0$-bound on the maps $u$ a posteriori from the $W^{2,2}$-bound that is established in step $3$ and which makes use of the fact that the complex structure on $\C^N$ is standard, so that we can circumvent the bubbling analysis completely.

\subsection{Energy of $\eps$-vortices}

Given a configuration $(u,A)$ for the usual vortex equations $(*)_{\eps = 1}$ the \emph{energy} is defined as
\begin{equation} \label{def:energy}
E(u,A) \;:=\; \frac{1}{2} \int_\Sigma \left( |\mathrm{d}_A u|^2 + |*F_A|^2 + |\mu(u) - \tau - \bar{\kappa}|^2  \right) \dvol_\Sigma
\end{equation}
with $\bar{\kappa} := \frac{\kappa}{\vol(\Sigma)}$. The norms of the second and third term of the above integrand are defined via the chosen inner product on $\LT$ that is also used to identify $\LT \cong \LT^*$. The norm of the first term is given by the metric associated to the symplectic form $\omega$ and an invariant compatible almost complex structure $J$ on the target manifold $X$. The following identity holds due to Cieliebak, Gaio, Mundet and Salamon \cite[Proposition 2.2]{CGMS}:
\begin{eqnarray} \label{eqn:energy_identity}
E(u,A) & = & \int_\Sigma \left( |\bar{\partial}_{J,A} u|^2 + \frac{1}{2} |*F_A + \mu(u) - \tau - \bar{\kappa}|^2  \right) \dvol_\Sigma \nonumber\\
& & + \left\langle \left[ \omega - \mu + \tau + \bar{\kappa} \right] , [u] \right\rangle
\end{eqnarray}
Here the last term is the pairing of equivariant cohomology with homology: The difference $\omega - \mu$ is a $\mathrm{d}_T$-closed equivariant differential form, as are the elements $\tau, \bar{\kappa} \in \LT^*$. And the equivariant map $u$ from a principal $T$-bundle $P$ into the manifold $X$ represents an equivariant homology class.

Now as remarked in \ref{rem:epsilon_omega} the deformed equations $(*)_\eps$ can be interpreted as the usual vortex equations with respect to the rescaled symplectic form $\eps \cdot \omega$. This suggests to define the $\eps$-energy $E_\eps(u,A)$ by introducing a factor $\eps$ to the norm $|\mathrm{d}_A u|^2$, the moment map $\mu$ and the parameter $\tau$ in \ref{def:energy}. Since $\kappa$ is not affected by this rescaling we include it into the curvature term and define
\begin{equation} \label{def:epsilon_energy}
E_\eps(u,A) \;:=\; \frac{1}{2} \int_\Sigma \left( \eps |\mathrm{d}_A u|^2 + |*F_A - \bar{\kappa}|^2 + \eps^2 |\mu(u) - \tau|^2  \right) \dvol_\Sigma.
\end{equation}
The analogous computations as in \cite[Section 2.3]{CGMS} then yield the following energy identity:
\begin{eqnarray} \label{eqn:epsilon_energy_identity}
E_\eps(u,A) & = & \int_\Sigma \left( \eps |\bar{\partial}_{J,A} u|^2 + \frac{1}{2} |*F_A + \eps \mu(u) - \eps \tau - \bar{\kappa}|^2  \right) \dvol_\Sigma \nonumber\\
& & + \eps \left\langle \left[ \omega - \mu + \tau \right] , [u] \right\rangle \nonumber\\
& & - \eps  \int_\Sigma \left\langle \mu(u) - \tau , \bar{\kappa} \right\rangle \dvol_\Sigma
\end{eqnarray}
Hence for a solution $(\eps,u,A)$ to $(*)_\eps$ we obtain
$$
E_\eps(u,A) \;=\; \eps \left\langle \left[ \omega - \mu + \tau \right] , [u] \right\rangle.
$$
Note that the third term in \ref{eqn:epsilon_energy_identity} vanishes because of equation $(\mathrm{II\/I})$. So if we restrict to solutions with $[u] \in H_T^2(X)$ representing a fixed class then this identity gives uniform $L^2$-bounds on
$$
|\mathrm{d}_A u| \quad , \quad \frac{1}{\sqrt{\eps}} \cdot |*F_A - \bar{\kappa}| \quad \mathrm{and} \quad \sqrt{\eps} \cdot |\mu(u) - \tau|.
$$
But as for the usual vortex equations this is a Sobolev borderline case and these bounds do not suffice to get compactness. The bounds would be good enough to carry out the bubbling analysis, but without a $\mcC^0$-bound on the maps $u$ we cannot even start it.

\section{Givental's toric map spaces}

\label{chap:giventals_toric_map_spaces}

Let $X := X_{\C^N,\tau}$ be a toric manifold given by a proper collection of weights $w_\nu$ and a super-regular element $\tau$. Recall the description of the associated genus zero vortex moduli space for degree $\kappa \in \Lambda$ as the toric manifold
$$
X_\kappa \;:=\; X_{V,\tau}
$$
from theorem \ref{thm:genus_0_moduli_space}: It is given by the same weights $w_\nu$  as $X$, only the spaces $V_\nu$ on which the torus $T$ acts via those weights are changed from $\C$ to the spaces of holomorphic sections in the complex line bundles $\mcL_\nu$ of degree $d_\nu = \langle w_\nu , \kappa \rangle$. The complex dimension of $V_\nu$ is $n_\nu = \max( 0 , 1 + d_\nu )$ and it is equipped with the natural Hermitian form
$$
(u,v) \;:=\; \int_{S^2} u \cdot \bar{v} \;\dvol_{S^2}.
$$
The same space $X_\kappa$ appears in Givental's work \cite{Giv} with the name \emph{toric map space}. More precisely Givental replaces every component of $\C^N$ by the $n_\nu$-dimensional space of polynomials $\gamma_\nu$ in one complex variable $\zeta$ of degree $\deg(\gamma_\nu) \le d_\nu$,
$$
\gamma_\nu ( \zeta ) \;=\; \sum_{j = 0}^{d_\nu} z_{\nu,j} \zeta^j.
$$
Both descriptions yield the same manifold, because the diffeomorphism type of $X_{V,\tau}$ is uniquely determined by the chamber of $\tau$ and the dimensions of the components $V_\nu$: Any two Hermitian vector spaces of the same dimension are isomorphic and since the $T$-actions are given by a fixed character $T \ra S^1$ and complex multiplication, this isomorphism is also $T$-equivariant. In fact we can give the following explicit, though not canonical identification: If we write
$$
S^2 \;\cong\; \C \cup \infty,
$$
remove the point at infinity and trivialize the bundles $\mcL_\nu$ over $\C$, then we can identify holomorphic sections $u_\nu$ with polynomials of degree at most $d_\nu$. Those are the holomorphic maps $\C \ra \C$ that extend to a section in $\mcL_\nu$ over the point at infinity.

The motivation for Givental's toric map space is the following. If $\tau$ is an element of the K\"ahler cone
$$
K \;:=\; \bigcap_{\nu = 1}^N W_{ \{ 1 , \ldots , N \} \setminus \{ \nu \} } \subset \LT^*
$$
then none of the weights $w_\nu$ vanishes in cohomology and we can identify $H^2(X,\R) \cong \LT^*$ such that $H_2(X,\Z) \cong \Lambda$. Hence $\kappa$ also specifies an integral second homology class of $X$. Now $X$ is canonically identified with the quotient of some open and dense subset $U \subset \C^N$ by an action of the complexified torus $T_\C$. This description uses the notion of \emph{fans} rather than weight vectors. See Audin \cite{Aud} for the correspondence of the two constructions. Hence for generic value of the complex variable $\zeta$ the collection
$$
\left( \gamma_1(\zeta) , \ldots , \gamma_N(\zeta) \right) \in \C^N
$$
will actually lie in $U$ and will thus represent a point in $X$. One can use this to show that generic elements $\left[ \gamma_1 , \ldots , \gamma_N \right] \in X_\kappa$ actually represent holomorphic maps of degree $\kappa$ from $\C \cup \infty$ into $X$. Hence $X_\kappa$ can be viewed as a compactification of the space of such maps.

The reason to consider these toric compactifications is that all the $X_\kappa$ for $\kappa$ ranging over $H_2(X,\Z)$ can be seen in one single infinite-dimensional toric manifold $\mathbbm{X}$, that is given by the same weights $w_\nu$ and the same element $\tau$ but for the $V_\nu$ we take infinite-dimensional spaces of Laurent polynomials $\gamma$ in one complex variable $\zeta$,
$$
V_\nu \;:=\; \C \left[ \zeta , \zeta^{-1} \right].
$$
We choose the Hermitian metric such that the monomials $\zeta^j$ form an orthonormal system. So if we write an element $\gamma_\nu \in V_\nu$ as
$$
\gamma_\nu ( \zeta ) \;=\; \sum_{j \in \Z} z_{\nu,j} \zeta^j
$$
with only finitely many nonzero coefficients $z_{\nu,j}$, then the moment map is given by
$$
\begin{array}{cccc}
\mu : & \D \bigoplus_{\nu=1}^N V_\nu & \ra & \LT^* \\
& \left( \gamma_1 , \ldots , \gamma_N \right) & \mto & \D \pi \sum_{\nu=1}^N \sum_{j\in\Z} \left| z_{\nu,j} \right|^2 \cdot w_\nu.
\end{array}
$$
There is the natural inclusion of $X_\kappa$ into $\mathbbm{X}$ given by
$$
X_\kappa \;\equiv\; X^\kappa_0 \;:=\; \left\{ \left[ \gamma_1 , \ldots , \gamma_N \right] \in \mathbbm{X} \;\left|\; \gamma_\nu(\zeta) = \sum_{j = 0}^{d_\nu} z_{\nu,j} \zeta^j \right. \right\} \subset \mathbbm{X}.
$$
But now we can shift these submanifolds to different powers of $\zeta$. For elements $\kappa_0 , \kappa_1 \in \Lambda$ we define
$$
X^{\kappa_1}_{\kappa_0} \;:=\; \left\{ \left[ \gamma_1 , \ldots , \gamma_N \right] \in \mathbbm{X} \;\left|\; \gamma_\nu(\zeta) = \sum_{j = \langle w_\nu , \kappa_0 \rangle}^{\langle w_\nu , \kappa_1 \rangle} z_{\nu,j} \zeta^j \right. \right\}
$$
and observe that
$$
X_\kappa \;\cong\; X^{\kappa_1}_{\kappa_0} \quad \mathrm{for} \quad \kappa = \kappa_1 - \kappa_0.
$$
In fact all of $\mathbbm{X}$ is built from these compact and finite-dimensional toric manifolds $X^{\kappa_1}_{\kappa_0}$, because we have
$$
\mathbbm{X} \;=\; \bigcup_{\kappa_0,\kappa_1 \in \Lambda} X^{\kappa_1}_{\kappa_0}.
$$
To see this it suffices to have an element $\eta \in \Lambda$ such that $\left\langle w_\nu , \eta \right\rangle > 0$ for all $\nu$. Then the spaces $X^{n\eta}_{-n\eta}$ exhaust all of $\mathbbm{X}$ for $n \in \N$. But the existence of such an element $\eta$ follows by the assumption that the collection of weight vectors $w_\nu$ is proper.

Now there is an additional $S^1$ action on $\C[\zeta,\zeta^{-1}]$ by complex multiplication on the variable $\zeta$. The above formula for the moment map shows that this action on Laurent polynomials indeed descends to an action on $\mathbbm{X}$. By Iritani \cite{Iri} it is shown that equivariant Morse-Bott theory with respect to this action on $\mathbbm{X}$ gives rise to an \emph{abstract $\mcD$-module structure} that can be identified with the \emph{quantum $\mcD$-module} of the underlying toric manifold $X$. Philosophically the quantum $\mcD$-module is an object associated to the mirror of $X$. So this result can be interpreted as a mirror theorem without explicitly knowing the mirror.

If we return to our point of view and consider $X_\kappa$ as a vortex moduli space, then it is slightly mysterious what the corresponding construction should be. The $S^1$ action on $\zeta \in \C$ clearly corresponds to a rotation of $S^2$. But this action does not naturally lift to sections in the bundles $\mcL_\nu$ over $S^2$. The point is that our identification of $X_\kappa$ with the corresponding toric map space in not canonical: We have to remove a point from $S^2$ and choose a certain trivialization to write sections $u_\nu$ as polynomials in $\zeta$. And if we want to obtain the identical moment map as the one used by Givental then the natural Hermitian metric by integration over $S^2$ does not work: There is no volume form on $\C$ such that the monomials $\zeta^j$ are $L^2$-orthonormal. Restriction to and integration over $S^1 \subset \C$ would give such a product. So if we also remove the origin and pick a suitably scaled metric $f(r) \mathrm{d}r \wedge \mathrm{d}\theta$ on $\C^*$ we could match the moment maps. This leads to the intuition that actually vortices on the cylinder are the correct object to study in this context. We refer to the results by Frauenfelder \cite{Frau} in this direction.


\begin{thebibliography}{AAAA}

\bibitem{Aud}
M.\,Audin, {\em Torus Actions on Symplectic Manifolds}, Progress in Mathematics {\bf 93}, Birkh\"auser (2004).

\bibitem{AB}
M.\,F.\,Atiyah, R.\,Bott, {\em The moment map and equivariant cohomo\-logy}, Topology {\bf 23} (1984), no.~1, 1--28.

\bibitem{BT}
R.\,Bott, L.\,W.\,Tu, {\em Differential Forms in Algebraic Topology}, Gra\-du\-ate Texts in Mathematics {\bf 82}, Springer (1982).

\bibitem{CGMS}
K.\,Cieliebak, A.\,Gaio, I.\,Mundet, D.\,Salamon, {\em The symplectic vortex equations and invariants of Hamiltonian group actions}, J.~Symp.~Geom.~{\bf 1} (2002), no.~3, 543--645.

\bibitem{CGS}
K.\,Cieliebak, A.\,Gaio, D.\,Salamon, {\em $J$-holomorphic curves, moment maps, and invariants of Hamiltonian group actions}, Int.~Math.~Res.~Not.~{\bf 2000:16}, 831--882.

\bibitem{CMS}
K.\,Cieliebak, I.\,Mundet, D.\,Salamon, {\em Equivariant moduli problems, branched manifolds, and the Euler class}, Topology {\bf 42} (2003), no.~3, 641--700.

\bibitem{CS}
K.\,Cieliebak, D.\,Salamon, {\em Wall crossing for symplectic vortices and quantum cohomology}, Math.~Ann.~{\bf 335} (2006), no.~1, 133--192.

\bibitem{Dan}
V.\,I.\,Danilov, {\em The geometry of toric varieties}, Russian Math.~Surveys {\bf 33} (1978), 97--154.

\bibitem{DK}
S.\,K.\,Donaldson, P.\,B.\,Kronheimer, {\em The geometry of four-manifolds}, Oxford University Press (1990).

\bibitem{Frau}
U.\,Frauenfelder, {\em Vortices on the cylinder}, Int.~Math.~Res.~Not. {\bf 2006}, Art.~ID 63130, 34 pp.

\bibitem{Giv}
A.\,Givental, {\em A mirror theorem for toric complete intersections}, Topological Field Theory, Primitive Forms, and Related Topics, Kyoto, 1996, 141--175, Progress in Mathematics {\bf 160}, Birkh\"auser (1998).

\bibitem{GW}
E.\,Gonzalez, C.\,Woodward, {\em Area dependence in gauges Gromov-Witten theory}, arXiv:0811.3358v1 [math SG].

\bibitem{GGK}
V.\,Guillemin, V.\,Ginzburg, Y.\,Karshon, {\em Moment Maps, Cobordisms, and Hamiltonian Group Actions}, Math.~Surveys and Monographs {\bf 98}, AMS (2002).

\bibitem{GK}
V.\,Guillemin, J.\,Kalkman, {\em The Jeffrey-Kirwan localization theorem and residue operations in equivariant cohomology}, J.~reine angew.~Math.~{\bf 470} (1996), 123--142.

\bibitem{GS}
V.\,Guillemin, S.\,Sternberg, {\em Supersymmetry and Equivariant de Rham Theory}, Mathematics Past and Present, Springer (1999).

\bibitem{Iri}
H.\,Iritani, {\em Quantum D-modules and equivariant Floer theory for free loop spaces}, Math.~Zeit.~{\bf 252} (2006), no.~3, 577--622.

\bibitem{JK}
L.\,Jeffrey, F.\,Kirwan {\em Localization for nonabelian group actions}, Topology {\bf 34} (1995), no.~2, 291--327.

\bibitem{Ki}
F.\,Kirwan, {\em Cohomology of Quotients in Symplectic and Algebraic Geometry}, Princeton University Press (1984).

\bibitem{Mar}
S.\,Martin, {\em Transversality theory, cobordisms, and invariants of symplectic quotients}, to appear in Annals of Mathematics.

\bibitem{McS}
D.\,McDuff, D.\,Salamon, {\em $J$-holomorphic Curves and Symplectic Topology}, Colloquium Publications {\bf 52}, AMS (2004).

\bibitem{Ono}
K.\,Ono, {\em Obstruction to Circle Group Action Preserving Symplectic Structure}, Hakkaido Mathematical Journal {\bf 21} (1992), 99--102.

\bibitem{tomD}
T.\,tom Dieck, {\em Transformation Groups}, De Gruyter studies in Mathemactics {\bf 8}, De Gruyter (1987).

\bibitem{JW}
J.\,Wehrheim, {\em Vortex Invariants and Toric Manifolds}, Di\-gi\-ta\-le Hochschulschriften LMU M\"unchen, http://edoc.ub.uni-muenchen.de/9144/, Thesis (2008).

\bibitem{KW}
K.\,Wehrheim, {\em Uhlenbeck Compactness}, EMS Series of Lectures in Mathematics, EMS (2004).

\bibitem{Uhl}
K.\,K.\,Uhlenbeck, {\em Connections with $L^p$-bounds on curvature}, Comm.~Math.~Phys.~{\bf 83} (1982), 31--42.

\end{thebibliography}
\end{document}